\documentclass{article}

\usepackage[english, french]{babel}
\usepackage{csquotes}
\usepackage[T1]{fontenc}

\usepackage{amsmath}
\usepackage{amsfonts}
\usepackage{amssymb}
\usepackage{dsfont} 
\numberwithin{equation}{section} 

\usepackage{amsthm} 

\usepackage{hyperref}
\hypersetup{
    colorlinks,
    linkcolor={blue},
    citecolor={blue},
    urlcolor={black}}

\usepackage{graphicx}
\usepackage{caption}

\usepackage{geometry}
\renewcommand\labelenumi{(\roman{enumi})}
\renewcommand\theenumi\labelenumi

\newtheorem{theorem}{Theorem}[section]
\newtheorem{proposition}[theorem]{Proposition}
\newtheorem{lemma}[theorem]{Lemma}
\newtheorem{corollary}[theorem]{Corollary}
\newtheorem{remark}[theorem]{Remark}

\theoremstyle{definition}

\newcommand{\C}{\mathbb{C}}

\newcommand{\E}{\mathbb{E}}

\newcommand{\N}{\mathbb{N}}

\renewcommand{\P}{\mathbb{P}}

\newcommand{\R}{\mathbb{R}}

\newcommand{\cC}{\mathcal{C}}
\newcommand{\cD}{\mathcal{D}}
\newcommand{\cE}{\mathcal{E}}
\newcommand{\cF}{\mathcal{F}}
\newcommand{\cG}{\mathcal{G}}

\newcommand{\cM}{\mathcal{M}}
\newcommand{\cN}{\mathcal{N}}

\newcommand{\cP}{\mathcal{P}}
\newcommand{\cQ}{\mathcal{Q}}

\newcommand{\cT}{\mathcal{T}}
\newcommand{\cU}{\mathcal{U}}

\renewcommand{\tilde}{\widetilde}
\renewcommand{\hat}{\widehat}
\newcommand{\ie}{i.e.\@ }
\newcommand{\eg}{e.g.\@ }
\renewcommand{\d}[1]{\, \mathrm{d}#1}
\newcommand{\Pd}[2]{#1(\mathrm{d}#2)}
\newcommand{\e}{\mathrm{e}}
\newcommand{\iu}{\mathrm{i}}
\newcommand{\Prob}[1]{\mathbb{P}\left(#1\right)}

\newcommand{\Expec}[1]{\mathbb{E}\left[#1\right]}
\newcommand{\condExpec}[2]{\mathbb{E}\left[#1 \middle| #2\right]}

\newcommand{\condLaw}[2]{\mathcal{L}\left(#1 \middle| #2\right)}

\DeclareMathOperator{\Log}{Log}
\DeclareMathOperator{\Supp}{Supp}

\let\originalleft\left
\let\originalright\right
\renewcommand{\left}{\mathopen{}\mathclose\bgroup\originalleft}
\renewcommand{\right}{\aftergroup\egroup\originalright}

\begin{document}

\selectlanguage{english}

\begin{titlepage}
   \begin{center}
        \vspace*{-1cm}
        
        \includegraphics[scale=0.58]{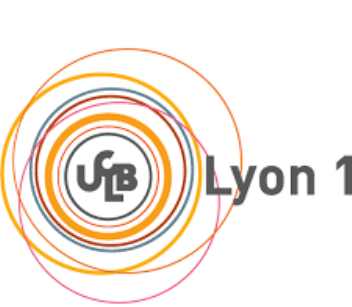} \hspace{1cm}
        \includegraphics[scale=0.19]{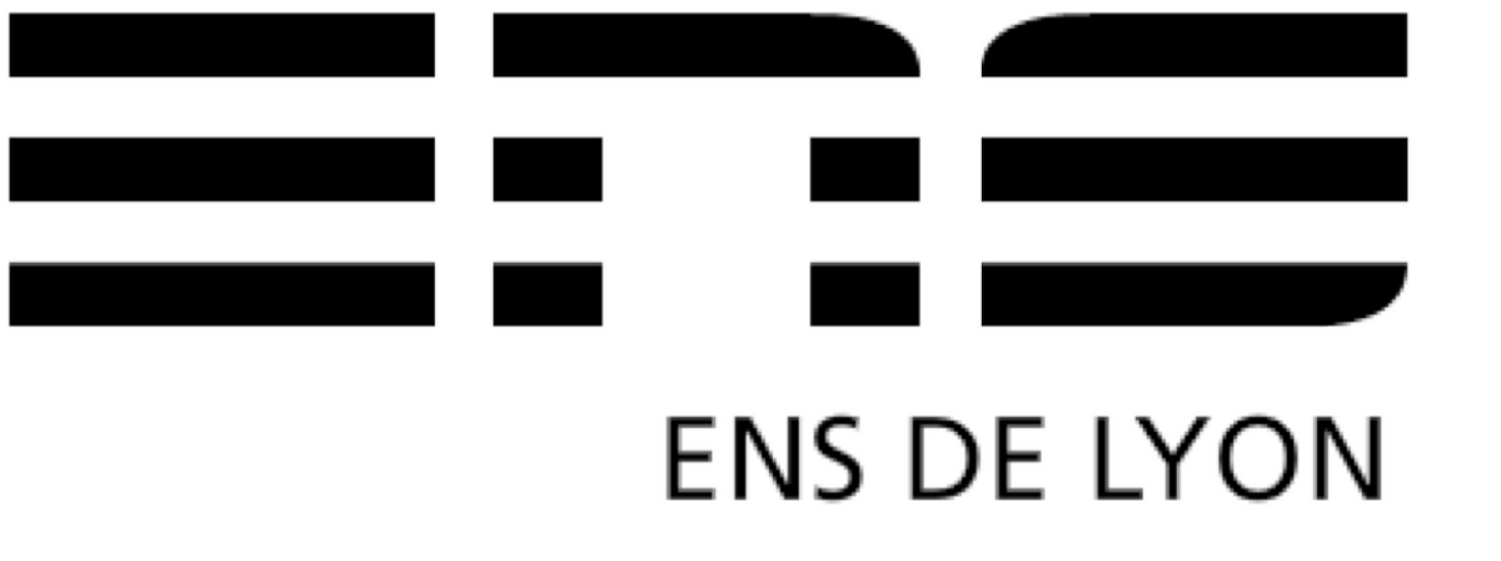} \hspace{1cm}
        \includegraphics[scale=0.16]{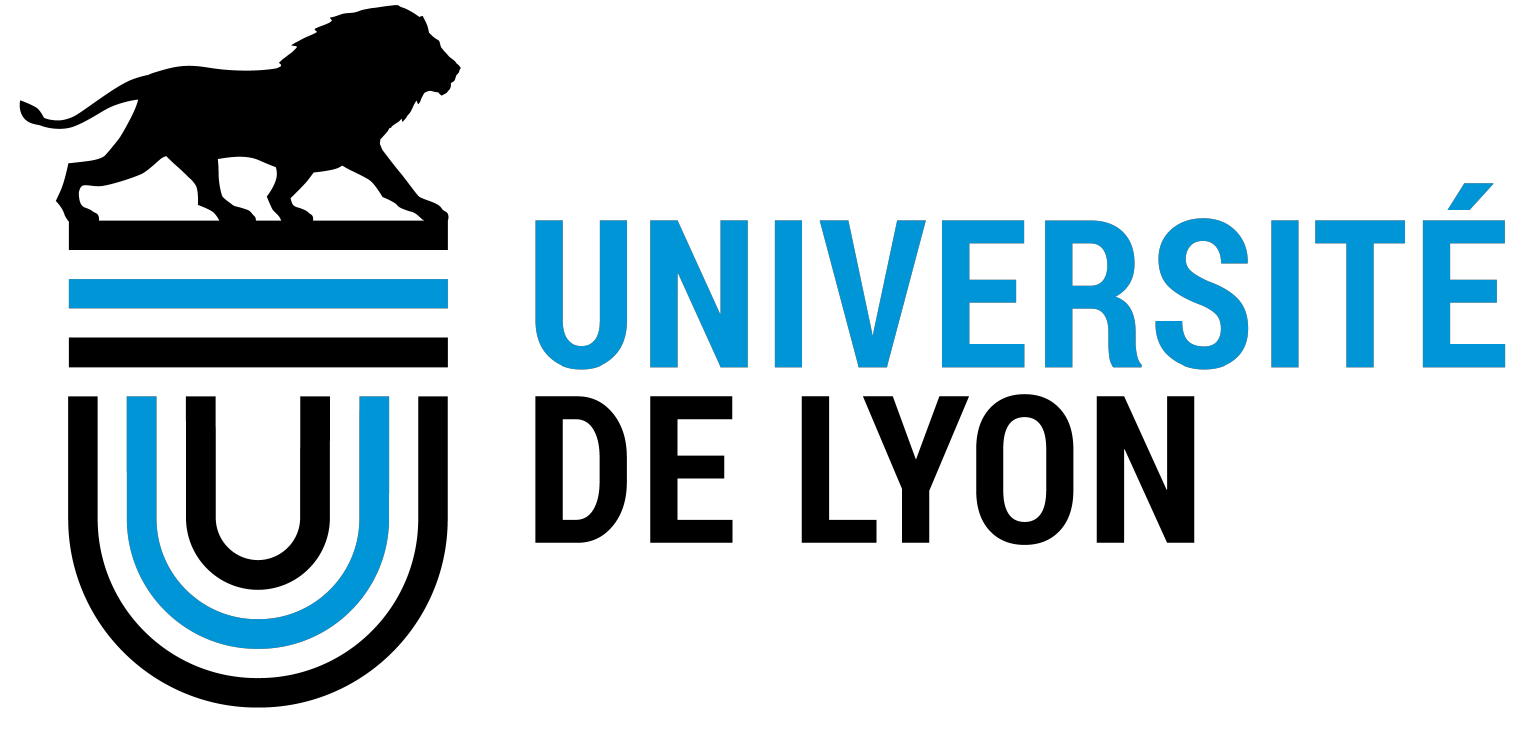}
        
        \vspace{1cm}
        
        \rule[1ex]{\linewidth}{1pt}
        
        \LARGE
        \textbf{Additive martingales of the branching Brownian motion}
        
        \rule[1ex]{\linewidth}{1pt}
        
        \vspace{0.6cm}
        
        \large
        Louis \textsc{Chataignier}
        
        \vspace{0.1cm}
        
        \normalsize
        supervised by Michel \textsc{Pain} and Pascal \textsc{Maillard}

        \vspace{0.5cm}

        \today
        
        \vspace{0.5cm}
            
        \begin{abstract}
        In this thesis, we study asymptotic properties of the standard branching Brownian motion, with a specific emphasis on the additive martingales at high temperature.
        We start by presenting classic and fundamental tools for our investigation.
        Subsequently, we establish various convergence results that enhance our understanding of the model.
        In particular, these results include the determination of particles contributing to the additive martingales, the description of the fluctuations of these martingales around their limits, and an approximation of the so-called overlap distribution.
        Regarding the latter, we believe this is the first time that such an approximation is given.
        Remarkably, we identify a specific regime in which stable distributions emerge.
        \end{abstract}
        
        \vspace{0.1cm}
        
        \selectlanguage{french}
        
        \begin{abstract}
        Dans ce mémoire, nous étudions les propriétés asymptotiques du mouvement brownien branchant standard, en mettant particulièrement l'accent sur les martingales additives à haute température.
        Nous commençons par présenter des outils classiques et fondamentaux pour notre analyse
        Ensuite, nous établissons divers résultats de convergence qui apportent une meilleure compréhension du modèle.
        Cela comprend notamment la détermination des particules contribuant aux martingales additives, la description des fluctuations de ces martingales autour de leurs limites et une approximation de l'\emph{overlap distribution}.
        Concernant ce dernier résultat, nous pensons qu'une telle approximation est inédite.
        De façon notable, nous identifions un régime spécifique dans lequel des lois stables apparaissent.
        \end{abstract}
        
         \selectlanguage{english}
        
        \vfill
        
        \includegraphics[scale=1.1]{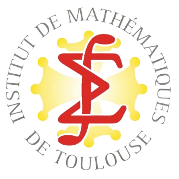}
   \end{center}
\end{titlepage}
\setcounter{page}{2}
I am deeply grateful to my supervisors, Michel Pain and Pascal Maillard, for the subject they proposed and for supporting me throughout my internship.
Their expertise, dedication, and insightful feedback have been invaluable in shaping this thesis.
I also want to thank my friends Kelly Doeuvre, Julia García Cristóbal, Mathilde Girolet, Colin Scarato, my parents Anouk Chataignier and Laurent Chataignier, and my partner Léa Champeau, for their support and proofreading.

\tableofcontents

\section{Introduction}

Branching processes are probabilistic models describing the evolution of populations or systems over discrete generations.
They find applications in various fields such as biology, economics, epidemiology and statistical physics.
One of the earliest instances arose from the investigation of the extinction of family names.
This was carried out independently by Bienaymé \cite{Bienaymé1845} in 1845 and by Galton-Watson \cite{GaltonWatson1875} in 1875.
In their model, generation~$0$ is made of a single individual, the \emph{ancestor}.
This one gives birth to a random number of children and, for the sake of simplicity, dies at the same time.
Subsequently, the individuals of each generation die while giving birth to their own children, following the same distribution as the ancestor's offspring, independently of each other and of the past.
The resulting sequence of individual counts per generation forms a Markov chain, known as the \emph{Galton-Watson process}.
Unsurprisingly, one can show that if the mean offspring count is less than $1$, the process becomes extinct beyond a certain time.
We can also consider the tree whose vertices are the individuals and the edges are the filial relationships, called \emph{Galton-Watson tree}.
It satisfies the \emph{branching property}: trees originating from individuals of a given generation are independent and have the same distribution as the initial tree.

In this thesis, we are interested in a slightly more developed branching process, called branching Brownian motion (BBM).
From a biological point of view, it is a generic model for describing growth and dispersal of a population.
It was introduced by Adke-Moyal \cite{AdkeMoyal1963} in 1963 and has been much studied since several mathematicians \cite{ItoMcKean1965, IkedaNagasawaWatanabe1968, McKean1975} made a deep connection with a classic partial differential equation, namely the \emph{F-KPP equation} (see Section~\ref{sct:speed_of_the_extremal_displacement_1} for further insights).

The one-dimensional branching Brownian motion can be depicted as follows.
At time $0$, a single particle starts a Brownian motion in $\R$ from the origin.
After a random time with exponential distribution, it splits into a random number of particles or, equivalently, it dies while giving birth to a random number of children.
Each of these new particles then repeats the same process, independently of each other.
More precisely, the children perform independent Brownian motions from the position of their parent at its death and, after independent exponential times, they in turn die, leaving a new generation to take over.
The numbers of children of the particles are assumed to be independent and identically distributed.
In particular, their genealogy is a Galton-Watson tree.
The choice of the exponential distributions for the lifetimes is justified by our wish to obtain a continuous-time Markovian process, since they are the only memoryless continuous distributions.
We also want the process to be time-homogeneous, which requires the trajectories to have independent and stationary increments.
By opting for Gaussian increments, we obtain Brownian motions.
This is a very natural choice, owing to their universal character.
Furthermore, this establishes connections with other mathematical fields such as the previously mentioned F-KPP equation (see Section~\ref{sct:speed_of_the_extremal_displacement_1}) or models from statistical physics (see Section~\ref{sct:rescaled_overlap_distribution}).
In Figure~\ref{fig:bbm_realizations}, we drew three realizations of branching Brownian motion.

\begin{figure}[ht]
    \centering
    \includegraphics[scale = 0.35]{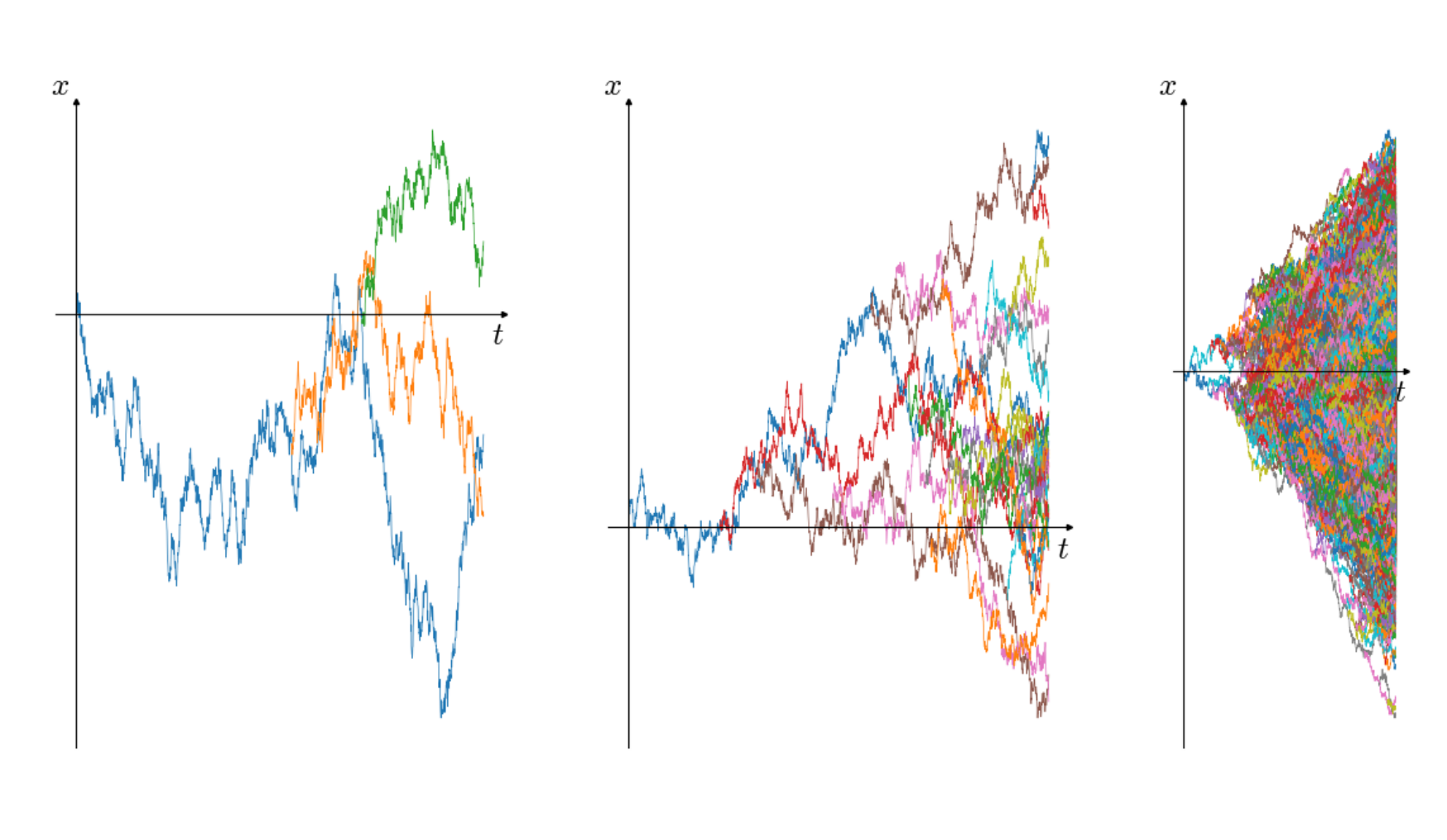}
    \caption{Realizations of branching Brownian motion over three different times.}\label{fig:bbm_realizations}
\end{figure}

Branching Brownian motion has a discrete-time counterpart, branching random walk (BRW).
In this model, the genealogy of the population is also a Galton-Watson tree but the particles do not perform any movement.
Instead, at time $n+1$, the particles of the $n^\textnormal{th}$ generation die while giving birth to children, scattered around their positions according to independent copies of a random point process.
Branching random walk naturally shares many properties with branching Brownian motion.
Nevertheless, some results established in discrete time are difficult to extend to continuous time, especially concerning the almost sure convergence.
This difficulty is illustrated below by the presence of technical lemmas such as Lemma~\ref{lem:almost_sure_convergence_of_functional_1_brutal_bounds} and Lemma~\ref{lem:almost_sure_convergence_of_functional_2_brutal_bounds}.

Of course, branching Brownian motion and branching random walk are far too simple to accurately describe the evolution of a real population (of cells, for example).
Firstly, it seems quite unreasonable to let the individuals diffuse in an unrestricted linear habitat.
Beyond that, in the event that the population is not subject to extinction, our construction lets it grow indefinitely at exponential speed.
However, one can add some complexity to the models in order to make it more realistic.

Here, our purpose is different.
Instead of the biological considerations, we are motivated by ideas from statistical physics.
In this context, we identify the positions of the particles alive at time $t \geq 0$ as the energy levels of a physical system.
Each particle can then be thought of as a \emph{configuration} of the system.
Let us denote by $\cN(t)$ the set of particles alive at time $t$ and by $X_u(t)$ the position of a given particle $u \in \cN(t)$.
The \emph{Gibbs measure}\footnote{The notion of Gibbs measure is central in statistical physics. For further motivations about this, we refer the interested reader to \cite[Chapter~1]{FriedliVelenik2017}.} at inverse temperature $\beta \geq 0$ is then the (random) probability measure on $\cN(t)$ that assigns to each configuration $u \in \cN(t)$ a weight proportional to $\e^{\beta X_u(t)}$.
By rescaling the associated \emph{partition function}, we obtain a martingale
\begin{equation*}
	W_t(\beta) = \e^{-(1+\beta^2/2)t} \sum_{u \in \cN(t)} \e^{\beta X_u(t)}.
\end{equation*}
This one is called \emph{additive martingale} of the branching Brownian motion at inverse temperature $\beta$.
As pointed out in the title, this is our main subject of study.
Note that, since it is a non-negative martingale, it converges almost surely to a random variable $W_\infty(\beta)$ as $t$ goes to infinity.

The thesis is organized as follows.
We start by giving a formal definition of branching Brownian motion in Section~\ref{sct:definition}.
We then present fundamental results about this model.
Among them, the \emph{many-to-one} and the \emph{many-to-two} formulas are very useful for first and second moment calculations.
They enable us to show in Section~\ref{sct:phase_transition} that a phase transition occurs: there exists a parameter $\beta_c > 0$ such that $W_\infty(\beta) = 0$ almost surely if $\beta \geq \beta_c$ and $W_\infty(\beta) > 0$ almost surely if $0 \leq \beta < \beta_c$.
Section~\ref{sct:front_and_extremes} is dedicated to results corresponding to the highest particles.
These results do not constitute the core of our subject but they are needed for the next sections.
Therefore, we present most of them without proof but with the appropriate references.
In Section~\ref{sct:position_under_gibbs_measure}, we study the particles that asymptotically contribute to $W_t(\beta)$.
It comes down to understanding the typical behavior of a particle chosen according to the Gibbs measure at inverse temperature $\beta$.
In Section~\ref{sct:growth_rates}, we show that $W_\infty(\beta)$ describes in a sense the asymptotic size of the population along the line with equation $x = \beta t$.
For this reason, this quantity plays a key role in the study of the long-term behavior of branching Brownian motion.
The main results of Section~\ref{sct:some_convergence_results}, Theorem~\ref{th:almost_sure_convergence_of_functional_1} and Theorem~\ref{th:growth_rates}, give the convergence of two processes, both almost surely and in $L^p$, for explicit values of $p$.
This was already known for branching random walk.
However, the extension of the almost sure convergence to continuous-time is not trivial and, to our knowledge, had never been done.
Afterward, we investigate in Section~\ref{sct:fluctuations} the fluctuations of the additive martingales around their almost sure limits.
Essentially, there are two different regimes, depending on whether $\beta$ is smaller or larger than $\beta_c/2$.
The fluctuations are first Gaussian and then $\alpha$-stable with $1 < \alpha < 2$ (see \eg \cite{Sato1999} for an introduction to stable distributions).
This was already studied in large part by Iksanov-Kolesko-Meiners in \cite{IksanovKoleskoMeiners2020} for the branching random walk.
We deepen their analysis by identifying stable distributions.
Finally, in Section~\ref{sct:rescaled_overlap_distribution}, we are interested in a question coming from spin glass theory: given $a \in (0, 1)$ and $\beta \in [0, \beta_c)$, if we pick two particles at time $t$ according to the Gibbs measure at inverse temperature $\beta$, what is the probability that their last common ancestor died after time $at$?
It is known that this probability converges to $0$ as $t$ goes to infinity.
We give a more accurate approximation in Theorem~\ref{th:renormalized_subcritical_overlap}, depending on the value of $\beta$.
In the regime $\beta_c/2 < \beta < \beta_c$, we observe $\alpha$-stable fluctuations with $0 < \alpha < 1$.
To our knowledge, this is the first time that such an asymptotic description is given.

Throughout the thesis, we use
the Bachmann-Landau notations $o$ and $O$ with their usual meaning.
We use the symbol $\mathds{1}$ for the indicator functions.
Given two random variables $X$ and $Y$, we write
\begin{equation*}
    X \overset{d}{=} Y
\end{equation*}
to indicate that they have the same distribution.
The letter $t$ denotes a certain point in time, which may vary in $[0, \infty)$ or remain fixed, depending on the context.
We avoid introducing it in each statement in order to lighten the writing.
The other notations, unless very standard, are introduced below when first used.

\section{Definitions and first tools}

The purpose of this section, once the model is properly defined, is to gather fundamental and now classic results about branching Brownian motion.
They are so commonly accepted that it is sometimes hard to find a formal proof in the literature.
We were inspired by several works to present these results and their demonstrations, including lecture notes by Berestycki \cite{Berestycki2014}, books by Athreya-Ney \cite{AthreyaNey1972} and by Bovier \cite{Bovier2016}.

\subsection{Definition of the branching Brownian motion}\label{sct:definition}

In order to give a formal definition of branching Brownian motion, we need to introduce the notion of (marked) trees and to choose the $\sigma$-algebra we want to work with.
Here, we follow the construction of Chauvin-Rouault \cite{ChauvinRouault1988} which is itself based on a work of Neveu \cite{Neveu1986}.

We denote by $\N$ the set of non-negative integers $0, 1, 2 \ldots$ and we define $\N^* = \N \setminus \{0\}$.
We equate the set of labels, sometimes called set of \emph{Ulam-Harris labels}, with the set of finite sequences of positive integers $\cU = \bigcup_{k \geq 0} (\N^*)^k$, with the convention $(\N^*)^0 = \{\varnothing\}$.
The elements $u \in \cU$ will represent potential \emph{individuals}, which can also be referred to as \emph{nodes} or \emph{particles}, depending on the context.
For two labels $u, v \in \cU$, we denote by $uv$ their concatenation, with the convention $u \varnothing = \varnothing u = u$.
We say that $u$ is an \emph{ancestor} of $v$ and that $v$ is a \emph{descendant} of $u$ if there exists $w \in \cU$ such that $v = uw$.
In this case, we write $u \leq v$.
The particle $\varnothing$ has therefore a specific role since it is the ancestor of all.
Given a particle $u = u_1 \ldots u_k \in \cU \setminus \{\varnothing\}$, we denote by $p(u) = u_1 \ldots u_{k-1}$ its \emph{parent}.
Finally, we denote by $u \wedge v$ the most recent common ancestor of two particles $u$ and $v$, \ie $u \wedge v = u_1 \ldots u_{k_0}$ with $k_0 = \sup\{k \geq 1 : u_1 \ldots u_k = v_1 \ldots v_k\}$ if $u_1 = v_1$, and $u \wedge v = \varnothing$ otherwise.

A \emph{rooted ordered tree}, also called a \emph{plane tree} or simply a \emph{tree}, is a subset $\cT$ of $\cU$ such that
\begin{enumerate}
    \item $\varnothing \in \cT$,
    \item for all $u, v \in \cU$, if $uv \in \cT$, then $u \in \cT$,
    \item for all $u \in \cT$, there exists an integer $A_u \geq 0$ such that for all $j \geq 1$, the label $uj$ is in $\cT$ if and only if $1 \leq j \leq A_u$.
\end{enumerate}

Given a topological space $X$, we denote by $\cC(X)$ the set of continuous functions from $X$ to $\R$.
A \emph{marked tree} is a tree $\cT$ endowed with \emph{marks} $(\sigma_u, Y_u)_{u \in \cT}$, where $\sigma_u \in \R_+$ and $Y_u \in \cC(\R_+)$ represent the lifetime and the displacement of the particle $u$, respectively.
We denote by $\Omega$ the space of marked trees.
It will be convenient to introduce $b_u = \sum_{v \leq p(u)} \sigma_v$ the birthtime of a particle $u$, and $d_u = \sum_{v \leq u} \sigma_v$ its deathtime, with the convention $b_\varnothing = 0$.
The set of particles alive at time $t \geq 0$ is omnipresent in the study of branching processes, we denote it by
\begin{equation*}
    \cN(t) = \{u \in \cT : b_u \leq t < d_u\}.
\end{equation*}
Besides, instead of the displacement, we shall prefer to work with the position of a particle, \ie its displacement shifted by the position of its parent when the latter died.
More precisely, we define inductively the position at time $t \geq 0$ of a particle $u \in \cN(t)$
\begin{equation*}
    X_u(t) = \begin{cases}
        Y_\varnothing(t) &\text{if } u = \varnothing, \\
        X_{p(u)}(b_u) + Y_u(t - b_u) &\text{otherwise}.
    \end{cases} 
\end{equation*}
We extend the notion of position for a particle $u \in \cN(t)$ to the whole interval $[0, t]$: for all $s \in [0, t)$, we set $X_u(s) = X_v(s)$, where $v$ is the ancestor of $u$ alive at time $s$.

Let us consider, for $u \in \cU$, $\Omega_u = \{(\cT, (\sigma, Y)) : u \in \cT\}$ the set of marked trees which contain $u$.
Note that, given $u \in \cU$, $A_u$, $\sigma_u$ and $(X_u(s))_{b_u \leq s < d_u}$ naturally define maps from $\Omega_u$ to $\N$, $\R_+$ and $\R^{[0, \infty)}$, respectively\footnote{Here, we identify $(X_u(s))_{b_u \leq s < d_u}$ with $(X_u(s) \mathds{1}_{b_u \leq s < d_u})_{s \geq 0}$.}.
We equip $\Omega$ with the $\sigma$-algebra $\cF$ generated by these maps, \ie
\begin{equation*}
    \cF = \sigma\left(A_u, \sigma_u, (X_u(s))_{b_u \leq s < d_u} : u \in \cU\right),
\end{equation*}
where we consider the Borel $\sigma$-algebra on $\R$ and the product $\sigma$-algebra on $\R^{[0, \infty)}$.
In order to obtain a branching property (see \eqref{eq:branching_property}), let us denote by $\Theta_u$ the operator which associates to $(\cT, (\sigma, Y)) \in \Omega_u$ the marked tree $(\cT', (\sigma_{uv}, Y_{uv})_{v \in \cT'})$, where $\cT' = \{v \in \cU : uv \in \cT\}$.
This shift operator is measurable with respect to the $\sigma$-algebra induced by $\cF$ on $\Omega_u$.

For $t \geq 0$, we define $\cF_t$ the $\sigma$-algebra that contains all the information available at time $t$
\begin{equation*}
    \cF_t = \sigma\left(A_u, \sigma_u, (X_u(s))_{b_u \leq s < d_u} : u \in \cU \text{ such that } d_u \leq t\right) \vee \sigma\left((X_u(s))_{b_u \leq s \leq t} : u \in \cU \text{ such that } b_u \leq t < d_u\right).
\end{equation*}
The collection $(\cF_t)_{t \geq 0}$ is then a filtration on $\Omega$.
As usual, we denote by $\cF_\infty$ the $\sigma$-algebra generated by this filtration
\begin{equation*}
    \cF_\infty = \bigvee_{t \geq 0} \cF_t.
\end{equation*}

The following proposition is a well-known result which enables us to define the branching Brownian motion.
It gives the existence and the uniqueness of its distribution.
One can deduce it from \cite[Section~3]{Neveu1986}.

\begin{proposition}\label{prop:characterization_of_bbm}
    Let $\mu = (\mu(k))_{k \geq 0}$ be a distribution on $\N$ and $\lambda > 0$. There exists a unique probability measure $\P$ on $(\Omega, \cF)$ such that
    \begin{enumerate}
        \item $A_\varnothing$ has distribution $\mu$, $\sigma_\varnothing$ has exponential distribution with rate parameter $\lambda$, $Y_\varnothing$ is a standard Brownian motion, $A_\varnothing$, $\sigma_\varnothing$, $Y_\varnothing$ are independent,
        \item by conditioning on $\cF_{\sigma_\varnothing}$, the trees $\Theta_1, \ldots, \Theta_{A_\varnothing}$ are independent and identically distributed with distribution $\P$, \ie for all sequences $(f_k)_{k \geq 1}$ of non-negative measurable functions from $\Omega$ to $\R$,
        \begin{equation}\label{eq:branching_property}
            \condExpec{\prod_{k = 1}^{A_\varnothing}{f_k \circ \Theta_k}}{\cF_{\sigma_\varnothing}} = \prod_{k = 1}^{A_\varnothing} \Expec{f_k}.
        \end{equation}
    \end{enumerate}
\end{proposition}
The branching property \eqref{eq:branching_property} can be generalized to any stopping time $\tau$ (not necessarily $\sigma_\varnothing$): at time $\tau$, the particles in $\cN(\tau)$ all start independent branching Brownian motions shifted by their positions.
See \cite[Proposition~2.1]{Chauvin1991} for a formalization.

The measure $\P$ in Proposition~\ref{prop:characterization_of_bbm} is the \emph{distribution of the branching Brownian motion} with offspring distribution $\mu$ and rate parameter $\lambda$.
We can, without loss of generality, and shall in what follows, set $\lambda = 1$.
Furthermore, we shall assume that $\mu$ has mean $2$ and has a finite second moment, \ie
\begin{equation}\label{eq:assumption_offspring_distribution}
    \sum_{k \geq 0} \mu(k) k = 2 \quad \text{and} \quad \sum_{k \geq 0} \mu(k) k^2 < \infty.
\end{equation}
Note that nothing happens if a particle splits into one particle.
However, we shall keep this eventuality in order to maintain the mean number of offspring equal to $2$.
In many papers dealing with the branching Brownian motion, the authors prevent the extinction event by assuming $\mu(0) = 0$, or even impose a binary branching $\mu = \delta_2$, where $\delta_x$ denotes the Dirac delta distribution concentrated at $x$.
Most of the results obtained with the latter settings still hold under \eqref{eq:assumption_offspring_distribution}.
We believe that this is the case for all the content of this thesis.
Therefore, we take it as our basic framework and we clearly point out the changes of assumptions.

\subsection{Some results about the genealogy}

Given a set $E$, we denote by $\#E$ its cardinality.
Let us define $n(t) = \#\cN(t)$ the number of particles alive at time $t$.
Note that $n(t)$ is $\cF_t$-measurable.

\begin{proposition}\label{prop:expectation_nt}
    We have $\Expec{n(t)} = \e^t$.
\end{proposition}

\begin{proof}
    The first step is to check that the number of particles does not explode in finite time.
    Let us define, for $x \in [0, 1]$,
    \begin{equation*}
        G(t, x) = \sum_{k \geq 0} \Prob{n(t) = k} x^k = \Expec{x^{n(t)} \mathds{1}_{n(t) < \infty}}.
    \end{equation*}
    Let $\tau$ be the first branching time of the system.
    If $\tau > t$, then $n(t) = 1$.
    Otherwise, the initial particle is replaced at time $\tau \leq t$ by new particles $u_1, \ldots, u_A$, where $A$ follows the offspring distribution $\mu$.
    In this case, we have $n(t) = n_1(t - \tau) + \cdots + n_A(t - \tau)$, where $n_1(s), \ldots, n_A(s)$ are the numbers of descendants of $u_1, \ldots, u_A$, respectively, alive at time $\tau + s$.
    By construction of the branching Brownian motion, the random variables $\tau$, $A$, $(n_i)_{i \geq 1}$ are independent and $n_i$, $i \geq 1$, are independent copies of $n$.
    Thus, for any $x \in [0, 1)$,
    \begin{align}
        G(t, x) &= \Expec{x \mathds{1}_{\tau > t}} + \Expec{x^{n_1(t - \tau) + \cdots + n_A(t - \tau)} \mathds{1}_{n_1(t - \tau) + \cdots + n_A(t - \tau) < \infty} \mathds{1}_{\tau \leq t}} \notag \\
        &= x \Prob{\tau > t} + \Expec{G(t - \tau, x)^A \mathds{1}_{\tau \leq t}} \notag \\
        &= x \e^{-t} + \int_0^t f(G(t - s, x)) \e^{-s} \d{s}, \label{eq:rewriting_generating_function}
    \end{align}
    where $f : y \mapsto \Expec{y^A}$ is smooth on $(0, 1)$ and continuous on $[0, 1]$.
    Letting $x$ grow to $1$, \eqref{eq:rewriting_generating_function} becomes
    \begin{equation*}
        G(t, 1) = \e^{-t} + \int_0^t f(G(t - s, 1))  \e^{-s} \d{s},
    \end{equation*}
    by dominated convergence.
    Then, the function $t \mapsto G(t, 1)$ is the unique solution of the differential equation $y' = -y + f(y)$ with initial condition $y(0) = 1$. Hence, $\Prob{n(t) < \infty} = G(t, 1) = 1$.
    
    Now, we present the arguments of \cite[Section~4.5]{AthreyaNey1972} to show that $\Expec{n(t)}$ is finite.
    We will use that $\Expec{n(t)} = \lim_{x \to 1^-} \partial_x G(t, x)$, by monotone convergence.
    Since $0 \leq G(t, x) \leq 1$ for any $x \in (0, 1)$, we can differentiate \eqref{eq:rewriting_generating_function} in $x \in (0, 1)$
    \begin{align*}
        \partial_x G(t, x) &= \e^{-t} + \int_0^t \partial_x G(t - s, x) \Expec{f'(G(t - s, x))} \e^{-s} \d{s} \\
        &\leq h(t) + \Expec{A} (g_x * h) (t),
    \end{align*}
    where $g_x(t) := \partial_x G(t, x) \mathds{1}_{t \geq 0}$, $h(t) := \e^{-t} \mathds{1}_{t \geq 0}$ and $g_x * h$ denotes the convolution of $g_x$ and $h$.
    Recall that $\Expec{A} = 2$.
    By induction, for any $k \geq 1$,
    \begin{equation}\label{eq:rewriting_differential_with_convolution}
        \partial_x G(t, x) \leq \sum_{i = 1}^{k} 2^{i - 1} h^{*i}(t) + 2^k (g_x * h^{*k})(t),
    \end{equation}
    where $h^{*i}$ denotes the $i$-fold convolution of $h$.
    Note that we can rewrite $h^{*i}(t) = \Prob{E_1 + \cdots + E_i \leq t}$, where $E_1, \ldots, E_i$ are independent exponential variables with mean $1$.
    By Markov's inequality, for any $\lambda > 0$,
    \begin{equation*}
        h^{*i}(t) \leq \e^{\lambda t} \Expec{\e^{-\lambda(E_1 + \cdots + E_i)}} = \e^{\lambda t} \Expec{\e^{-\lambda E_1}}^i.
    \end{equation*}
    Taking $\lambda$ large enough, we see that the series $\sum_i 2^{i - 1} h^{*i}(t)$ converges.
    In particular, letting $k$ go to infinity, \eqref{eq:rewriting_differential_with_convolution} becomes
    \begin{equation*}
        \partial_x G(t, x) \leq \sum_{i \geq 1} 2^{i - 1} h^{*i}(t) < \infty.
    \end{equation*}
    Since this bound does not depend on $x$, we obtain $\Expec{n(t)} = \lim_{x \to 1^-} \partial_x G(t, x) < \infty$.
    
    It remains to compute $\Expec{n(t)}$.
    The same arguments as for \eqref{eq:rewriting_generating_function} lead to
    \begin{equation*}
        \Expec{n(t)} = \e^{-t} + 2 \int_0^t \Expec{n(t - s)} \e^{-s} \d{s}.
    \end{equation*}
    Then, the function $t \mapsto \Expec{n(t)}$ satisfies the differential equation $y' = y$ with initial condition $y(0) = 1$, which concludes.
\end{proof}

\begin{remark}
    In the case of binary branching $\mu = \delta_2$, the process $(n(t))_{t \geq 0}$ is a \emph{Yule process} with parameter $1$, \ie a continuous-time Markov chain on $\N^*$ with transition-rate matrix given by $q_{i, i+1} = i$ and $q_{i, i} = -i$ for all $i \in \N^*$.
    The distribution of $n(t)$ is then geometric.
    Indeed, \eqref{eq:rewriting_generating_function} yields, for any $x \in [0, 1)$,
    \begin{equation*}
        \Expec{x^{n(t)}} = x \e^{-t} + \int_0^t \Expec{x^{n(t-s)}}^2 \e^{-s} \d{s}.
    \end{equation*}
    The above function is solution of the differential equation $y' = \e^{-t} y^2$ with initial condition $y(0) = x$.
    This implies that
    \begin{equation}\label{eq:generating_function_nt}
        \Expec{x^{n(t)}} = \frac{x \e^{-t}}{1 - x(1 - \e^{-t})},
    \end{equation}
    which is the probability-generating function of the geometric distribution on $\N^*$ with parameter $\e^{-t}$.
\end{remark}

More generally, we can define $\cN([s, t])$ the set of particles that have been alive between times $s$ and $t$, and $n([s, t]) = \#\cN([s, t])$.

\begin{lemma}\label{lem:particles_alive_in_an_interval}
    We have $\Expec{n([s, t])} = \e^t + \mu(0)(\e^t - e^s)$.
\end{lemma}

\begin{proof}
    With the same arguments as for \eqref{eq:rewriting_generating_function}, we obtain that the expected number of particles that died before time $t$ satisfies the differential equation $y' = y + \mu(0)$.
    Hence,
    \begin{equation}\label{eq:particles_alive_between_0_t}
        \Expec{n([0, t])} = \e^t + \mu(0)(\e^t - 1).
    \end{equation}
    Now,
    \begin{equation*}
        \Expec{n([s, t])} = \Expec{\sum_{u \in \cN(s)} n_s^{(u)}([0, t-s])},
    \end{equation*}
    where, for $u \in \cN(s)$, we denote by $n_s^{(u)}([0, r])$ the number of descendants of $u$ that have been alive between times $s$ and $r + s$.
    By conditioning on the underlying tree $(\cT, \sigma) = (\cT, (\sigma_u)_{u \in \cU})$ defined in Section~\ref{sct:definition}, we obtain
    \begin{equation}\label{eq:particles_alive_between_s_t}
        \Expec{n([s, t])} = \Expec{\sum_{u \in \cN(s)} \condExpec{n_s^{(u)}([0, t-s])}{(\cT, \sigma)}} = \e^s \Expec{n([0, t-s])}.
    \end{equation}
    Lemma~\ref{lem:particles_alive_in_an_interval} directly follows from \eqref{eq:particles_alive_between_0_t} and \eqref{eq:particles_alive_between_s_t}.
\end{proof}

The next lemma will be useful for a later result, namely the \emph{many-to-two formula} (Lemma~\ref{lem:many-to-two}), which is itself a fundamental tool for the study of the branching Brownian motion.
In the proof, we use some ideas from Sawyer \cite[p.~686]{Sawyer1976}.

\begin{lemma}\label{lem:death_functional}
    Let $f : \R_+ \to \R$ be a measurable function.
    Assume $f$ to be non-negative or bounded.
    Then,
    \begin{equation}\label{eq:death_functional}
        \Expec{\sum_{\substack{u, v \in \cN(t) \\ u \neq v}} f(d_{u \wedge v})} = K \e^{2t} \int_0^t f(s) \e^{-s} \d{s},
    \end{equation}
    where $K := \sum_{k \geq 0} \mu(k)k(k-1)$ and with the convention $\sum_\varnothing = 0$.
\end{lemma}

\begin{proof}
    Let us define $(\tau_n)_{n \geq 1}$ the sequence of the consecutive branching times of the whole branching Brownian motion.
    This allows us to rewrite
    \begin{align}
        \sum_{\substack{u, v \in \cN(t) \\ u \neq v}} f(d_{u \wedge v}) &= \sum_{n \geq 1} \sum_{u, v \in \cN(t)} f(\tau_n) \mathds{1}_{d_{u \wedge v} = \tau_n \leq t} \notag \\
        &= \sum_{n \geq 1} f(\tau_n) \mathds{1}_{\tau_n \leq t} \#\{u, v \in \cN(t) : d_{u \wedge v} = \tau_n\} \notag \\
        &= \sum_{n \geq 1} f(\tau_n) \mathds{1}_{\tau_n \leq t} \sum_{\substack{u', v' \in \cN(\tau_n) \\ d_{u' \wedge v'} = \tau_n}} \#\{u \in \cN(t) : u \geq u'\} \#\{v \in \cN(t) : v \geq v'\}. \label{eq:many-to-two_rewriting}
    \end{align}
    Conditionally on $\cF_{\tau_n}$, if $u'$ and $v'$ are two particles of $\cN(\tau_n)$ such that $d_{u' \wedge v'} = \tau_n$, then, at time $\tau_n$, they start two branching Brownian motions, independent of each other and of $\cF_{\tau_n}$.
    Thus, using Proposition~\ref{prop:expectation_nt}, we obtain
    \begin{equation}\label{eq:many-to-two_first_conditioning}
        \condExpec{\#\{u \in \cN(t) : u \geq u'\} \#\{v \in \cN(t) : v \geq v'\}}{\cF_{\tau_n}} = \e^{2t - 2\tau_n}.
    \end{equation}
    We also have
    \begin{align*}
        \#\{u', v' \in \cN(\tau_n) : d_{u' \wedge v'} = \tau_n\} &= \#\{u', v' \in \cN(\tau_n) : b_{u'} = b_{v'} = \tau_n, u' \neq v'\} \\
        &= A_n(A_n - 1),
    \end{align*}
    where $A_n = \#\{u' \in \cN(\tau_n) : b_{u'} = \tau_n\}$ follows the offspring distribution $\mu$ and is independent of $\tau_n$.
    Thus,
    \begin{equation}\label{eq:many-to-two_second_conditioning}
        \condExpec{\#\{u', v' \in \cN(\tau_n) : d_{u' \wedge v'} = \tau_n\}}{\tau_n} = \Expec{A_n(A_n - 1)} = K.
    \end{equation}
    By conditioning, first on $\cF_{\tau_n}$ then on $\tau_n$, the equations \eqref{eq:many-to-two_rewriting}, \eqref{eq:many-to-two_first_conditioning}, \eqref{eq:many-to-two_second_conditioning} yield
    \begin{equation}\label{eq:many-to-two_quantity_of_interest}
        \Expec{\sum_{\substack{u, v \in \cN(t) \\ u \neq v}} f(d_{u \wedge v})} = \Expec{\sum_{n \geq 1} K \e^{2t - 2\tau_n} f(\tau_n) \mathds{1}_{\tau_n \leq t}}.
    \end{equation}
    Let us temporarily replace $s \mapsto \e^{2t - 2s} f(s)$ with some continuously differentiable function $\phi : \R \to \R$.
    We can write
    \begin{equation}\label{eq:many-to-two_mute_function_method}
        \Expec{\sum_{n \geq 1} \phi(\tau_n) \mathds{1}_{\tau_n \leq t}} = \phi(0) \Expec{\sum_{n \geq 1} \mathds{1}_{\tau_n \leq t}} + \int_0^t \phi'(s) \Expec{\sum_{n \geq 1} \mathds{1}_{s < \tau_n \leq t}} \d{s}.
    \end{equation}
    Note that, by definition of $\tau_n$ and $A_n$,
    \begin{equation*}
        \e^t = \Expec{n(t)} = \Expec{1 + \sum_{n \geq 1} \mathds{1}_{\tau_n \leq t} (A_n - 1)} = 1 + \Expec{\sum_{n \geq 1} \mathds{1}_{\tau_n \leq t}}.
    \end{equation*}
    Then, \eqref{eq:many-to-two_mute_function_method} becomes
    \begin{equation*}
        \Expec{\sum_{n \geq 1} \phi(\tau_n) \mathds{1}_{\tau_n \leq t}} = \phi(0) (\e^t - 1) + \int_0^t \phi'(s) (\e^t - \e^s) \d{s} = \int_0^t \phi(s) \e^s \d{s}.
    \end{equation*}
    Since the total mass $\Expec{\sum_n \mathds{1}_{\tau_n \leq t}}$ is finite, the above formula still holds for any measurable function $\phi : \R \to \R$ which is non-negative or bounded.
    Taking $\phi : s \mapsto \e^{2t - 2s} f(s)$, \eqref{eq:many-to-two_quantity_of_interest} becomes \eqref{eq:death_functional}.
\end{proof}

\subsection{The many-to-one lemma}\label{sct:many-to-one}

One of the most basic tools in the study of branching Brownian motion, and branching processes in general, is the \emph{many-to-one formula}.
It enables the calculation of first moments for certain functionals.

\begin{lemma}[Many-to-one lemma]
    Let $F : \cC([0, t]) \to \R$ be a measurable functional.
    Assume that $F$ is non-negative or that $F(B_s, 0 \leq s \leq t)$ admits a first moment.
    Then,
    \begin{equation}\label{eq:many-to-one}
        \Expec{\sum_{u \in \cN(t)} F(X_u(s), 0 \leq s \leq t)} = \e^t \Expec{F(B_s, 0 \leq s \leq t)},
    \end{equation}
    with the convention $\sum_\varnothing = 0$, where $(B_t)_{t \geq 0}$ is a standard Brownian motion under $\P$.
\end{lemma}

\begin{proof}
    The calculation is straightforward, by conditioning on the underlying tree $(\cT, \sigma) = (\cT, (\sigma_u)_{u \in \cU})$ defined in Section~\ref{sct:definition}.
    Indeed, we have
    \begin{align*}
        \Expec{\sum_{u \in \cN(t)} F(X_u(s), 0 \leq s \leq t)} &= \Expec{\sum_{u \in \cN(t)} \condExpec{F(X_u(s), 0 \leq s \leq t)}{(\cT, \sigma)}} \\
        &= \Expec{n(t)} \Expec{F(B_s, 0 \leq s \leq t)},
    \end{align*}
    since each trajectory $(X_u(s), 0 \leq s \leq t)$ follows a Brownian motion independent of $(\cT, \sigma)$.
    Using Proposition~\ref{prop:expectation_nt}, we obtain \eqref{eq:many-to-one}.
\end{proof}

\subsection{Additive and derivative martingales}\label{sct:additive_and_derivative_martingales}

From now on, $B = (B_t)_{t \geq 0}$ will denote a standard Brownian motion under $\P$.
Furthermore, we will assume $B$ to be independent of the whole branching Brownian motion.

\begin{proposition}
    For each $t \geq 1$, let $F_t : \cC([0, t]) \to \R$ be a measurable functional and set
    \begin{equation*}
        M_t = F_t(B_s, 0 \leq s \leq t).
    \end{equation*}
    If $(M_t)_{t \geq 0}$ is a martingale with respect to the natural filtration of $(B_t)_{t \geq 0}$, then
    \begin{equation*}
        W_t := \e^{-t} \sum_{u \in \cN(t)} F_t(X_u(s), 0 \leq s \leq t)
    \end{equation*}
    defines a martingale with respect to $(\cF_t)_{t \geq 0}$.
\end{proposition}

\begin{proof}
    Let $0 \leq s \leq t$. We have
    \begin{equation}\label{eq:cond_expec_Zt}
        \condExpec{W_t}{\cF_s} = \e^{-t} \sum_{u \in \cN(s)} \condExpec{\sum_{\substack{v \in \cN(t) \\ v \geq u}} F_t(X_v(r), 0 \leq r \leq t)}{\cF_s}.
    \end{equation}
    In order to use the Markov property of the branching Brownian motion, let us fix $u \in \cN(s)$ and define, for any $\phi \in \cC([0, t-s])$ such that $\phi(0) = 0$,
    \begin{equation*}
        F_t^{(u)}(\phi(r), 0 \leq r \leq t - s) = F_t(\tilde{\phi}(r), 0 \leq r \leq t),
    \end{equation*}
    where
    \begin{equation*}
        \tilde{\phi}(r) = \begin{cases}
            X_u(r) &\text{if } 0 \leq r \leq s, \\
            X_u(s) + \phi(r - s) &\text{if } s < r \leq t.
        \end{cases}
    \end{equation*}
    Recall that $B$ is independent of the whole branching Brownian motion.
    We can then rewrite the summand in \eqref{eq:cond_expec_Zt} as
    \begin{align*}
        \condExpec{\sum_{\substack{v \in \cN(t) \\ v \geq u}} F_t^{(u)}(X_v(s + r) - X_u(s), 0 \leq r \leq t-s)}{\cF_s} = \e^{t - s} \condExpec{F_t^{(u)}(B_r, 0 \leq r \leq t-s)}{\cF_s},
    \end{align*}
    where we have used that, conditionally on $\cF_s$, the particle $u$ starts an independent branching Brownian motion with initial position $X_u(s)$, to which we can apply the many-to-one formula \eqref{eq:many-to-one}.
    Coming back to \eqref{eq:cond_expec_Zt}, it yields
    \begin{align*}
        \condExpec{W_t}{\cF_s} &= \e^{-s} \sum_{u \in \cN(s)} \condExpec{F_t^{(u)}(B_r, 0 \leq r \leq t-s)}{\cF_s} \\
        &= \e^{-s} \sum_{u \in \cN(s)} F_s(X_u(r), 0 \leq r \leq s),
    \end{align*}
    since $\left(F_{s + t}^{(u)}(B_r, \ 0 \leq r \leq t)\right)_{t \geq 0}$\footnote{Formally, we should replace $F_t^{(u)}(B_r, 0 \leq r \leq t)$ with
    \begin{equation*}
        F_t^{(u)}(B_r, 0 \leq r \leq t) \mathds{1}_{u \in \cN(s)} + F_t(B_r, 0 \leq r \leq s+t) \mathds{1}_{u \notin \cN(s)},
    \end{equation*}
    since the former expression is not well defined.} has the same distribution as $(M_{s + t})_{t \geq 0}$, which is a martingale.
\end{proof}

\begin{corollary}
    For any $\beta \in \R$,
    \begin{equation}\label{eq:definition_of_W_t}
        W_t(\beta) := \e^{-c(\beta)t} \sum_{u \in \cN(t)} \e^{\beta X_u(t)}
    \end{equation}
    and
    \begin{equation}\label{eq:definition_of_Z_t}
        Z_t(\beta) := \e^{-c(\beta)t} \sum_{u \in \cN(t)} (X_u(t) - \beta t) \e^{\beta X_u(t)}
    \end{equation}
    define martingales, where $c(\beta) = 1+\beta^2/2$.
\end{corollary}

The processes $(W_t(\beta))_{t \geq 0}$ and $(Z_t(\beta))_{t \geq 0}$ are called, respectively, the \emph{additive martingale} and the \emph{derivative martingale} of the branching Brownian motion associated with the inverse temperature $\beta$.
The first one was introduced by McKean in \cite{McKean1975}.
Since it is positive, it converges almost surely to a random variable $W_\infty(\beta)$.
We will see in Section~\ref{sct:phase_transition} that this limit is non-degenerate if and only if $|\beta|$ is smaller than an explicit critical value.
As for the derivative martingale, it was introduce by Lalley-Sellke in \cite{LalleySellke1987}.
We will discuss its asymptotic behavior in Section~\ref{sct:phase_transition}.

\subsection{The many-to-two lemma}

The many-to-one formula introduced in Section~\ref{sct:many-to-one} does not provide any information about the correlations between the trajectories of the particles.
Indeed, at time $t$, it gives the same result as for $\e^t$ independent Brownian motions.
A way to capture the correlations is to compute the $k^\mathrm{th}$ moment of
sums over the particles, with $k \geq 2$.
The case $k = 2$ is made possible by the following \emph{many-to-two formula}.
In \cite{HarrisRoberts2017}, Harris-Roberts developed generalizations of this formula, called \emph{many-to-few}.

\begin{lemma}[Many-to-two lemma]\label{lem:many-to-two}
    Let $F : \cC([0, t])^2 \to \R$ be a measurable functional.
    Assume that $F$ is non-negative or bounded.
    Then,
    \begin{multline}\label{eq:many-to-two}
        \Expec{\sum_{\substack{u, v \in \cN(t) \\ u \neq v}} F\left((X_u(s))_{0 \leq s \leq t}, (X_v(s))_{0 \leq s \leq t}\right)} \\
        = K \e^{2t} \int_0^t \Expec{F\left((B^{1, s}_r))_{0 \leq r \leq t}, (B^{2, s}_r)_{0 \leq r \leq t}\right)} e^{-s} \d{s},
    \end{multline}
    where $K = \sum_{k \geq 0} \mu(k)k(k-1)$ and the processes $B^{1, s}$ and $B^{2, s}$ are two Brownian motions that coincide on $[0, s]$ and are independent afterward.
    More precisely, the process $(B^{1, s}_r)_{0 \leq r \leq t}$ is a Brownian motion and $(B^{2, s}_r)_{0 \leq r \leq t}$ is defined by
    \begin{equation*}
        B^{2, s}_r = \begin{cases}
            B^{1, s}_r &\text{if } 0 \leq r \leq s, \\
            B^{1, s}_s + \tilde{B}_{r - s} &\text{if } s < r \leq t,
        \end{cases}
    \end{equation*}
    where $\tilde{B}$ is a Brownian motion independent of $B^{1, s}$.
\end{lemma}

The above many-to-two lemma is a generalization of \cite[Lemma~10]{Bramson1978}.
In this article, Bramson gives heuristic arguments, which we make here rigorous thanks to Lemma~\ref{lem:death_functional}.

\begin{proof}[Proof of Lemma~\ref{lem:many-to-two}]
    As for the many-to-one lemma, we start by conditioning on the underlying tree.
    Since the trajectories of the particles are independent of this tree, we have
    \begin{equation*}
        \sum_{\substack{u, v \in \cN(t) \\ u \neq v}} \condExpec{F\left((X_u(s))_{0 \leq s \leq t}, (X_v(s))_{0 \leq s \leq t}\right)}{(\cT, \sigma)} = \sum_{\substack{u, v \in \cN(t) \\ u \neq v}} f(d_{u \wedge v}),
    \end{equation*}
    where, with the notation introduced in the statement,
    \begin{equation*}
        f(s) = \Expec{F\left((B^{1, s}_r)_{0 \leq r \leq t}, (B^{2, s}_r)_{0 \leq r \leq t}\right)}.
    \end{equation*}
    We conclude using \eqref{eq:death_functional}.
\end{proof}

\subsection{Phase transition}\label{sct:phase_transition}

Recall that the additive martingales and the derivative martingales are defined by \eqref{eq:definition_of_W_t} and \eqref{eq:definition_of_Z_t}.

\begin{theorem}[Phase transition]\label{th:phase_transition_additive_martingales}
    The almost sure limit of the additive martingale $(W_t(\beta))_{t \geq 0}$ is non-degenerate if and only if $|\beta| < \beta_c$, with $\beta_c = \sqrt{2}$.
    More precisely,
    \begin{itemize}
        \item if $|\beta| \geq \beta_c$, then $W_\infty(\beta) = 0$ almost surely,
        \item if $|\beta| < \beta_c$, then $W_\infty(\beta)$ is an $L^1$ limit and is almost surely positive on the event of survival.
    \end{itemize}
\end{theorem}

By appealing to \emph{spine decompositions}, Kyprianou established in \cite[Theorem~1]{Kyprianou2004} that Theorem~\ref{th:phase_transition_additive_martingales} still holds if we replace the assumption \eqref{eq:assumption_offspring_distribution} with the following one
\begin{equation*}
    \mu(0) = 0, \quad \sum_{k \geq 0} \mu(k) k < \infty, \quad \sum_{k \geq 0} \mu(k) k \log_+ k < \infty,
\end{equation*}
where $\log_+ x = \log x$ if $x \geq 1$ and $\log_+ x = 0$ if $x < 1$.
Furthermore, this condition is optimal in the sense that if $\sum_k \mu(k) k \log_+ k = \infty$, then $W_\infty(\beta) = 0$ almost surely, whatever $\beta$.

We will see in Theorem~\ref{th:growth_rates} that $W_\infty(\beta)$ describes the asymptotic size of the population along the slope $\beta$.
Therefore, this quantity naturally appears in the asymptotic of processes which mainly depend on what happens around $\beta t$ (see \eg Theorem~\ref{th:almost_sure_convergence_of_functional_1} and Theorem~\ref{th:renormalized_subcritical_overlap}).
If $|\beta| \geq \beta_c$, though, $W_\infty(\beta)$ is almost surely equal to $0$.
In this case, the only information that it contains is that there is no population around $\beta t$, with high probability.
Another quantity is then appropriate to study what happens around $\beta_c t$ and beyond, that is the almost sure limit of the critical derivative martingale (see \eg Proposition~\ref{prop:convergence_of_the_critical_additive_martingale}, Proposition~\ref{prop:convergence_of_the_supercritical_additive_martingales}, Theorem~\ref{th:convergence_of_the_extremal_process}).

A phase transition is also observed in the asymptotic behavior of the derivative martingale.
In the case of the branching random walk, Biggins showed in \cite[Theorem~3]{Biggins1991} that, if $|\beta| < \beta_c$ and under appropriate moment conditions, $(Z_t(\beta))_{t \geq 0}$ converges almost surely and in mean to a non-degenerate random variable $Z_\infty(\beta)$.
The critical case $\beta = \beta_c$ was studied by Lalley-Sellke in \cite[Theorem~1]{LalleySellke1987} and by Neveu in \cite[Proposition~2]{Neveu1988}, for the branching Brownian motion with binary branching.
They proved that, in this case, $Z_t := Z_t(\beta_c)$ converges almost surely to a positive random variable $Z_\infty := Z_\infty(\beta_c)$.
Yang-Ren showed in \cite[Theorem~1]{YangRen2011} that the latter result still holds on the event of survival if we replace the binary assumption with the optimal condition $\Expec{A \log_+^2 A} < \infty$, where $A$ has distribution~$\mu$.
Finally, for the supercritical regime $|\beta| > \beta_c$, Kyprianou established in \cite[Theorem~3]{Kyprianou2004} that the derivative martingale converges almost surely to $0$.

Let us now tackle the proof of Theorem~\ref{th:phase_transition_additive_martingales}.
First of all, we need the following \emph{zero-one law}.

\begin{lemma}\label{lem:zero-one_law_additive_martingales}
    Let us denote by $q$ the probability of extinction, \ie $q := \Prob{\exists t \geq 0, \cN(t) = \varnothing}$.
    For any $\beta \in \R$, the probability $\Prob{W_\infty(\beta) = 0}$ is either $q$ or $1$.
\end{lemma}

\begin{proof}
    Set $p = \Prob{W_\infty(\beta) = 0}$.
    Let us rewrite, for $s, t \geq 0$,
    \begin{equation}\label{eq:definition_of_W_s^(u)}
        W_{t + s}(\beta) = \sum_{u \in \cN(t)} \e^{\beta X_u(t) - c(\beta)t} \underbrace{\sum_{\substack{v \in \cN(t + s) \\ v \geq u}} \e^{\beta (X_v(t + s) - X_u(t)) - c(\beta)s}}_{=: W_s^{(u)}(\beta)}.
    \end{equation}
    By the Markov property, conditionally on $\cF_t$, the processes $(W_s^{(u)}(\beta))_{s \geq 0}$, $u \in \cN(t)$, are the additive martingales of independent copies of the initial branching Brownian motion.
    Letting $s$ go to infinity, we obtain
    \begin{equation}\label{eq:definition_of_W_infty^(u)}
        W_{\infty}(\beta) = \sum_{u \in \cN(t)} \e^{\beta X_u(t) - c(\beta)t} W_\infty^{(u)}(\beta) \quad \text{almost surely}, 
    \end{equation}
    where given $\cF_t$, $W_\infty^{(u)}(\beta)$, $u \in \cN(t)$, are independent copies of $W_\infty(\beta)$.
    
    Now, assume that $p < 1$.
    We have
    \begin{equation}\label{eq:zero_one_law_expectation}
        p = \Prob{\sum_{u \in \cN(t)} \e^{\beta X_u(t) - c(\beta)t} W_\infty^{(u)}(\beta) = 0} = \Prob{\forall u \in \cN(t), W_\infty^{(u)}(\beta) = 0} = \Expec{p^{n(t)}}.
    \end{equation}
    In the case of a binary branching, \eqref{eq:zero_one_law_expectation} is equal to ${p \e^{-t}}/(1 - p(1 - \e^{-t}))$, by \eqref{eq:generating_function_nt}, which implies that $p = 0 = q$.
    More generally, let us show that, on the event of survival, $n(t)$ tends almost surely to infinity as $t$ goes to infinity.
    If the offspring distribution satisfies $\mu(0) = 0$, then $(n(t))_{t \geq 0}$ is almost surely non-decreasing and the probability that it is constant from a certain time is clearly zero.
    Now assume that $\mu(0) > 0$.
    Note that $(n(t))_{t \geq 0}$ is a continuous-time Markov chain on $\N$ with no infinitely many jumps in finite time.
    Besides, $0$ is an absorbing state and it is accessible from any other state.
    The positive integers are then all transient states, \ie for any $i \in \N^*$, the set $\{t \geq 0 : n(t) = i\}$ is almost surely bounded\footnote{See \eg \cite[Section~3.4]{Norris1997} for a discussion about recurrence and transience for continuous-time Markov chains.}.
    We deduce that $n(t)$ tends almost surely to infinity on the event where $0$ is never reached, that is the event of survival.
    Thus, the process $(p^{n(t)})_{t \geq 0}$ converges almost surely to $0$ on the event of survival and to $1$ on the event of extinction.
    By dominated convergence, \eqref{eq:zero_one_law_expectation} converges to $q$ as $t$ goes to infinity.
\end{proof}

Before proving Theorem~\ref{th:phase_transition_additive_martingales}, let us mention an easier case.

\begin{proposition}\label{prop:phase_transition_easy_case}
    If $0 \leq \beta < 1$, then the martingale $(W_t(\beta))_{t \geq 0}$ is bounded in $L^2$.
\end{proposition}

\begin{proof}
    We compute
    \begin{align*}
        \Expec{W_t(\beta)^2} &= \e^{-(2 + \beta^2)t} \Expec{\sum_{u, v \in \cN(t)} \e^{\beta(X_u(t) + X_v(t))}} \\
        &= \e^{-(2 + \beta^2)t} \left(\Expec{\sum_{u \in \cN(t)} \e^{2 \beta X_u(t)}} + \Expec{\sum_{\substack{u, v \in \cN(t) \\ u \neq v}} \e^{\beta(X_u(t) + X_v(t))}}\right) \\
        &= \e^{-(2 + \beta^2)t} \left(\e^t \Expec{\e^{2 \beta B_t}} + K \e^{2t} \int_0^t \Expec{\e^{\beta(B^{1, s}_t + B^{2, s}_t)}} \e^{-s} \d{s}\right),
    \end{align*}
    where $K := \sum_{k \geq 0} \mu(k) k(k-1)$, according to the many-to-one formula \eqref{eq:many-to-one} and the many-to-two formula \eqref{eq:many-to-two}.
    Using that, for $0 \leq s \leq t$, the increments $B^{1, s}_t - B^{1, s}_s$, $B^{2, s}_t - B^{1, s}_s$ and $B^{1, s}_s$ are independent, we obtain
    \begin{equation}\label{eq:second_moment_of_additive_martingales}
        \Expec{W_t(\beta)^2} = \e^{-(1 - \beta^2)t} + K \int_0^t \e^{-(1 - \beta^2)s} \d{s} = O(1),
    \end{equation}
    that is Proposition~\ref{prop:phase_transition_easy_case}.
\end{proof}

The two following lemmas will be useful for the proof of Theorem~\ref{th:phase_transition_additive_martingales} and the rest of the thesis.

\begin{lemma}
    Recall that $\beta_c = \sqrt{2}$. For any $L > 0$, we have
    \begin{equation}\label{eq:prob_under_barrier}
        \P\bigl(\underbrace{\forall t \geq 0, \forall u \in \cN(t), X_u(t) < \beta_c t + L}_{=: A_L}\bigr) \geq 1 - \e^{-\beta_c L}.
    \end{equation}
\end{lemma}

\begin{proof}
    Let us consider the hitting time of the barrier $t \mapsto \beta_c t + L$
    \begin{equation*}
        \tau := \inf\left\{t \geq 0 : \exists u \in \cN(t), X_u(t) = \beta_c t + L\right\}
    \end{equation*}
    and denote by $s \wedge t$ the minimum of two numbers $s$ and $t$.
    The optional stopping theorem yields
    \begin{align*}
        1 = \Expec{W_{t \wedge \tau}(\beta_c)} &\geq \Expec{W_\tau(\beta_c) \mathds{1}_{\tau \leq t}} \\
        &\geq \Expec{\e^{-c(\beta_c) \tau} \e^{\beta_c(\beta_c \tau + L)} \mathds{1}_{\tau \leq t}} = \e^{\beta_c L} \Prob{\tau \leq t}.
    \end{align*}
    Letting $t$ go to infinity, it yields $\Prob{A_L^c} \leq \e^{-\beta_c L}$ and hence \eqref{eq:prob_under_barrier}.
\end{proof}

\begin{lemma}[\cite{Biggins1992}]\label{lem:biggins_lemma}
    If $X_1, \ldots, X_n$ are independent centered random variables on a probability space $(\Omega_0, \cF_0, \P_0)$, then for all $p \in [1, 2]$,
    \begin{equation}\label{eq:biggins_lemma}
        \E\left|\sum_{i=1}^n X_i\right|^p \leq 2^p \sum_{i = 1}^n \E\left|X_i\right|^p.
    \end{equation}
    More generally, consider $\cF_1 \subset \cF_0$ a $\sigma$-algebra and $N$ a $\cF_1$-measurable random variable taking its values in $\N$.
    If, conditionally on $\cF_1$, the random variables $X_1, \ldots, X_N$ are independent and centered, then for all $p \in [1, 2]$,
    \begin{equation}\label{eq:conditional_biggins_lemma}
        \condExpec{\left|\sum_{i =1}^N X_i\right|^p}{\cG} \leq 2^p \sum_{i = 1}^N \condExpec{|X_i|^p}{\cG}.
    \end{equation}
\end{lemma}

\begin{proof}
    We refer to Biggins \cite[Lemma~1]{Biggins1992} for a proof of \eqref{eq:biggins_lemma}.
    This is a consequence of Jensen's, Clarkson's and Minkowski's inequalities.
    Since each of these inequalities can all be generalized to conditional expectations, \eqref{eq:conditional_biggins_lemma} also holds.
\end{proof}

\begin{proof}[Proof of Theorem~\ref{th:phase_transition_additive_martingales}]
    \textbf{Supercritical case.} Assume $\beta \geq \beta_c$.
    We start by fixing $L > 0$ and showing that
    \begin{equation}\label{eq:limit_add_mart_barrier}
        \lim_{t \to \infty} \Expec{W_t(\beta) \mathds{1}_{A_L}} = 0,
    \end{equation}
    where $A_L$ is defined in \eqref{eq:prob_under_barrier}.
    Using the many-to-one formula \eqref{eq:many-to-one}, we obtain
    \begin{align*}
        \Expec{W_t(\beta) \mathds{1}_{A_L}} &\leq \Expec{\sum_{u \in \cN(t)} \e^{\beta X_u(t) - c(\beta) t} \mathds{1}_{\forall s \in [0, t], X_u(s) \leq \beta_c s + L}} \\
        &= \Expec{\e^{\beta B_t - \beta^2 t/2} \mathds{1}_{\forall s \in [0, t], B_s \leq \beta_c s + L}}.
    \end{align*}
    Classically, we recognize a Girsanov transform: under the probability measure $\hat{\P}$ defined by
    \begin{equation*}
        \frac{\d{\hat{\P}|_{\cF_t}}}{\d{\P|_{\cF_t}}} = \e^{\beta B_t - \beta^2 t/2},
    \end{equation*}
    the process $(B_t)_{t \geq 0}$ is a Brownian motion with drift $\beta$ and then $(B_t - \beta_c t)_{t \geq 0}$ is a Brownian motion with drift $\beta - \beta_c \geq 0$.
    Hence,
    \begin{equation*}
        \Expec{W_t(\beta) \mathds{1}_{A_L}} \leq \hat{\P}\left(\forall s \in [0, t], B_s - \beta_c s \leq L\right) \xrightarrow[t \to \infty]{} 0,
    \end{equation*}
    that is \eqref{eq:limit_add_mart_barrier}.

    Let us deduce that the additive martingale converges in probability to $0$.
    Fix $\varepsilon > 0$ and $\delta > 0$.
    We have
    \begin{equation*}
        \Prob{W_t(\beta) \geq \varepsilon} = \Prob{W_t(\beta) \mathds{1}_{A_L} \geq \varepsilon} + \Prob{W_t(\beta) \mathds{1}_{A_L^c} \geq \varepsilon}
    \end{equation*}
    By \eqref{eq:prob_under_barrier}, we can fix $L$ large enough so that the second term is bounded by $\delta$.
    By \eqref{eq:limit_add_mart_barrier}, we can fix $t_0 = t_0(L) > 0$ large enough so that, for any $t \geq t_0$, the first term is bounded by $\delta$.
    Hence the convergence in probability to $0$.
    Thus, the almost sure limit $W_\infty(\beta)$ is zero.
    
    For the subcritical case, we present two methods.
    The first one is to show that the additive martingale is uniformly integrable, thanks to a \emph{truncation} of the process, which is bounded in $L^2$.
    The second one is to show that the additive martingale is bounded in $L^p$ for some $p \in (1, 2]$, following arguments from Biggins \cite[Theorem~1]{Biggins1992} for the branching random walk.

    \textbf{Subcritical case, first method.} Assume that $0 \leq \beta < \beta_c$.
    Consider $L > 0$ and $\beta' > \beta$ whose value will be fixed later.
    We start by showing that the following quantity is close to $W_t(\beta)$ in $L^1$
    \begin{equation*}
        \tilde{W}_t(\beta) := \sum_{u \in \cN(t)} \e^{\beta X_u(t) - c(\beta)t} \mathds{1}_{\forall s \in [0, t], X_u(s) < \beta' s + L}.
    \end{equation*}
    Using the many-to-one formula \eqref{eq:many-to-one} and the same tilting as for the supercritical case, we obtain
    \begin{align}
        \E\left|W_t(\beta) - \tilde{W}_t(\beta)\right| &= \Prob{\exists s \in [0, t], B_s + \beta s \geq \beta' s + L} \notag \\
        &\leq \Prob{\sup_{s \geq 0} \left(B_s - (\beta' - \beta)s\right) \geq L} = \e^{-2(\beta' - \beta)L}. \label{eq:phase_transition_close_in_L1}
    \end{align}
    See \eg \cite[Equation~2.1.1.4]{BorodinSalminen2002} for the last equality.

    We now show that the second moment of $\tilde{W}_t(\beta)$ is bounded, uniformly in $t \geq 0$.
    Using the many-to-one formula \eqref{eq:many-to-one} and the many-to-two formula \eqref{eq:many-to-two}, we obtain
    \begin{multline}\label{eq:phase_transition_surcritical_bound}
        \Expec{\tilde{W}_t(\beta)^2} = \e^{-(1 + \beta^2)t} \Expec{\e^{2 \beta B_t} \mathds{1}_{\forall s \in [0, t], B_s < \beta' s + L}} \\
        + K \e^{-\beta^2t} \int_0^t \Expec{\e^{\beta B^{1,s}_t + \beta B^{2,s}_t} \mathds{1}_{\forall r \in [0, t], B^{1,s}_r < \beta' r + L} \mathds{1}_{\forall r \in [0, t], B^{2,s}_r < \beta' r + L}} \e^{-s} \d{s}.
    \end{multline}
    Let us control the second term in \eqref{eq:phase_transition_surcritical_bound}, the first one is easier.
    The expectation inside the integral is bounded by
    \begin{align*}
        \Expec{\e^{\beta B^{1,s}_t + \beta B^{2,s}_t} \mathds{1}_{B^{1,s}_s < \beta' s + L}} &= \Expec{\e^{\beta B_{t - s}}}^2 \Expec{\e^{2 \beta B_s} \mathds{1}_{B_s < \beta' s + L}} \\
        &= \e^{\beta^2 (t+s)} \Prob{B_s + 2 \beta s < \beta' s + L} \\
        &\leq \e^{\beta^2 (t+s)} \times \begin{cases}
            \e^{-((2\beta - \beta')s - L)^2/2s}/2 &\text{if } (2\beta - \beta')s > L, \\
            1 &\text{otherwise},
        \end{cases}
    \end{align*}
    where we have used the classic inequality $\Prob{N > x} \leq \e^{-x^2/2}/2$ for a standard Gaussian variable $N$ and any $x \geq 0$.
    Thus, if $\beta' < 2\beta$, there exist positive constants $t_0 = t_0(\beta, \beta', L)$, $C_0 = C_0(\beta, \beta', L, K)$, $C_1 = C_1(\beta, \beta', L, K)$ such that for any $t \geq t_0$, the second term in \eqref{eq:phase_transition_surcritical_bound} is bounded by
    \begin{equation*}
        C_0 + C_1 \e^{-\beta^2 t} \int_{t_0}^t \e^{\beta^2(t + s)} \e^{-(2\beta - \beta')^2 s/2} \e^{-s} \d{s} = C_0 + C_1 \int_{t_0}^t \e^{-(2\beta^2 + 2 - 4\beta\beta' + \beta'^2)s/2} \d{s}.
    \end{equation*}
    We now choose $\beta' \in (\beta, 2 \beta)$ such that the coefficient in the above exponent is negative, \ie $2\beta^2 + 2 - 4\beta\beta' + \beta'^2 > 0$.
    One can easily check that this is always possible, since $\beta < \sqrt{2}$.
    This shows that the second term in \eqref{eq:phase_transition_surcritical_bound} is bounded, uniformly in $t$.
    The first term can be controlled in the same way.
    Hence, there exists a constant $C = C(\beta, \beta', L, K) > 0$ such that $\Expec{\tilde{W}_t(\beta)^2} \leq C$, uniformly in $t$.

    We are now able to show that the additive martingale is uniformly integrable.
    Let $\varepsilon > 0$.
    Thanks to \eqref{eq:phase_transition_close_in_L1}, we can choose $L > 0$ such that $\Expec{W_t(\beta) - \tilde{W}_t(\beta)} \leq \varepsilon$, uniformly in $t$.
    We have, for any $M > 0$,
    \begin{align*}
        \Expec{W_t(\beta) \mathds{1}_{W_t(\beta) \geq M}} &\leq \Expec{\tilde{W}_t(\beta) \mathds{1}_{W_t(\beta) \geq M}} + \Expec{W_t(\beta) - \tilde{W}_t(\beta)} \\
        &\leq \Expec{\tilde{W}_t(\beta)^2}^{1/2} \Prob{W_t(\beta) \geq M} + \varepsilon \\
        &\leq \frac{C^{1/2}}{M} + \varepsilon.
    \end{align*}
    This implies the uniform integrability.
    Then, the limit $W_\infty(\beta)$ is also an $L^1$ limit and we have $W_t(\beta) = \condExpec{W_\infty(\beta)}{\cF_t}$, almost surely.
    By Lemma~\ref{lem:zero-one_law_additive_martingales}, this implies that $W_\infty(\beta)$ is almost surely positive on the event of survival.
    
    \textbf{Subcritical case, second method.} Using the many-to-two formula \eqref{eq:many-to-two}, we see that $W_s(\beta)$ is bounded in $L^2$, uniformly in $s \in [0, 1]$.
    Consequently, for all $p \in (1, 2]$,
    \begin{equation}\label{eq:phase_transition_Lp_bound}
        \sup_{s \in [0, 1]} \E\left|W_s(\beta) - 1\right|^p < \infty.
    \end{equation}
    In particular, we have $\E\left|W_1(\beta) - 1\right|^p < \infty$.
    Let us now control the quantity $\E\left|W_{k+1}(\beta) - W_k(\beta)\right|^p$, for an integer $k \geq 1$.
    To this end, note that for $0 \leq s \leq t$, we can rewrite
    \begin{align*}
        W_t(\beta) - W_s(\beta) &= \sum_{u \in \cN(s)} \e^{\beta X_u(s) - c(\beta)s} \left(\sum_{\substack{v \in \cN(t) \\ v \geq u}} \e^{\beta(X_v(t) - X_u(ks)) - c(\beta)(t - s)} - 1\right) \\
        &=: \sum_{u \in \cN(s)} \e^{\beta X_u(s) - c(\beta)s} Y^{(u)}_{t,s}.
    \end{align*}
    Conditionally on $\cF_s$, the random variables $Y^{(u)}_{t,s}$, $u \in \cN(s)$, are centered and independent, by construction of the branching Brownian motion.
    By Lemma~\ref{lem:biggins_lemma}, we obtain
    \begin{align}
        \condExpec{\left|W_t(\beta) - W_s(\beta)\right|^p}{\cF_s} &\leq 2^p \sum_{u \in \cN(s)} \left(\e^{\beta X_u(s) - c(\beta)s}\right)^p \condExpec{\left|Y^{(u)}_{t,s} \right|^p}{\cF_s} \notag \\
        &= 2^p W_s(p\beta) \e^{(p-1)(p\beta^2/\beta_c^2-1)s} \E\left|W_{t-s}(\beta) - 1\right|^p. \label{eq:phase_transition_cond_expec_F_s}
    \end{align}
    Since $\beta \in [0, \beta_c)$, there exists $p \in (1, 2]$ small enough so that the coefficient $(p-1)(p\beta^2/\beta_c^2-1)$ in the above exponent is negative.
    Thus, setting $s = k$, $t = k+1$ and taking the expectation of \eqref{eq:phase_transition_cond_expec_F_s}, the bound \eqref{eq:phase_transition_Lp_bound} implies that the series with general term $\E\left|W_{k+1}(\beta) - W_k(\beta)\right|^p$ is convergent.
    In particular, the sequence $(W_k(\beta))_{k \geq 0}$ is bounded in $L^p$.
    Thanks to \eqref{eq:phase_transition_Lp_bound} and \eqref{eq:phase_transition_cond_expec_F_s} applied with $s$ the floor function of $t$, we extend this result to the continuum: the whole process $(W_t(\beta))_{t \geq 0}$ is bounded in $L^p$.
    Then, this martingale is uniformly integrable and we conclude as in the first method.
\end{proof}

In the second method for the subcritical case, we have actually shown that for any $\beta \in [0, \beta_c)$, there exists $p > 1$ such that $(W_t(\beta))_{t \geq 0}$ is bounded in $L^p$.
More precisely, in view of the exponent in \eqref{eq:phase_transition_cond_expec_F_s}, we have the following result.

\begin{proposition}\label{prop:additive_martingales_bounded_in_Lp}
    Let $0 \leq \beta < \beta_c$.
    For any $p \in [1, 2]$ such that $p < \beta_c^2/\beta^2$, the martingale $(W_t(\beta))_{t \geq 0}$ is bounded in $L^p$.
    In particular, its almost sure limit $W_\infty(\beta)$ is in $L^p$ and is also an $L^p$ limit.
\end{proposition}

\section{Front and extremes of the branching Brownian motion}\label{sct:front_and_extremes}

Since this thesis is about the additive martingales of the branching Brownian motion, Theorem~\ref{th:phase_transition_additive_martingales} suggests to focus on the subcritical case $0 \leq \beta < \beta_c$.
In this regime, we will see in Section~\ref{sct:position_under_gibbs_measure} that only particles at distance $O(\sqrt{t})$ from $\beta t$ contribute to $W_t(\beta)$.
However, some related processes still have an asymptotic behavior governed by the \emph{front} of the branching Brownian motion, \ie the particles at distance $O(\sqrt{t})$ from the highest one, or by the \emph{extremes}, \ie the particles at distance $O(1)$ from the highest one (see \eg Theorem~\ref{th:fluctuations} and Theorem~\ref{th:renormalized_subcritical_overlap}).
This section is dedicated to the presentation of the needed results about the front.
Since it does not constitute the core of our subject, we present most of them without proof.
We give all the appropriate references for the interested reader.

\subsection{Critical and supercritical additive martingales}

Here, for both regimes $\beta = \beta_c$ and $\beta > \beta_c$, we give without proof the suitable rate function $\gamma(t)$ such that $\gamma(t)W_t(\beta)$ converges to a non-degenerate distribution.
These results will be necessary in Section~\ref{sct:gaussian_boundary_case} and Section~\ref{sct:extremal_regime}, where we study the fluctuations of the additive martingales around their almost sure limits.
Recall that the derivative martingales are defined in \eqref{eq:definition_of_Z_t} and that, almost surely, the critical one converges to a random variable $Z_\infty$ that is positive on the event of survival.

\begin{proposition}\label{prop:convergence_of_the_critical_additive_martingale}
    The process $(\sqrt{t} W_t(\beta_c))_{t \geq 0}$ converges in $\P$-probability to $\sqrt{2/\pi} Z_\infty$ as $t$ goes to infinity.
\end{proposition}

Proposition~\ref{prop:convergence_of_the_critical_additive_martingale} was first obtained by Aïdékon-Shi in \cite{AidekonShi2014} for the branching random walk.
In the case of the branching Brownian motion, it can be directly deduced from Maillard-Pain \cite[Proposition~2.2]{MaillardPain2019}.

\begin{proposition}\label{prop:convergence_of_the_supercritical_additive_martingales}
    Assume that $\mu = \delta_2$. For any $\beta > \beta_c$, the process $(t^{\frac{3 \beta}{2\sqrt{2}}} \e^{(1 - \beta/\beta_c)^2 t} W_t(\beta))_{t \geq 0}$ converges in distribution to $(C Z_\infty)^{\beta/\beta_c} S$ as $t$ goes to infinity, where $S$ is a non-degenerate $\beta_c/\beta$-stable random variable independent of $Z_\infty$ and $C$ is a positive constant.
\end{proposition}

Proposition~\ref{prop:convergence_of_the_supercritical_additive_martingales} was proved by Barral-Rhodes-Vargas in \cite[Theorem~1]{BarralRhodesVargas2018} for the branching random walk with more general offspring distributions.
As for the binary branching Brownian motion, it can be deduced from Bonnefont \cite[Proposition~2.5]{Bonnefont2022}.
We deduce tightness.
Indeed, let $\varepsilon > 0$.
The convergence in distribution to a finite random variable implies that there exists $M > 0$ and $t_0 \geq 0$ such that
\begin{equation*}
    \sup_{t \geq t_0} \Prob{t^{\frac{3}{2\sqrt{2}}} \e^{(1 - \beta/\beta_c)^2 t} W_t(\beta) > M} \leq \varepsilon.
\end{equation*}
Using Markov's inequality and the many-to-one formula \eqref{eq:many-to-one}, we see that we can assume $M$ to be large enough so that
\begin{equation*}
    \sup_{0 \leq t \leq t_0} \Prob{t^{\frac{3}{2\sqrt{2}}} \e^{(1 - \beta/\beta_c)^2 t} W_t(\beta) > M} \leq \varepsilon.
\end{equation*}

\begin{corollary}\label{cor:tightness_of_the_supercritical_additive_martingales}
    Assume that $\mu = \delta_2$. The process $(t^{\frac{3}{2\sqrt{2}}} \e^{(1 - \beta/\beta_c)^2 t} W_t(\beta))_{t \geq 0}$ is tight.
\end{corollary}

\subsection{Speed of the maximal displacement, first order}\label{sct:speed_of_the_extremal_displacement_1}

Let us denote by $M(t)$ the maximum position of the particles at time $t$
\begin{equation}\label{eq:definition_of_M(t)}
    M(t) := \max_{u \in \cN(t)} X_u(t).
\end{equation}
This process has been much studied (see \eg \cite{Bramson1978, Bramson1983, LalleySellke1987}) since McKean exhibited in \cite{McKean1975} its connection with the \emph{F-KPP equation}
\begin{equation}\label{eq:f-kpp_equation}
    \begin{cases}
        \partial_t u = \frac{1}{2} \partial_x^2 u + \sum_{k \geq 0} \mu(k) u^k - u & \text{on } \R_+ \times \R, \\
        u(0, \cdot) = u_0(\cdot) & \text{on } \R,
    \end{cases}
\end{equation}
also called \emph{KPP equation}, \emph{Fisher's equation} or \emph{Kolmogorov equation}.
This equation was first considered by Fisher in \cite{Fisher1937} and by Kolmogorov-Petrovsky-Piskunov in \cite{KolmogorovPetrovskyPiskunov1937}.
It belongs to the class of \emph{reaction–diffusion equations}, which naturally appear in many fields, as population genetics, chemical combustion theory and flame
propagation (see \eg \cite{AronsonWeinberger1975, FifeMcLeod1977}).
McKean noticed in \cite{McKean1975} that the distribution function $u(t, x) := \Prob{M(t) \leq x}$ is the unique solution of \eqref{eq:f-kpp_equation} with Heaviside initial condition $u_0 = \mathds{1}_{[0, \infty)}$.
Here, we give a first approximation of the maximal displacement $M(t)$ defined in \eqref{eq:definition_of_M(t)}.

\begin{proposition}
    Recall that $\beta_c := \sqrt{2}$. We have
    \begin{equation}\label{eq:domination_extremal_particle}
        \lim_{t \to \infty} M(t) - \beta_c t = -\infty \quad \P\text{-almost surely}.
    \end{equation}
    In addition, on the event of survival,
    \begin{equation}\label{eq:speed_extremal_particle}
        \lim_{t \to \infty} \frac{M(t)}{t} = \beta_c \quad \P\text{-almost surely}.
    \end{equation}
\end{proposition}

\begin{proof}
    By Theorem~\ref{th:phase_transition_additive_martingales}, the additive martingale $(W_t(\beta_c))_{t \geq 0}$ converges almost surely to $0$.
    But we have
    \begin{equation*}
        \e^{\sqrt{2} M(t) - 2t} \leq W_t(\beta_c).
    \end{equation*}
    Then, the quantity $M(t) - \beta_c t$ tends almost surely to $-\infty$ as $t$ goes to $\infty$.
    This implies that
    \begin{equation*}
        \limsup_{t \to \infty} \frac{M(t)}{t} \leq \beta_c \quad \P\text{-almost surely}.
    \end{equation*}
    It remains to show that, on the event of survival,
    \begin{equation}\label{eq:extremal_particle_1st_order_lower_bound}
        \liminf_{t \to \infty} \frac{M(t)}{t} \geq \beta_c \quad \P\text{-almost surely}.
    \end{equation}
    To this end, assume that there is no extinction and consider $t \geq 1$ as well as two parameters $\beta > \beta'$ in $[0, \beta_c)$.
    We start by approximating $W_t(\beta)$ with
    \begin{equation*}
        \tilde{W}_t(\beta) = \sum_{u \in \cN(t)} \e^{\beta X_u(t) - c(\beta)t} \mathds{1}_{\forall s \in [t-1, t], X_u(s) \geq \beta' s}.
    \end{equation*}
    Using the many-to-one formula \eqref{eq:many-to-one} and Girsanov's theorem in the same way as in the proof of Theorem~\ref{th:phase_transition_additive_martingales}, we obtain
    \begin{align*}
        \E\left|W_t(\beta) - \tilde{W}_t(\beta)\right| &= \Prob{\exists s \in [t-1, t], B_s + \beta s < \beta' s} \\
        &\leq \Prob{\inf_{s \in [0, t]} B_s < -(\beta - \beta')(t-1)} \leq \e^{-(\beta - \beta')^2 (t - 1)^2/2t}.
    \end{align*}
    By the Borel-Cantelli lemma, the sequence $(\tilde{W}_n(\beta))_{n \geq 1}$ converges almost surely to $W_\infty(\beta)$.
    This limit is almost surely positive since there is no extinction, by assumption.
    Thus, in view of the definition of $\tilde{W}_t(\beta)$, there exists a rank from which the sequence $(\mathds{1}_{\forall s \in [n-1, n], M(s) \geq \beta' s})_{n \geq 1}$ is positive.
    In particular,
    \begin{equation*}
        \liminf_{t \to \infty} \frac{M(t)}{t} \geq \beta' \quad \P\text{-almost surely}.
    \end{equation*}
    Since $\beta' \in [0, \beta_c)$ is arbitrary, this implies \eqref{eq:extremal_particle_1st_order_lower_bound}.
\end{proof}

\subsection{Speed of the maximal displacement, second order}\label{sct:speed_of_the_extremal_displacement_2}

In this section, we assume that $0$ is not an atom of the offspring distribution, \ie $\mu(0) = 0$.
We give without proof another approximation of $M(t)$, which is more accurate than \eqref{eq:speed_extremal_particle}.
It was obtained by Bramson in \cite{Bramson1978}.
Based on the classic paper of Kolmogorov-Petrovsky-Piscounov \cite{KolmogorovPetrovskyPiskunov1937}, the author showed that the maximal displacement is at distance $O(1)$ of the following trajectory
\begin{equation}\label{eq:definition_of_m(t)}
    m(t) := \sqrt{2}t - \frac{3}{2 \sqrt{2}} \log t.
\end{equation}

\begin{proposition}[\cite{Bramson1978}]\label{prop:speed_extremal_particle_2}
    For $\alpha \in (0, 1)$, let us denote by $q_\alpha(t)$ the $\alpha$-quantile of $M(t) - m(t)$, \ie the value such that
    \begin{equation*}
        \Prob{M(t) - m(t) \leq q_\alpha(t)} = \alpha.
    \end{equation*}
    Then, for any $\alpha \in (0, 1)$, we have $q_\alpha(t) = O(1)$ as $t$ goes to infinity.
\end{proposition}

Recall that the distribution function of $M(t)$ is solution of the F-KPP equation \eqref{eq:f-kpp_equation} with Heaviside initial condition.
Therefore, one can deduce from \cite{KolmogorovPetrovskyPiskunov1937} that there exists a non-degenerate random variable $W$ and a function $\gamma(t)$ asymptotically equivalent to $\sqrt{2}t$ such that $M(t) - \gamma(t)$ converges in distribution to $W$.
Bramson showed in \cite{Bramson1983} that we can take $\gamma(t) = m(t)$ as defined in \eqref{eq:definition_of_m(t)}.
Based on this work, Lalley-Sellke obtained in \cite[Theorem~1]{LalleySellke1987} the following representation of the limit distribution function.

\begin{theorem}[\cite{LalleySellke1987}]\label{th:convergence_of_M(t)-m(t)}
    As $t$ goes to infinity, $M(t) - m(t)$ converges in distribution to a random variable $W$ whose distribution function has the following representation
    \begin{equation*}
        \Prob{W \leq x} = \Expec{\exp\left(-C Z_\infty \e^{-\sqrt{2}x}\right)}, \quad x \in \R,
    \end{equation*}
    where $C > 0$ is a constant and $Z_\infty > 0$ is the almost sure limit of the critical derivative martingale defined in \eqref{eq:definition_of_Z_t}.
    In other words, $W$ has the same distribution as $(G + \log(C Z_\infty))/\sqrt{2}$, where $G$ follows a standard Gumbel distribution and is independent of $Z_\infty$.
\end{theorem}

With the same arguments as for Corollary~\ref{cor:tightness_of_the_supercritical_additive_martingales}, we deduce the following corollary.

\begin{corollary}\label{cor:tightness_of_M(t)-m(t)}
    The process $M(t) - m(t)$ is tight.
\end{corollary}

\subsection{Convergence of the extremal process}\label{sct:convergence_of_the_extremal_process}

As in Section~\ref{sct:speed_of_the_extremal_displacement_2}, we assume that the offspring distribution satisfies $\mu(0) = 0$.
Arguin-Bovier-Kistler \cite{ArguinBovierKistler2012} and Aïdékon-Berestycki-Brunet-Shi \cite{AidekonBerestyckiBrunetShi2013} independently established the convergence in distribution of the \emph{extremal process}, that is the point process associated with the branching Brownian motion seen from its tip
\begin{equation*}
    \cE_t := \sum_{u \in \cN(t)} \delta_{X_u(t) - m(t)}.
\end{equation*}

Let us state without proof a suitable consequence of their results.
Given a measure $\nu$ on some measurable space $X$ and a measurable function $f : X \to \R$ which is integrable or non-negative, we write
\begin{equation*}
    \nu(f) = \int_X f \d{\nu} = \int_X f(x) \Pd{\nu}{x}.
\end{equation*}
We denote by $\cM$ the space of Radon measures on $\R$.
We equip it with the \emph{vague topology}: a sequence $(\nu_n)_{n \geq 1}$ converges to $\nu_\infty$ in $\cM$ if and only if, for any continuous function $f : \R \to \R$ with compact support, the quantity $\nu_n(f)$ converges to $\nu_\infty(f)$ as $n$ goes to infinity.
With respect to this topology, the extremal process converges in distribution to a \emph{randomly shifted decorated Poisson point process}, \ie a Poisson point process $\cP$ where each atom is replaced by an independent copy of a certain \emph{decoration} $\cD$ shifted by the position of the atom, plus a shift depending on $Z_\infty$.

\begin{theorem}[\cite{ArguinBovierKistler2012, AidekonBerestyckiBrunetShi2013}]\label{th:convergence_of_the_extremal_process}
    The extremal process $(\cE_t)_{t \geq 0}$ converges in distribution in $\cM$ to the following point process
    \begin{equation}\label{eq:definition_of_the_limit_extremal_process}
        \cE_\infty := \sum_{i, j \geq 1} \delta_{p_i + \Delta_{ij} + \frac{1}{\sqrt{2}} \log\left(C Z_\infty\right)},
    \end{equation}
    where
    \begin{itemize}
        \item $\cP := \sum_{i} \delta_{p_i}$ is a Poisson point process on $\R$ with intensity measure $\sqrt{2} \e^{-\sqrt{2}x} \d{x}$, independent of $\cF_\infty$,
        \item the point processes $\cD_i := \sum_j \delta_{\Delta_{ij}}$, $i \geq 1$, are independent copies of an auxiliary one $\cD := \sum_j \delta_{\Delta_{j}}$ such that $\max \cD = 0$, and are independent of $\cP$ and $\cF_\infty$,
        \item $C$ is a positive constant.
    \end{itemize}
\end{theorem}

\begin{remark}\label{rem:limit_of_the_extremal_process}
    Since $\cP([0, \infty))$ is almost surely finite and $\max \cD = 0$, the atoms of $\cE_\infty =: \sum_k \delta_{\xi_k}$ almost surely admit an upper bound in $\R$.
    Moreover, since $\cE_\infty$ is a Radon measure, the set of its atoms has no accumulation point.
    We can then assume without loss of generality that they are arranged in non-increasing order $\xi_1 \geq \xi_2 \geq \ldots$
    Note that this also implies that $\cE_\infty([0, \infty))$ is almost surely finite.
\end{remark}

\begin{remark}
    Let $f : \R \to \R_+$ be continuous with compact support.
    The Laplace transform of $\cE_\infty$ is given by
    \begin{equation}\label{eq:laplace_transform_limit_extremal_process}
        \Expec{\e^{-\cE_\infty(f)}} = \Expec{\exp\left(-C Z_\infty \int \Expec{1 - \e^{-\int f(y + z) \Pd{\cD}{z}}} \sqrt{2} \e^{-\sqrt{2} y} \d{y}\right)}.
    \end{equation}
    Indeed, by conditioning on $\cP$ and $Z_\infty$, and by dominated convergence, we obtain
    \begin{align*}
        \Expec{\e^{-\cE_\infty(f)}} &= \Expec{\exp \sum_{i \geq 1} \log \condExpec{\e^{-\int f\left(p_i + z + \frac{1}{\sqrt{2}} \log(C Z_\infty)\right) \Pd{\cD}{z}}}{\cP, Z_\infty}} \\
        &= \E\condExpec{\exp\left(-\int \phi\left(y + \frac{1}{\sqrt{2}}\log(C Z_\infty)\right) \Pd{\cP}{y}\right)}{Z_\infty},
    \end{align*}
    where
    \begin{equation*}
        \phi(y) := -\log \Expec{\e^{-\int f(y + z) \Pd{\cD}{z}}}.
    \end{equation*}
    To conclude, it is then sufficient to apply the exponential formula for Poisson point processes (see \eg \cite[Section~0.5]{Bertoin1996}).
\end{remark}

The convergence to a limiting point process
was already implicit in a physics paper by Brunet-Derrida \cite{BrunetDerrida2009}.
The main result of \cite{ArguinBovierKistler2012, AidekonBerestyckiBrunetShi2013} is then a complete characterization of the decoration $\cD$.
In this thesis, we decide not to present their two constructions of $\cD$.
However, it will be useful to have a minimum of information about its structure.
The following proposition gives an approximation of the asymptotic mean number of its atoms at height $-x$ or above.
It is due to Cortines-Hartung-Louidor \cite[Proposition~1.5]{CortinesHartungLouidor2019}.
The authors only treated the binary branching $\mu = \delta_2$ but, in their own words, the extension to branching Brownian motion with a general offspring distribution requires only minor changes in their proofs.
We therefore believe that the following proposition and its corollary hold as soon as the offspring distribution satisfies the generic condition \eqref{eq:assumption_offspring_distribution}.

\begin{proposition}[\cite{CortinesHartungLouidor2019}]\label{prop:asymptotic_mean_number_decoration}
    Assume that $\mu = \delta_2$. There exists a constant $C_\star > 0$ such that $\E\cD([-x, 0]) \sim C_\star \e^{\beta_c x}$ as $x$ goes to infinity.
\end{proposition}

\begin{corollary}\label{cor:asymptotic_mean_number_decoration}
    Assume that $\mu = \delta_2$ and $\beta > \beta_c$.
    Then, $\Expec{\sum_j \e^{\beta \Delta_j}} < \infty$.
    Furthermore, the sum $\sum_k \e^{\beta \xi_k}$ is almost surely finite and satisfies
    \begin{equation}\label{eq:rewriting_sum_ppp}
        \sum_{k \geq 1} \e^{\beta \xi_k} \overset{d}{=} \left(C' Z_\infty\right)^{\beta/\beta_c} S,
    \end{equation}
    where $S := \sum_i \e^{\beta p_i}$ is a non-degenerate $\beta_c/\beta$-stable random variable independent of $Z_\infty$ and $C'$ is a positive constant.
\end{corollary}

\begin{remark}
    As a matter of fact, one can show that \eqref{eq:rewriting_sum_ppp} is also the limit distribution appearing in Proposition~\ref{prop:convergence_of_the_supercritical_additive_martingales} (see Remark~\ref{rem:bonnefont_on_extremal_process}).
\end{remark}

\begin{proof}[Proof of Corollary~\ref{cor:asymptotic_mean_number_decoration}]
    As done in \cite[Lemma~3.1]{Bonnefont2022}, we rewrite
    \begin{equation*}
        \Expec{\sum_{j \geq 1} \e^{\beta \Delta_j}} = \E\int_0^\infty \cD([-x, 0]) \beta \e^{-\beta x} \d{x}.
    \end{equation*}
    Applying the Fubini-Tonelli theorem and Proposition~\ref{prop:asymptotic_mean_number_decoration}, we see that the above quantity is finite.
    Now, let us rewrite
    \begin{equation}\label{eq:rewriting_sum_exp}
        \sum_{k \geq 1} \e^{\beta \xi_k} = (C Z_\infty)^{\beta/\beta_c} \sum_{i \geq 1} \e^{\beta(p_i + X_i)},
    \end{equation}
    where $X_i := \beta^{-1} \log \sum_j \e^{\beta \Delta_{ij}}$.
    The random variables $X_i$, $i \geq 1$ are independent and identically distributed.
    Furthermore, they are independent of $\cP$ and satisfy
    \begin{equation*}
        \Expec{\e^{\beta_c X_1}} = \Expec{\left(\sum_{j \geq 1} \e^{\beta \Delta_j}\right)^{\beta_c/\beta}} \leq \Expec{\sum_{j \geq 1} \e^{\beta \Delta_j}} < \infty.
    \end{equation*}
    By \cite[Proposition~8.7.a]{BolthausenSznitman2002},
    \begin{equation}\label{eq:rewriting_poisson_distribution}
        \sum_{i \geq 1} \delta_{p_i + X_i} \overset{d}{=} \sum_{i \geq 1} \delta_{p_i + \frac{1}{\beta_c} \log \Expec{\e^{\beta_c X_1}}}.
    \end{equation}
    In view of \eqref{eq:rewriting_sum_exp} and \eqref{eq:rewriting_poisson_distribution}, we have $\sum_k \e^{\beta \xi_k} < \infty$ almost surely if and only if $\sum_i \e^{\beta p_i} < \infty$ almost surely.
    By computing the expectations, we see that both the number of atoms of $\cP$ inside $[0, \infty)$ and the sum $\sum_i \e^{\beta p_i} \mathds{1}_{p_i < 0}$ are almost surely finite.
    Hence, $\sum_i \e^{\beta p_i} < \infty$ almost surely.
    Besides, since $Z_\infty$ is independent of $\cP$ and $\cD_i$, $i \geq 1$, \eqref{eq:rewriting_sum_exp} and \eqref{eq:rewriting_poisson_distribution} yield
    \begin{equation*}
        \sum_{k \geq 1} \e^{\beta \xi_k} \overset{d}{=} \left(C Z_\infty \Expec{\e^{\beta_c X_1}}\right)^{\beta/\beta_c} \sum_{i \geq 1} \e^{\beta p_i}.
    \end{equation*}
    
    It remains to show that $S := \sum_i e^{\beta p_i}$ is $\alpha$-stable with $\alpha := \beta_c/\beta$.
    Let us fix $\lambda > 0$ and compute the associated Laplace transform.
    Applying successively the exponential formula for Poisson point processes (see \eg \cite[Section~0.5]{Bertoin1996}) and the change of variable $y = \e^{\beta x}$, we obtain
    \begin{align*}
        \Expec{\exp\left(-\lambda \sum_{i \geq 1} \e^{\beta p_i}\right)} &= \exp\left(-\int_\R (1 - \e^{-\lambda \e^{\beta x}}) \sqrt{2} \e^{-\sqrt{2} x} \d{x}\right) \\
        &= \exp\left(-\alpha \int_0^\infty (1 - \e^{-\lambda y}) y^{-1-\alpha} \d{y}\right).
    \end{align*}
    By \cite[Theorem~24.11]{Sato1999}, this is the Laplace transform of $X_1$, where $(X_t)_{t \geq 0}$ is a subordinator with and drift $\gamma_0 = 0$ and Lévy measure
    \begin{equation*}
        \Pd{\nu}{x} = \alpha x^{-1-\alpha} \mathds{1}_{x > 0} \d{x}.
    \end{equation*}
    By \cite[Theorem~14.3]{Sato1999}, $S$ is $\alpha$-stable.
\end{proof}

The following Proposition~\ref{prop:first_moment_estimate_extremal_process} corresponds to \cite[Lemma~4.2]{CortinesHartungLouidor2019} by Cortines-Hartung-Louidor.
It concerns the first moment estimate of $\cE_t([x, \infty))$, where $x \in \R$.
Using the many-to-one formula \eqref{eq:many-to-one}, we see that this first moment blows up as $t$ goes to infinity.
Therefore, we must consider a truncation event.
This is the role of $\{M(t) - m(t) \leq y\}$, $y \geq x$.
As mentioned above, the authors only treated the binary branching but their arguments should be extended to our generic assumption \eqref{eq:assumption_offspring_distribution}.

\begin{proposition}[\cite{CortinesHartungLouidor2019}]\label{prop:first_moment_estimate_extremal_process}
    Assume that $\mu = \delta_2$. There exists a constant $C'' > 0$ such that for all $t \geq 0$ and $x \leq y$,
    \begin{equation}\label{eq:first_moment_estimate_extremal_process}
        \Expec{\cE_t([x, \infty)) \mathds{1}_{M(t) - m(t) \leq y}} \leq C'' \e^{-\beta_c x} (y - x + 1)(y^+ + 1)(\e^{-x^2/4t} + \e^{x/2}),
    \end{equation}
    where $y^+ := \max(y, 0)$.
\end{proposition}

\subsection{Joint convergence of the extremal process and the additive martingales}

In \cite{AidekonBerestyckiBrunetShi2013}, Aïdékon-Berestycki-Brunet-Shi showed that the pair $(\cE_t, Z_t)$ converges jointly in distribution to $(\cE_\infty, Z_\infty)$ as $t$ goes to infinity.
Thanks to Proposition~\ref{prop:convergence_extremal_process_cond_F_s} below, one can also obtain the convergence of the extremal process jointly with other functionals of the branching Brownian motion, \eg the additive martingales.
The latter will be useful in Section~\ref{sct:renormalized_subcritical_overlap_2} and is the subject of Corollary~\ref{cor:joint_convergence_extremal_process_additive_martingale}.
To our knowledge, Proposition~\ref{prop:convergence_extremal_process_cond_F_s} and Corollary~\ref{cor:joint_convergence_extremal_process_additive_martingale} represent original contributions.

\begin{lemma}\label{lem:critical_derivative_martingale_representation}
    For any $s \geq 0$,
    \begin{equation}\label{eq:critical_derivative_martingale_representation}
        Z_\infty = \sum_{u \in \cN(s)} \e^{\sqrt{2} X_u(s) - 2s} Z_\infty^{(u)} \quad \text{almost surely},
    \end{equation}
    where, conditionally on $\cF_s$, the random variables $Z_\infty^{(u)}$, $u \in \cN(s)$, are independent copies of $Z_\infty$.
\end{lemma}

\begin{proof}
    In the same vein as for \eqref{eq:definition_of_W_s^(u)}, we rewrite, for $0 \leq s \leq t$,
    \begin{align*}
        Z_t &= \sum_{u \in \cN(s)} \e^{\beta_c X_u(s) - c(\beta_c)s} \sum_{\substack{v \in \cN(t) \\ v \geq u}} (X_v(t) - \beta_c t) \e^{\beta_c(X_v(t) - X_u(s)) - c(\beta_c)(t - s)} \\
        &= \sum_{u \in \cN(s)} \e^{\beta_c X_u(s) - c(\beta_c)s} \left(Z_{t - s}^{(u)} + (X_u(s) - \beta_c s) W_{t - s}^{(u)}(\beta_c)\right),
    \end{align*}
    where, conditionally on $\cF_s$, the processes $(W_r^{(u)}(\beta_c))_{r \geq 0}$ and $(Z_r^{(u)})_{r \geq 0}$, $u \in \cN(s)$, are, respectively, the critical additive martingales and the critical derivative martingales of independent copies of the initial branching Brownian motion.
    Letting $t$ go to infinity, we obtain \eqref{eq:critical_derivative_martingale_representation}.
\end{proof}

\begin{proposition}\label{prop:convergence_extremal_process_cond_F_s}
    Assume that $\mu = \delta_2$ and let $f : \R \to \R_+$ be continuous with compact support.
    We have
    \begin{equation*}
        \lim_{s \to \infty} \lim_{t \to \infty} \condExpec{\e^{-\cE_t(f)}}{\cF_s} = \condExpec{\e^{-\cE_\infty(f)}}{\cF_\infty} \quad \text{almost surely}.
    \end{equation*}
\end{proposition}

\begin{proof}
    Let $0 \leq s \leq t$.
    We have
    \begin{align}
        \condExpec{\e^{-\cE_t(f)}}{\cF_s} &= \prod_{u \in \cN(s)} \condExpec{\exp\left(-\sum_{\substack{v \in \cN(t) \\ v \geq u}} f(X_v(t) - m(t))\right)}{\cF_s} \notag \\
        &= \prod_{u \in \cN(s)} \varphi_{t, s}(X_u(s)), \label{eq:joint_convergence_0}
    \end{align}
    where, for all $x \in \R$,
    \begin{align*}
        \varphi_{t,s}(x) &= \Expec{\exp\left(-\sum_{v \in \cN(t-s)} f(X_v(t-s) + x - m(t))\right)} \\
        &= \Expec{\e^{-\int f(y + x - m(t) + m(t-s)) \Pd{\cE_{t-s}}{y}}}.
    \end{align*}
    
    Let us show that for all $x \in \R$,
    \begin{equation}\label{eq:joint_convergence_1}
        \lim_{t \to \infty} \varphi_{t,s}(x) = \Expec{\e^{-\int f(y + x - \sqrt{2}s) \Pd{\cE_\infty}{y}}}.
    \end{equation}
    Since $y \mapsto f(y + x - \sqrt{2}s)$ is continuous with compact support, to prove \eqref{eq:joint_convergence_1}, it is sufficient to show that
    \begin{equation}\label{eq:joint_convergence_2}
        \lim_{t \to \infty} \left|\varphi_{t,s}(x) - \Expec{\e^{-\int f(y + x - \sqrt{2}s) \Pd{\cE_{t - s}}{y}}}\right| = 0.
    \end{equation}
    Let $\varepsilon > 0$. By uniform continuity of $f$, there exists $\delta \in (0, 1]$ such that $|f(y_1) - f(y_2)| \leq \varepsilon$ as soon as $|y_1 - y_2| \leq \delta$.
    Besides, since $m(t) - m(t-s)$ converges to $\sqrt{2}s$ as $t$ goes to infinity, we can fix $t_0 = t_0(s) \geq 0$ such that for all $t \geq t_0$, $|m(t) - m(t-s)| \leq \delta$.
    In particular, for all $t \geq t_0$ and all $y \in \R$,
    \begin{equation*}
        |f(y - m(t) + m(t-s)) - f(y - \sqrt{2}s)| \leq \varepsilon
    \end{equation*}
    and then,
    \begin{equation*}
        \int \left|f(y + x - m(t) + m(t-s)) - f(y + x - \sqrt{2}s)\right| \Pd{\cE_{t - s}}{y} \leq \varepsilon \int_{\Supp f + [-\delta, \delta] - x + \sqrt{2}s} \Pd{\cE_{t - s}}{y}.
    \end{equation*}
    Hence,
    \begin{multline}\label{eq:joint_convergence_3}
        \limsup_{t \to \infty} \E\int \left|f(y + x - m(t) + m(t-s)) - f(y + x - \sqrt{2}s)\right| \Pd{\cE_{t - s}}{y} \\
        \leq \varepsilon \sup_{t \geq 0} \E \int_{\Supp f + [-1, 1] - x + \sqrt{2}s} \Pd{\cE_{t - s}}{y}.
    \end{multline}
    Since the support of $f$ is compact, the bound \eqref{eq:first_moment_estimate_extremal_process} implies that the supremum in the right-hand side of \eqref{eq:joint_convergence_3} is finite.
    Since $\varepsilon$ is arbitrarily small, the left-hand side of \eqref{eq:joint_convergence_3} is equal to $0$.
    Hence \eqref{eq:joint_convergence_2} and \eqref{eq:joint_convergence_1}.
    
    Now, by \eqref{eq:joint_convergence_1} and \eqref{eq:laplace_transform_limit_extremal_process},
    \begin{equation}\label{eq:joint_convergence_4}
        \lim_{t \to \infty} \varphi_{t,s}(x) = \Expec{\exp\left(- C \e^{\sqrt{2}x - 2s} Z_\infty \int \Expec{1 - \e^{-\int f(y + z) \Pd{\cD}{z}}} \sqrt{2}\e^{-\sqrt{2}y} \d{y}\right)}.
    \end{equation}
    Coming back to \eqref{eq:joint_convergence_0}, the limit \eqref{eq:joint_convergence_4} yields
    \begin{align*}
        \lim_{t \to \infty} &\condExpec{\e^{-\cE_t(f)}}{\cF_s} \\
        &= \prod_{u \in \cN(s)} \condExpec{\exp\left(- C  \e^{\sqrt{2} X_u(s) - 2s} Z_\infty^{(u)} \int \Expec{1 - \e^{-\int f(y + z) \Pd{\cD}{z}}} \sqrt{2}\e^{-\sqrt{2}y} \d{y}\right)}{\cF_s} \\
        &= \condExpec{\exp\left(- C \sum_{u \in \cN(s)} \e^{\sqrt{2} X_u(s) - 2s} Z_\infty^{(u)} \int \Expec{1 - \e^{-\int f(y + z) \Pd{\cD}{z}}} \sqrt{2}\e^{-\sqrt{2}y} \d{y}\right)}{\cF_s},
    \end{align*}
    where the random variables $Z_\infty^{(u)}$, $u \in \cN(s)$, are defined in Lemma~\ref{lem:critical_derivative_martingale_representation}.
    By \eqref{eq:critical_derivative_martingale_representation}, we can rewrite
    \begin{equation*}
        \lim_{t \to \infty} \condExpec{\e^{-\cE_t(f)}}{\cF_s} = \condExpec{\exp\left(- C Z_\infty \int \Expec{1 - \e^{-\int f(y + z) \Pd{\cD}{z}}} \sqrt{2}\e^{-\sqrt{2}y} \d{y}\right)}{\cF_s}.
    \end{equation*}
    Letting $s$ go to infinity, we obtain
    \begin{equation*}
        \lim_{s \to \infty} \lim_{t \to \infty} \condExpec{\e^{-\cE_t(f)}}{\cF_s} = \condExpec{\exp\left(- C Z_\infty \int \Expec{1 - \e^{-\int f(y + z) \Pd{\cD}{z}}} \sqrt{2}\e^{-\sqrt{2}y} \d{y}\right)}{Z_\infty}.
    \end{equation*}
    As done for \eqref{eq:laplace_transform_limit_extremal_process}, it is straightforward to check that the above conditional expectation is equal to $\condExpec{\e^{-\cE_\infty(f)}}{\cF_\infty}$.
\end{proof}

\begin{corollary}\label{cor:joint_convergence_extremal_process_additive_martingale}
    Assume that $\mu = \delta_2$ and and let $\beta \geq 0$. The pair $(\cE_t, W_t(\beta))$ converges jointly in distribution to $(\cE_\infty, W_\infty(\beta))$ as $t$ goes to infinity.
\end{corollary}

\begin{proof}
    Fix $\lambda > 0$ and let $f : \R \to \R_+$ be continuous with compact support. For any $s \geq 0$,
    \begin{multline}
        \left|\Expec{\e^{-\cE_t(f) - \lambda W_t(\beta)}} - \Expec{\e^{-\cE_\infty(f) - \lambda W_\infty(\beta)}}\right| \leq \E\left|\e^{-\lambda W_t(\beta)} - \e^{-\lambda W_s(\beta)}\right| \\
        + \left|\Expec{\e^{-\cE_t(f) - \lambda W_s(\beta)}} - \Expec{\e^{-\cE_\infty(f) - \lambda W_\infty(\beta)}}\right|.
    \end{multline}
    Since the additive martingale converges almost surely and by Proposition~\ref{prop:convergence_extremal_process_cond_F_s}, the above quantity vanishes as $t$ and $s$ go to infinity.
\end{proof}

\section{Some convergence results}\label{sct:some_convergence_results}

From now on, we come back to our basic framework \eqref{eq:assumption_offspring_distribution}, unless otherwise stated.

\subsection{Position of particles under the Gibbs measure}\label{sct:position_under_gibbs_measure}

Let $\beta \in [0, \beta_c)$.
We are interested in the asymptotic behavior of the following quantity
\begin{equation*}
    W_t(\beta, f) = \sum_{u \in \cN(t)} \e^{\beta X_u(t) - c(\beta)t} f\left(\frac{X_u(t) - \beta t}{\sqrt{t}}\right),
\end{equation*}
where $f$ is a measurable function from $\R$ to $\R$.
In \cite{AsmussenKaplan1976a, AsmussenKaplan1976b}, Asmussen-Kaplan studied the counterpart of $W_t(\beta, f)$ for the branching random walk at $\beta = 0$.

\begin{theorem}\label{th:almost_sure_convergence_of_functional_1}
    Assume $f$ to be continuous and bounded.
    Then we have the following convergence, $\P$-almost surely and in $L^p$ for any $p \in [1, 2]$ such that $p < \beta_c^2/\beta^2$,
    \begin{equation}\label{eq:almost_sure_convergence_of_functional_1}
        \lim_{t \to \infty} W_t(\beta, f) = W_\infty(\beta) \int_\R f(x) \frac{\e^{-x^2/2}}{\sqrt{2 \pi}} \d{x} =: W_\infty(\beta, f).
    \end{equation}
\end{theorem}

\begin{remark}
    It should be possible to generalize \eqref{eq:almost_sure_convergence_of_functional_1} to functionals that depend on all the trajectories of the particles.
    More precisely, if $F : \cC([0, 1]) \to \R$ is a bounded continuous functional, then we expect the following limit to hold, $\P$-almost surely and in $L^p$ for any $p \in [1, 2]$ such that $p < \beta_c^2/\beta^2$,
    \begin{equation}\label{eq:almost_sure_convergence_of_functional_2}
        \lim_{t \to \infty} \sum_{u \in \cN(t)} \e^{\beta X_u(t) - c(\beta)t} F\left(\frac{X_u(st) - \beta st}{\sqrt{t}}, s \in [0, 1]\right) = W_\infty(\beta) \Expec{F\left(B_s, s \in [0, 1]\right)},
    \end{equation}
    where $(B_t)_{t \geq 0}$ is a standard Brownian motion under $\P$.
    The convergence in probability \eqref{eq:almost_sure_convergence_of_functional_2} is proved by Pain in \cite[Section~C]{Pain2018} for the branching random walk.
    The left-hand side of \eqref{eq:almost_sure_convergence_of_functional_2} has also been studied at $\beta = \beta_c$ by Madaule in \cite[Theorem~1.2]{Madaule2016} and at $\beta > \beta_c$ by Chen-Madaule-Mallein in \cite[Theorem~1.1]{ChenMadauleMallein2019}, both for the branching random walk.
\end{remark}

A direct consequence of Theorem~\ref{th:almost_sure_convergence_of_functional_1} is that the particles which contribute to the additive martingale $W_t(\beta)$ are those at distance $O(\sqrt{t})$ from $\beta t$.
We see this by taking $f$ the indicator function of an interval.

Regarding the almost sure convergence in Theorem~\ref{th:almost_sure_convergence_of_functional_1}, we will actually prove a stronger result.
Namely, $\P$-almost surely, for any bounded continuous function $f$, the limit \eqref{eq:almost_sure_convergence_of_functional_1} holds.
We decompose the proof into three lemmas. 

\begin{lemma}\label{lem:almost_sure_convergence_of_functional_1_cond_Fs}
    Assume $f$ to be continuous with compact support.
    Then, taking $t = t(s)$ such that $s^2 = o(t)$ when $s$ goes to infinity, we have
    \begin{equation*}
        \lim_{s \to \infty} \condExpec{W_t(\beta, f)}{\cF_s} = W_\infty(\beta, f) \quad \P\text{-almost surely}.
    \end{equation*}
\end{lemma}

\begin{proof}
    We start by dealing with another process, defined, for $0 \leq s \leq t$, by
    \begin{equation*}
        \tilde{W}_{t, s}(\beta, f) = \sum_{u \in \cN(t)} \e^{\beta X_u(t) - c(\beta)t} f\left(\frac{X_u(t) - \beta t}{\sqrt{t}}\right) \mathds{1}_{|X_u(s)| \leq \beta_c s}.
    \end{equation*}
    Note that, in the above formula, we can replace $\beta_c$ with any $\beta' > \beta_c$ without changing the arguments that follow.
    Using the branching property and the many-to-one formula \eqref{eq:many-to-one}, we rewrite its conditional expectation with respect to $\cF_s$ as
    \begin{multline}\label{eq:almost_sure_convergence_of_functional_1_approx}
        \sum_{u \in \cN(s)} \e^{\beta X_u(s) - c(\beta) s} \mathds{1}_{|X_u(s)| \leq \beta_c s} \condExpec{\sum_{\substack{v \in \cN(t) \\ v \geq u}} \e^{\beta(X_v(t) - X_u(s)) - c(\beta)(t - s)} f\left(\frac{X_v(t) - \beta t}{\sqrt{t}}\right)}{\cF_s} \\
        = \sum_{u \in \cN(s)} \e^{\beta X_u(s) - c(\beta) s} \mathds{1}_{|X_u(s)| \leq \beta_c s} \int_\R f\left(x \sqrt{1 - \frac{s}{t}} - \beta \frac{s}{\sqrt{t}} + \frac{X_u(s)}{\sqrt{t}}\right) \frac{\e^{-x^2/2}}{\sqrt{2 \pi}} \d{x}.
    \end{multline}
    Let $\varepsilon > 0$.
    Since $f$ is continuous with compact support, there exists a constant $C = C(f, \varepsilon) > 0$ such that for all $t = t(s) \geq C s^2$, for all $x \in \R$ and for all $|y| \leq \beta_c s$,
    \begin{equation*}
        \left|f\left(x \sqrt{1 - \frac{s}{t}} - \beta \frac{s}{\sqrt{t}} + \frac{y}{\sqrt{t}}\right) - f(x)\right| \leq \varepsilon.
    \end{equation*}
    In view of \eqref{eq:almost_sure_convergence_of_functional_1_approx}, this directly implies that, for all $t \geq C s^2$,
    \begin{equation}\label{eq:almost_sure_convergence_of_functional_1_approx_bound}
        \left|\condExpec{\tilde{W}_{t,s}(\beta, f)}{\cF_s} - \sum_{u \in \cN(s)} \e^{\beta X_u(s) - c(\beta) s} \mathds{1}_{|X_u(s)| \leq \beta_c s} \int_\R f(x) \frac{\e^{-x^2/2}}{\sqrt{2 \pi}} \d{x}\right| \leq \varepsilon W_s(\beta).
    \end{equation}
    By \eqref{eq:domination_extremal_particle}, the quantity $M(s) - \beta_c s$ tends almost surely to $-\infty$ as $s$ goes to infinity.
    In the same way, the sum of $\beta_c s$ and the lowest position of the particles at time $s$ tends almost surely to $\infty$ as $s$ goes to infinity.
    Consequently, the random variable $\mathds{1}_{\exists u \in \cN(s), |X_u(s)| > \beta_c s}$ converges almost surely to $0$.
    Thanks to \eqref{eq:almost_sure_convergence_of_functional_1_approx_bound} and the almost sure convergence of the additive martingale, taking $t = t(s)$ such that $s^2 = o(t)$, these considerations yield
    \begin{equation*}
        \lim_{s \to \infty} \condExpec{\tilde{W}_{t,s}(\beta, f)}{\cF_s} = W_\infty(\beta, f) \quad \P\text{-almost surely}.
    \end{equation*}
    Since, once again, the random variable $\mathds{1}_{\exists u \in \cN(s), |X_u(s)| > \beta_c s}$ converges almost surely to $0$ when $s$ goes to infinity, it is straightforward to deduce that
    \begin{equation*}
        \lim_{s \to \infty} \condExpec{W_t(\beta, f)}{\cF_s} = W_\infty(\beta, f) \quad \P\text{-almost surely},
    \end{equation*}
    which concludes the proof.
\end{proof}

\begin{lemma}\label{lem:almost_sure_convergence_of_functional_1_discretization}
    Assume $f$ to be continuous with compact support.
    Then, for any mesh span $\delta > 0$, the sequence $(W_{k \delta}(\beta, f))_{k \in \N}$ converges almost surely to $W_\infty(\beta, f)$.
\end{lemma}

\begin{proof}
    We show that $W_t(\beta, f)$ is asymptotically concentrated around its conditional expectation with respect to $\cF_s$, where $\log t \ll s \ll t$.
    To this end, let us define
    \begin{align*}
        \Delta_{t, s} &= W_t(\beta, f) - \condExpec{W_t(\beta, f)}{\cF_s} \\
        &= \sum_{u \in \cN(s)} \e^{\beta X_u(s) - c(\beta)s} \left(Y_{t, s}^{(u)} - \condExpec{Y_{t, s}^{(u)}}{\cF_s}\right),
    \end{align*}
    where, for all $u \in \cN(s)$,
    \begin{equation*}
        Y_{t, s}^{(u)} = \sum_{\substack{v \in \cN(t) \\ v \geq u}} \e^{\beta(X_v(t) - X_u(s)) - c(\beta)(t - s)} f\left(\frac{X_v(t) - \beta t}{\sqrt{t}}\right).
    \end{equation*}
    For any $p > 1$,
    \begin{equation*}
        \condExpec{\left|Y_{t, s}^{(u)}\right|^p}{\cF_s} \leq \|f\|_\infty^p \E|W_{t-s}(\beta)|^p.
    \end{equation*}
    Let us fix $p \in (1, 2]$ such that $p < \beta_c^2/\beta^2$.
    By Proposition~\ref{prop:additive_martingales_bounded_in_Lp}, $W_{s}(\beta)$ is bounded in $L^p$, uniformly in $s$.
    Besides, conditionally on $\cF_s$, the random variables $Y_{t, s}^{(u)}$, $u \in \cN(s)$, are independent.
    We then proceed as we did for \eqref{eq:phase_transition_cond_expec_F_s}: using Lemma~\ref{lem:biggins_lemma}, we obtain
    \begin{align*}
        \condExpec{|\Delta_{t, s}|^p}{\cF_s} &\leq 4^p \sum_{u \in \cN(s)} \left(\e^{\beta X_u(s) - c(\beta)s}\right)^p \condExpec{\left|Y_{t, s}^{(u)}\right|^p}{\cF_s} \\
        &\leq C W_{s}(p \beta) \e^{-(p-1)(1 - p\beta^2/\beta_c^2)s},
    \end{align*}
    for some constant $C = C(\beta, p, f) > 0$. Taking \eg $t = k\delta$ and $s = k^\alpha$ with $0 < \alpha < 1/2$, we deduce from the Borel-Cantelli lemma that $\Delta_{t, s}$ converges almost surely to $0$ as $k$ goes to infinity.
    By Lemma~\ref{lem:almost_sure_convergence_of_functional_1_cond_Fs}, since $s^2 = o(t)$, this implies that $W_{k\delta}(\beta, f)$ converges almost surely to $W_\infty(\beta, f)$.
\end{proof}

The next lemma is useful to straighten the convergence in discrete time into a convergence in continuous time.
Obviously, we do not need such a result to prove the counterpart of Theorem~\ref{th:almost_sure_convergence_of_functional_1} for the branching random walk.

\begin{lemma}\label{lem:almost_sure_convergence_of_functional_1_brutal_bounds}
    Assume furthermore $f$ to be continuously differentiable with compact support and define, for $0 \leq s \leq t$,
    \begin{equation*}
        \hat{W}_{t, s}(\beta, f) = \sum_{u \in \cN(t)} \e^{\beta X_u(t) - c(\beta)t} f\left(\frac{X_u(s) - \beta s}{\sqrt{s}}\right).
    \end{equation*}
    Then, for any $\varepsilon > 0$, there exists $\delta > 0$ such that, $\P$-almost surely,
    \begin{align}
        &\limsup_{\substack{s \to \infty \\ s \in \delta \N}} \sup_{t \in [s, s + \delta]} \left|W_t(\beta, f) - \hat{W}_{t, s}(\beta, f)\right| \leq \varepsilon \|f'\|_\infty W_\infty(\beta), \label{eq:almost_sure_convergence_of_functional_1_brutal_bound_1} \\
        &\limsup_{\substack{s \to \infty \\ s \in \delta \N}} \sup_{t \in [s, s + \delta]} \left|\hat{W}_{t, s}(\beta, f) - W_s(\beta, f)\right| \leq \varepsilon \|f\|_\infty W_\infty(\beta), \label{eq:almost_sure_convergence_of_functional_1_brutal_bound_2}
    \end{align}
    where $\|\cdot\|_\infty$ denotes the uniform norm.
\end{lemma}

\begin{proof}
    We first prove \eqref{eq:almost_sure_convergence_of_functional_1_brutal_bound_1}.
    Applying the mean value inequality to $f$, we obtain
    \begin{equation*}
        \left|W_t(\beta, f) - \hat{W}_{t, s}(\beta, f)\right| \leq \| f' \|_\infty \sum_{u \in \cN(t)} \e^{\beta X_u(t) - c(\beta)t} \frac{|X_u(t) - X_u(s)| + \beta |t - s|}{\sqrt{s}}.
    \end{equation*}
    Then, for any mesh span $\delta > 0$, we have the following bound,
    \begin{multline}\label{eq:almost_sure_convergence_of_functional_1_brutal_bound_1_sup}
        \sup_{t \in [s, s + \delta]} \left|W_t(\beta, f) - \hat{W}_{t, s}(\beta, f)\right|
        \\ \leq \|f'\|_\infty \left(\sup_{t \in [s, s + \delta]} W_t(\beta)\right) \left(\sup_{u \in \cN([s, s + \delta])} \sup_{t \in [s, s + \delta]} \frac{|\tilde{X}_u(t) - X_u(s)|}{\sqrt{s}}\right) + o(1),
    \end{multline}
    where $\cN([s, s + \delta])$ is the set of particles that have been alive between times $s$ and $s + \delta$, and $\tilde{X}_u$ is the trajectory of $u$ extended after its death by a Brownian motion independent of all the rest.
    Applying the union bound, the many-to-one formula \eqref{eq:many-to-one} and Lemma~\ref{lem:particles_alive_in_an_interval}, we obtain
    \begin{align*}
        \Prob{\sup_{u \in \cN([s, s + \delta])} \sup_{t \in [s, s+\delta]} \frac{|\tilde{X}_u(t) - X_u(s)|}{\sqrt{s}} > \varepsilon} &\leq C \e^s \Prob{\sup_{t \in [s, s+\delta]} \frac{|B_t - B_s|}{\sqrt{s}} > \varepsilon} \\
        &\leq C \exp(s - \varepsilon^2 s/2 \delta),
    \end{align*}
    for some constant $C = C(\mu, \delta)$.
    Then, choosing $\delta > 0$ such that $1 - \varepsilon^2/2\delta < 0$, the Borel-Cantelli lemma yields
    \begin{equation*}
        \limsup_{\substack{s \to \infty \\ s \in \delta \N}} \sup_{u \in \cN([s, s + \delta])} \sup_{t \in [s, s + \delta]} \frac{|\tilde{X}_u(t) - X_u(s)|}{\sqrt{s}} \leq \varepsilon \quad \P\text{-almost surely},
    \end{equation*}
    Coming back to \eqref{eq:almost_sure_convergence_of_functional_1_brutal_bound_1_sup}, this yields \eqref{eq:almost_sure_convergence_of_functional_1_brutal_bound_1}.
    
    We now prove \eqref{eq:almost_sure_convergence_of_functional_1_brutal_bound_2}, using the following bound
    \begin{align}
        \sup_{t \in [s, s + \delta]} \left|\hat{W}_{t, s}(\beta, f) - W_s(\beta, f)\right| &\leq \|f\|_\infty \sum_{u \in \cN(s)} \e^{\beta X_u(s) - c(\beta)s} \underbrace{\sup_{t \in [s, s + \delta]} \left|W_{t-s}^{(u)}(\beta) - 1\right|}_{=: S_{s, \delta}^{(u)}}, \label{eq:almost_sure_convergence_of_functional_1_brutal_bound_2_sup}
    \end{align}
    where $W_{t-s}^{(u)}(\beta)$ is defined in \eqref{eq:definition_of_W_s^(u)}.
    By Doob's maximal inequality, for all $u \in \cN(s)$,
    \begin{equation*}
        \condExpec{\left(S_{s,\delta}^{(u)}\right)^p}{\cF_s} \leq \left(\frac{p}{p-1}\right)^p \E\left|W_\delta(\beta) - 1\right|^p.
    \end{equation*}
    By dominated convergence, we can choose $\delta > 0$ small enough so that the above quantity is bounded by $\varepsilon^p$.
    Hence, using Hölder's inequality,
    \begin{equation}\label{eq:almost_sure_convergence_of_functional_1_bound_for_S}
        \limsup_{\substack{s \to \infty \\ s \in \delta \N}} \sum_{u \in \cN(s)} \e^{\beta X_u(s) - c(\beta)s} \condExpec{S_{s,\delta}^{(u)}}{\cF_s} \leq \varepsilon W_\infty(\beta) \quad \P\text{-almost surely}.
    \end{equation}
    Now, conditionally on $\cF_s$, the random variables $S_{s, \delta}^{(u)}$, $u \in \cN(s)$, are independent.
    We can then apply Lemma~\ref{lem:biggins_lemma} in the same way we did in \eqref{eq:phase_transition_cond_expec_F_s}. This yields
    \begin{equation*}
        \condExpec{\left|\sum_{u \in \cN(s)} \e^{\beta X_u(s) - c(\beta)s} \left(S_{s, \delta}^{(u)} - \condExpec{S_{s, \delta}^{(u)}}{\cF_s}\right)\right|^p}{\cF_s} \leq C \e^{(p-1)(p\beta^2/\beta_c^2 - 1)s} W_s(p \beta),
    \end{equation*}
    for any $p \in [1, 2]$, where $C = C(p, \delta) > 0$.
    Choosing $p > 1$ small enough, the Borel-Cantelli lemma yields
    \begin{equation}\label{eq:almost_sure_convergence_of_functional_1_bound_for_T}
        \lim_{\substack{s \to \infty \\ s \in \delta \N}} \sum_{u \in \cN(s)} \e^{\beta X_u(s) - c(\beta)s} \left(S_{s, \delta}^{(u)} - \condExpec{S_{s, \delta}^{(u)}}{\cF_s}\right) = 0 \quad \P\text{-almost surely}.
    \end{equation}
    Combining \eqref{eq:almost_sure_convergence_of_functional_1_brutal_bound_2_sup}, \eqref{eq:almost_sure_convergence_of_functional_1_bound_for_S} and \eqref{eq:almost_sure_convergence_of_functional_1_bound_for_T}, we obtain \eqref{eq:almost_sure_convergence_of_functional_1_brutal_bound_2}.
\end{proof}

\begin{proof}[Proof of Theorem~\ref{th:almost_sure_convergence_of_functional_1}]
    Let us denote by $\cC_c(\R)$ the set of continuous functions from $\R$ to $\R$ with compact support and by $\cC_c^1(\R)$ its intersection with the set of continuously differentiable functions.
    By Lemma~\ref{lem:almost_sure_convergence_of_functional_1_discretization} and Lemma~\ref{lem:almost_sure_convergence_of_functional_1_brutal_bounds}, if $f \in \cC_c^1(\R)$, then $(W_t(\beta, f))_{t \geq 0}$ converges almost surely to $W_\infty(\beta, f)$.
    Since $\cC_c^1(\R)$ is dense in $\cC_c(\R)$ for the uniform norm, the almost sure convergence still holds for any $f \in \cC_c(\R)$.
    Now consider a countable set $\{f_k\}$ which is dense in $\cC_c(\R)$.
    The event $\Omega'$ on which $(W_t(\beta))_{t \geq 0}$ converges to $W_\infty(\beta)$ and $(W_t(\beta, f_k))_{t \geq 0}$ converges to $W_\infty(\beta, f_k)$ for every $k$ has probability $1$.
    
    Let us fix a realization $\omega \in \Omega'$.
    Then, for any $f \in \cC_c(\R)$, the process $W_t(\beta, f) = W_t(\beta, f)(\omega)$ converges to $W_\infty(\beta, f) = W_\infty(\beta, f)(\omega)$.
    In terms of measures, we can define $\mu_t$ and $\mu_\infty$ the measures on $\R$ which associate to any bounded continuous function $f$ the quantities $W_t(\beta, f)$ and $W_\infty(\beta, f)$, respectively.
    We have proved that, for any $f \in \cC_c(\R)$, the quantity $\mu_t(f)$ converges to $\mu_\infty(f)$ as $t$ goes to infinity.
    Since these measures have finite mass, this convergence still holds for any bounded continuous functions.
    Hence the almost sure convergence stated in Theorem~\ref{th:almost_sure_convergence_of_functional_1}.
    
    Now, let $p \in (1, 2]$ be such that $p < \beta_c^2/\beta^2$.
    For any $t \geq 0$,
    \begin{equation}\label{eq:almost_sure_convergence_of_functional_1_domination}
        |W_t(\beta, f)| \leq \|f\|_\infty \sup_{s \geq 0} W_s(\beta).
    \end{equation}
    By Doob's maximal inequality and Proposition~\ref{prop:additive_martingales_bounded_in_Lp}, the right-hand side of \eqref{eq:almost_sure_convergence_of_functional_1_domination} is in $L^p$.
    By dominated convergence, we deduce the convergence of $W_t(\beta, f)$ in $L^p$ from the almost sure one.
\end{proof}

\begin{corollary}
    Assume $f$ to be integrable with respect to the measure $\Pd{\gamma}{x} := \e^{-x^2/2}/\sqrt{2 \pi} \d{x}$, but not necessarily bounded.
    Then, we have the following convergence in $L^1$
    \begin{equation}\label{eq:L1_convergence_of_functional_1}
        \lim_{t \to \infty} W_t(\beta, f) = W_\infty(\beta) \int_\R f(x) \Pd{\gamma}{x}.
    \end{equation}
\end{corollary}

\begin{proof}
    Fix $\varepsilon > 0$.
    There exists a bounded continuous function $f_\varepsilon$ such that the $L^1$ norm of $|f - f_\varepsilon|$ with respect to $\gamma$ is bounded by $\varepsilon$.
    We have
    \begin{align*}
        \E\left|W_t(\beta, f) - W_\infty(\beta, f)\right| &\leq \E \left|W_t(\beta, f) - W_t(\beta, f_\varepsilon)\right| + \E\left|W_t(\beta, f_\varepsilon) - W_\infty(\beta, f_\varepsilon)\right| + \varepsilon \\
        &\leq 2\varepsilon + \E\left|W_t(\beta, f_\varepsilon) - W_\infty(\beta, f_\varepsilon)\right|,
    \end{align*}
    where the last bound is obtained thanks to the many-to-one formula \eqref{eq:many-to-one}.
    We conclude by applying Theorem~\ref{th:almost_sure_convergence_of_functional_1} and by letting $\varepsilon$ go to $0$.
\end{proof}

\subsection{Growth rates in the branching Brownian motion}\label{sct:growth_rates}

Let $0 \leq \beta < \beta_c$.
The following statement tells us how the number of particles at distance $O(1)$ from $\beta t$ grows with $t$.
As in Theorem~\ref{th:almost_sure_convergence_of_functional_1}, it turns out that $W_\infty(\beta)$ describes this growth.
In \cite[Theorem~2]{AsmussenKaplan1976b}, Asmussen-Kaplan obtained the counterpart of Theorem~\ref{th:growth_rates} for the branching random walk at $\beta = 0$.
Later, Biggins generalized their result to any $\beta \geq 0$ in \cite[Theorem~B]{Biggins1979}.
The case of the multidimensional branching random walk was treated by Uchiyama in \cite{Uchiyama1982}.

\begin{theorem}\label{th:growth_rates}
    Let $f : \R \to \R$ be continuous and bounded.
    Then, we have the following convergence, almost surely and in $L^p$ for any $p \in [1, 2]$ such that $p < \beta_c^2/\beta^2$,
    \begin{equation}\label{eq:growth_rates}
        \lim_{t \to \infty} \underbrace{\sqrt{t} \e^{-(1 - \beta^2/2)t} \sum_{u \in \cN(t)} f(X_u(t) - \beta t)}_{=: V_t} = \underbrace{W_\infty(\beta) \int_\R f(x) \frac{\e^{-\beta x}}{\sqrt{2 \pi}} \d{x}}_{=: V_\infty}.
    \end{equation}
\end{theorem}

\begin{corollary}\label{cor:growth_rates}
    The convergence \eqref{eq:growth_rates} actually holds in $L^1$ as soon as $f$ is integrable with respect to the measure $\e^{-\beta x} \d{x}$.
\end{corollary}

Below, we introduce three lemmas which are very similar to Lemma~\ref{lem:almost_sure_convergence_of_functional_1_cond_Fs}, Lemma~\ref{lem:almost_sure_convergence_of_functional_1_discretization}, Lemma~\ref{lem:almost_sure_convergence_of_functional_1_brutal_bounds}.
The proofs of Theorem~\ref{th:growth_rates} and Corollary~\ref{cor:growth_rates} then follow from the arguments of Section~\ref{sct:position_under_gibbs_measure}.

\begin{lemma}\label{lem:almost_sure_convergence_growth_rates_cond_Fs}
    Assume $f : \R \to \R$ to be integrable with respect to the measure $\e^{-\beta x} \d{x}$.
    Then, taking $t = t(s)$ such that $s^2 = o(t)$ as $t$ goes to infinity, we have
    \begin{equation*}
        \lim_{s \to \infty} \condExpec{V_t}{\cF_s} = V_\infty \quad \P\text{-almost surely}.
    \end{equation*}
\end{lemma}

\begin{proof}
    Using the Markov property of the branching Brownian motion and the many-to-one formula \eqref{eq:many-to-one}, we obtain
    \begin{equation}\label{eq:growth_rates_cond_Fs_computation}
        \condExpec{V_t}{\cF_s} = \sqrt{\frac{t}{t-s}} \sum_{u \in \cN(s)} \e^{\beta X_u(s) - c(\beta)s} \int_\R f(x) \frac{\e^{-\beta x}}{\sqrt{2 \pi}} \e^{-(X_u(s) - x - \beta s)^2/2(t-s)} \d{x}.
    \end{equation}
    Since $\sup_{u \in \cN(s)} |X_u(s)|$ is almost surely equivalent to $\sqrt{2}s$ and $s^2 = o(t)$, for all $x \in \R$,
    \begin{equation*}
        \lim_{s \to \infty} \sup_{u \in \cN(s)} \left|1 - \e^{-(X_u(s) - x - \beta s)^2/2(t-s)}\right| = 0 \quad \P\text{-almost surely}.
    \end{equation*}
    By dominated convergence, \eqref{eq:growth_rates_cond_Fs_computation} converges almost surely to $V_\infty$ as $s$ goes to infinity.
\end{proof}

Similarly to Lemma~\ref{lem:almost_sure_convergence_of_functional_1_discretization}, we can prove that for any mesh span $\delta > 0$ and for $f$ continuous with compact support, the sequence $(V_{k \delta})_{k \in \N}$ converges almost surely to $V_\infty$.
However, contrary to what we obtained in Lemma~\ref{lem:almost_sure_convergence_of_functional_1_brutal_bounds}, we will not have a good control of $\left|V_t - V_s\right|$, uniformly in $t \in [s, s+\delta]$.
This comes from the fact that $|f(X_u(t) - \beta t) - f(X_u(s) - \beta s)|$ is not small enough in general.
We can overcome this obstacle by using an appropriate partition of the time $[0, \infty)$. 

For instance, define $t_{i,j} = i + j/i^{10}$ for all integers $i$ and $j$ such that $i \geq 1$ and $0 \leq j \leq i^{10}$.
In order to simplify the notations, define $(t_k)_{k \geq 1}$ the sequence of the consecutive terms of $\{t_{i, j}\}$.

\begin{lemma}
    Assume $f$ to be continuous with compact support.
    Then $(V_{t_k})_{k \geq 1}$ converges almost surely to $V_\infty$.
\end{lemma}

\begin{proof}
    Let $L > 0$ be such that the support of $f$ is included in $[-L, L]$ and fix $p \in (1, 2]$ such that $p < \beta_c^2/\beta^2$.
    For all $t \geq s \geq 0$,
    \begin{align}
        \condExpec{\left|V_t - \condExpec{V_t}{\cF_s}\right|^p}{\cF_s} &\leq 2^p t^{p/2} \e^{-p(1 - \beta^2/2)t} \sum_{u \in \cN(s)} \condExpec{\left| \sum_{\substack{v \in \cN(t) \\ v \geq u}} f(X_v(t) - \beta t) \right|^p}{\cF_s} \label{eq:almost_sure_convergence_of_functional_2_discretization_1} \\
        &\leq C t^{p/2} \e^{-(p - 1)(1 - p\beta^2/\beta_c^2)s} W_s(p \beta) \Expec{W_{t - s}(\beta)^p}, \label{eq:almost_sure_convergence_of_functional_2_discretization_2}
    \end{align}
    for some constant $C = C(\beta, p, f, L) > 0$.
    The bound \eqref{eq:almost_sure_convergence_of_functional_2_discretization_1} is obtained thanks to the same arguments as in the proof of Lemma~\ref{lem:almost_sure_convergence_of_functional_1_discretization}.
    As for \eqref{eq:almost_sure_convergence_of_functional_2_discretization_2}, we have used the inequalities
    \begin{equation*}
        f(x) \leq \|f\|_\infty \mathds{1}_{x \geq -L} \leq \|f\|_\infty \e^{\beta (x + L)}, \quad x \in \R,
    \end{equation*}
    in order to express the bound in terms of the additive martingale.
    Recall that $W_t(\beta)$ is bounded in $L^p$, uniformly in $t \geq 0$, by Proposition~\ref{prop:additive_martingales_bounded_in_Lp}.
    In the same vein as for Lemma~\ref{lem:almost_sure_convergence_of_functional_1_discretization}, we choose $t = t_k$ and $s = t_k^\alpha$, with $0 < \alpha < 1/2$.
    This enables us to conclude thanks to the Borel-Cantelli lemma and Lemma~\ref{lem:almost_sure_convergence_growth_rates_cond_Fs}, since $s^2 = o(t)$.
\end{proof}

\begin{lemma}\label{lem:almost_sure_convergence_of_functional_2_brutal_bounds}
    Assume furthermore $f$ to be continuously differentiable with compact support and define, for $0 \leq s \leq t$,
    \begin{equation*}
        \hat{V}_{t, s} = \sqrt{t} \e^{-(1 - \beta^2/2)t} \sum_{u \in \cN(t)} f(X_u(s) - \beta s).
    \end{equation*}
    Then, almost surely,
    \begin{align}
        &\lim_{k \to \infty} \sup_{t \in [t_k, t_{k+1}]} \left|V_t - \hat{V}_{t, t_k}\right| = 0, \label{eq:almost_sure_convergence_of_functional_2_brutal_bound_1} \\
        &\lim_{k \to \infty} \sup_{t \in [t_k, t_{k+1}]} \left|\hat{V}_{t, t_k} - V_{t_k}\right| = 0. \label{eq:almost_sure_convergence_of_functional_2_brutal_bound_2}
    \end{align}
\end{lemma}

The proof of Lemma~\ref{lem:almost_sure_convergence_of_functional_2_brutal_bounds} is similar to the one of Lemma~\ref{lem:almost_sure_convergence_of_functional_1_brutal_bounds}.
Therefore, we gloss over the details of the arguments.

\begin{proof}
    For all $t \geq s \geq 0$,
    \begin{multline}\label{eq:almost_sure_convergence_of_functional_2_mean_value}
        \left|V_t - \hat{V}_{t, s}\right|
        \leq \|f'\|_\infty \sqrt{t} \e^{-(1 - \beta^2/2)t} \sum_{u \in \cN(t)} \left(\mathds{1}_{X_u(t) - \beta t \in \Supp f} + \mathds{1}_{X_u(s) - \beta s \in \Supp f}\right) \\
        \times \left(|X_u(t) - X_u(s)| + \beta|t - s|\right).
    \end{multline}
    Now, recall that $\{t_k\} = \{t_{i, j}\}$ with $t_{i, j} := i + j/i^{10}$ and let $i \geq 1$, $0 \leq j \leq i^{10}-1$.
    Define $\varepsilon_i = 1/i$.
    This choice ensures, first, that $\varepsilon_i = o(1/\sqrt{t_{i,j}})$ and, second, that $\varepsilon_i$ is large enough in comparison with $\sqrt{t_{i, j+1} - t_{i, j}}$ so that we can apply the Borel-Cantelli lemma.
    In the end, this yields, for $i$ large enough and for $0 \leq j \leq i^{10}-1$,
    \begin{equation*}
        \sup_{u \in \cN([t_{i, j}, t_{i, j+1}])} \sup_{t \in [t_{i, j}, t_{i, j+1}]} |\tilde{X}_u(t) - X_u(t_{i, j})| \leq \varepsilon_i,
    \end{equation*}
    where $\tilde{X}_u$ is the trajectory of $u$ extended after its death by a Brownian motion independent of all the rest.
    In particular, there exists $L > 0$ such that, almost surely, for $i$ large enough, for $0 \leq j \leq i^{10}-1$ and $t_{i, j} \leq t \leq t_{i, j+1}$,
    \begin{equation*}
        \mathds{1}_{X_u(t) - \beta t \in \Supp f} + \mathds{1}_{X_u(t_{i, j}) - \beta t_{i, j} \in \Supp f} \leq 2 \mathds{1}_{X_u(t) - \beta t > -L} \leq 2 \e^{\beta X_u(t) - \beta^2 t + \beta L}.
    \end{equation*}
    Then, \eqref{eq:almost_sure_convergence_of_functional_2_mean_value} becomes
    \begin{equation*}
        \left|V_t - \hat{V}_{t, t_{i, j}}\right| \leq 2 \|f'\|_\infty \sqrt{t} W_t(\beta) (\varepsilon_i + \beta|t - t_{i, j}|\big) \quad \text{almost surely},
    \end{equation*}
    from which we deduce \eqref{eq:almost_sure_convergence_of_functional_2_brutal_bound_1}.
    
    For the convergence \eqref{eq:almost_sure_convergence_of_functional_2_brutal_bound_2}, we use the following bound
    \begin{equation*}
        \left|\sum_{u \in \cN(t)} f(X_u(s) - \beta s) - \sum_{u \in \cN(s)} f(X_u(s) - \beta s)\right| \leq \|f\|_\infty \sum_{u \in \cN(s)} \left(W_{t-s}^{(u)}(0) - 1\right),
    \end{equation*}
    where $W_{t-s}^{(u)}(0)$, $u \in \cN(s)$, are defined in \eqref{eq:definition_of_W_s^(u)}.
    Conditionally on $\cF_s$, these processes have the same distribution as $W_{t-s}(0)$.
    We can then apply Doob's maximal inequality in the same vein as done for \eqref{eq:almost_sure_convergence_of_functional_1_brutal_bound_2}.
    In the end, this yields \eqref{eq:almost_sure_convergence_of_functional_2_brutal_bound_2}.
\end{proof}

\section{Fluctuations of the additive martingales}\label{sct:fluctuations}

In this section, we determine the suitable rate function $\gamma(t)$ such that $\gamma(t)(W_\infty(\beta) - W_t(\beta))$ converges to a non-degenerate distribution as $t$ goes to infinity, where $\beta \in [0, \beta_c)$ is fixed.
Furthermore, we identify this limit distribution.
Essentially, there are two different regimes, depending on whether $\beta$ is smaller or larger than $\beta_c/2$.
The fluctuations are first Gaussian and then $\alpha$-stable with $1 < \alpha < 2$.

Following the notations of \cite{IksanovKoleskoMeiners2020}, given $(X_t)_{t \geq 0}$ and $X$ random variables, we write
\begin{equation}\label{eq:definition_weak_convergence_in_probability}
    \condLaw{X_t}{\cF_t} \xrightarrow[t \to \infty]{w} \condLaw{X}{\cF_\infty} \quad \text{in probability}
\end{equation}
if for every bounded continuous function $\phi : \R \to \R$, the conditional expectation $\condExpec{\phi(X_t)}{\cF_t}$ converges in probability to $\condExpec{\phi(X)}{\cF_\infty}$ as $t$ goes to infinity.
As explained in \cite[Section~2.2]{IksanovKoleskoMeiners2020}, it is not difficult to see that \eqref{eq:definition_weak_convergence_in_probability} implies that $X_t$ converges in distribution to $X$ as $t$ goes to infinity.

In the same way, we write
\begin{equation}\label{eq:definition_weak_convergence_almost_surely}
    \condLaw{X_t}{\cF_t} \xrightarrow[t \to \infty]{w} \condLaw{X}{\cF_\infty} \quad \text{almost surely}
\end{equation}
if for every bounded continuous function $\phi : \R \to \R$, the conditional expectation $\condExpec{\phi(X_t)}{\cF_t}$ converges almost surely to $\condExpec{\phi(X)}{\cF_\infty}$ as $t$ goes to infinity.
Clearly, the convergence \eqref{eq:definition_weak_convergence_almost_surely} also implies that $X_t$ converges in distribution to $X$ as $t$ goes to infinity.

\begin{theorem}\label{th:fluctuations}
    Let $N$ be a standard Gaussian variable independent of $\cF_\infty$ and recall that $K := \sum_{k \geq 0} \mu(k)k(k-1)$.
    \begin{enumerate}
        \item If $0 \leq \beta < \sqrt{2}/2$, then
        \begin{equation}\label{eq:fluctuations_1}
            \condLaw{\e^{(1-\beta^2)t/2} (W_\infty(\beta) - W_t(\beta))}{\cF_t} \xrightarrow[t \to \infty]{w} \condLaw{\sigma N}{\cF_\infty} \quad \text{almost surely},
        \end{equation}
        where $\sigma^2 = \left(\frac{K}{1 - \beta^2} - 1\right) W_\infty(2\beta)$.
        \item If $\beta = \sqrt{2}/2$, then
        \begin{equation}\label{eq:fluctuations_2}
            \condLaw{t^{1/4} \e^{t/4} (W_\infty(\beta) - W_t(\beta))}{\cF_t} \xrightarrow[t \to \infty]{w} \condLaw{\sigma N}{\cF_\infty} \quad \text{in probability},
        \end{equation}
        where $\sigma^2 = (2K - 1) \sqrt{2/\pi} Z_\infty$.
        \item Assume that the branching is binary $\mu = \delta_2$.
        If $\sqrt{2}/2 < \beta < \sqrt{2}$, then
        \begin{equation}\label{eq:fluctuations_3}
            \e^{c(\beta)t - \beta m(t)} \left(W_\infty(\beta) - W_t(\beta)\right) \xrightarrow[t \to \infty]{} (C Z_\infty)^{\beta/\beta_c} S \quad \text{in distribution},
        \end{equation}
        where $m(t)$ is defined in \eqref{eq:definition_of_m(t)} and $S$ is a non-degenerate $\beta_c/\beta$-stable random variable.
    \end{enumerate}
\end{theorem}

\begin{remark}
    \begin{enumerate}
        \item It should be possible to show that the convergence \eqref{eq:fluctuations_3} still holds if we replace the binary branching with the assumption \eqref{eq:assumption_offspring_distribution}.
        We make this assumption to use Corollary~\ref{cor:tightness_of_the_supercritical_additive_martingales}.
        We also need the offspring distribution to satisfy $\mu(0) = 0$ in order to use the results of Section~\ref{sct:convergence_of_the_extremal_process}.
        
        \item Any stable distribution is infinitely divisible, \ie for each $k \geq 1$, it can be expressed as the distribution of $X_1 + \cdots + X_k$, where $X_1, \ldots, X_k$ are independent and identically distributed.
        The Lévy-Khintchine representation \cite[Theorem~8.1]{Sato1999} then associates to any stable distribution $\lambda$ a unique \emph{generating triplet} $(A, \nu, \gamma)$ which characterizes $\lambda$, where $A \geq 0$, $\nu$ is a measure on $\R$ satisfying
        \begin{equation*}
            \nu(\{0\}) = 0 \quad \text{and} \quad \int_\R (x^2 \wedge 1) \Pd{\nu}{x} < \infty,
        \end{equation*}
        and $\gamma \in \R$.
        In Section~\ref{sct:extremal_fluctuations_stable}, we show that the random variable $S$ that appears in \eqref{eq:fluctuations_3} is infinitely divisible and use its generating triplet to obtain its stability.
        
        \item It is well-known that, to any infinitely divisible distribution $\lambda$, there corresponds a unique Lévy process $(X_t)_{t \geq 0}$ such that $X_1$ has distribution $\lambda$.
        Moreover, if $(X_t)_{t \geq 0}$ is a Lévy process and $(A, \nu, \gamma)$ is the generating triplet of $X_1$, then $(tA, t\nu, t\gamma)$ is the generating triplet of $X_t$ for any $t \geq 0$.
        In view of this remark, a careful reading of the proof of Theorem~\ref{th:extremal_fluctuation_stable} below reveals that
        \begin{equation*}
            (C Z_\infty)^{\beta/\beta_c} S \overset{d}{=} X_{Z_\infty},
        \end{equation*}
        where $(X_t)_{t \geq 0}$ is a $\beta_c/\beta$-stable Lévy process independent of $Z_\infty$.
        This makes the link with \cite[Corollary~1.5]{MaillardPain2021}, where Maillard-Pain obtained the fluctiations of the rescaled critical martingale around its limit in probability.
        Indeed, they showed that
        \begin{equation*}
            \sqrt{t}(\sqrt{t} W_t(\beta_c) - \sqrt{2/\pi} Z_\infty) \xrightarrow[t \to \infty]{} X_{Z_\infty} \quad \text{in distribution},
        \end{equation*}
        where $X := (X_t)_{t \geq 0}$ is a Cauchy process independent of $Z_\infty$.
        In particular, $X$ is a $1$-stable Lévy process, \ie $X_1$ has $1$-stable distribution.
    \end{enumerate}
\end{remark}

The convergences \eqref{eq:fluctuations_1} and \eqref{eq:fluctuations_2} were already obtained by Hartung-Klimovsky in \cite[Theorem~1.4, Theorem~1.7]{HartungKlimovsky2017}, with complex parameter $\beta$, by appealing to the Lindeberg-Feller condition.
We favor a more elementary approach and compute an asymptotic expansion of the characteristic function conditionally on $\cF_t$.
This is the purpose of Section~\ref{sct:gaussian_fluctuations} and Section~\ref{sct:gaussian_boundary_case}.
In \cite{IksanovKoleskoMeiners2020}, Iksanov-Kolesko-Meiners obtained a slight variation of \eqref{eq:fluctuations_3} for the branching random walk, for complex $\beta$ too.
Their proof is based on Madaule's work \cite{Madaule2017} about the convergence of the extremal process.
We reproduce their arguments in Section~\ref{sct:extremal_regime}.

\subsection{Gaussian fluctuations}\label{sct:gaussian_fluctuations}

Assume that $0 \leq \beta < \sqrt{2}/2$ and define $\gamma(t) = \e^{(1-\beta^2)t/2}$ for all $t \geq 0$.
Let us compute the characteristic function of the random variable $\gamma(t)(W_\infty(\beta) - W_t(\beta))$ conditionally on $\cF_t$.
To this end, we use the decomposition \eqref{eq:definition_of_W_infty^(u)} of $W_\infty(\beta)$ in terms of $W_\infty^{(u)}(\beta)$, $u \in \cN(t)$.
Recall that the latter are independent copies of $W_\infty(\beta)$ and that they are independent of $\cF_t$.
Let $\theta \geq 0$. We have
\begin{align}
    \condExpec{\exp\left(\iu \theta \gamma(t)(W_\infty(\beta) - W_t(\beta))\right)}{\cF_t} &= \prod_{u \in \cN(t)} \condExpec{\exp\left(\iu \theta \gamma(t) \e^{\beta X_u(t) - c(\beta)t} (W_\infty^{(u)}(\beta) - 1)\right)}{\cF_t} \notag \\
    &= \prod_{u \in \cN(t)} \phi\left(\theta \gamma(t) \e^{\beta X_u(t) - c(\beta)t}\right), \label{eq:fluctuations_of_additive_martingales_characteristic_function}
\end{align}
where $\phi$ is the characteristic function of $W_\infty(\beta) - 1$.
The latter is a centered random variable and admits a second moment that can be computed thanks to \eqref{eq:second_moment_of_additive_martingales}.
This leads to the following expansion
\begin{equation}\label{eq:fluctuations_of_additive_martingales_characteristic_function_expansion}
    \phi(\lambda) = 1 - \frac{\lambda^2}{2} \left(\frac{K}{1 - \beta^2} - 1\right) + R_1(\lambda),
\end{equation}
where $K = \sum_{k \geq 0} \mu(k)k(k-1)$ and $R_1(\lambda) = o(\lambda^2)$ as $\lambda$ goes to $0$.
Evaluating the expansion \eqref{eq:fluctuations_of_additive_martingales_characteristic_function_expansion} in
\begin{equation*}
    \lambda = \theta \gamma(t) \e^{\beta X_u(t) - c(\beta)t}
\end{equation*}
for each $u \in \cN(t)$, we can rewrite \eqref{eq:fluctuations_of_additive_martingales_characteristic_function} as
\begin{multline}\label{eq:fluctuations_of_additive_martingales_characteristic_function_2}
    \condExpec{\exp\left(\iu \theta \gamma(t)(W_\infty(\beta) - W_t(\beta))\right)}{\cF_t} \\
    = \prod_{u \in \cN(t)} \underbrace{\left(1 - \frac{\theta^2}{2} \gamma(t)^2 \e^{2 \beta X_u(t) - 2 c(\beta) t} \left(\frac{K}{1 - \beta^2} - 1\right) + R_1\left(\theta \gamma(t) \e^{\beta X_u(t) - c(\beta) t}\right)\right)}_{=: z_u}.
\end{multline}
By \eqref{eq:domination_extremal_particle}, the quantity $M(t) - \sqrt{2}t$ converges almost surely to $-\infty$ as $t$ goes to infinity.
It is then straightforward to check that
\begin{equation}\label{eq:fluctuations_of_additive_martingales_assumption}
    \lim_{t \to \infty} \gamma(t) \e^{\beta M(t) - c(\beta)t} = 0 \quad \text{almost surely}.
\end{equation}
In particular, almost surely, for $t$ large enough and all $u \in \cN(t)$, we have $z_u \in \C \setminus \R_-$, which ensures that the principal value $\Log z_u$ is well defined.
This allows us to rewrite \eqref{eq:fluctuations_of_additive_martingales_characteristic_function_2} as
\begin{multline}\label{eq:fluctuations_of_additive_martingales_characteristic_function_1}
    \condExpec{\exp\left(\iu \theta \gamma(t)(W_\infty(\beta) - W_t(\beta))\right)}{\cF_t} \\
    = \exp\left(-\frac{\theta^2}{2} \gamma(t)^2 \sum_{u \in \cN(t)} \e^{2 \beta X_u(t) - 2c(\beta)t} \left(\frac{K}{1 - \beta^2} - 1\right) + \sum_{u \in \cN(t)} R_2\left(\theta \gamma(t) \e^{\beta X_u(t) - c(\beta)t}\right)\right),
\end{multline}
where $R_2(\lambda) = o(\lambda^2)$ as $\lambda$ goes to $0$.

Let us now justify that, when $t$ goes to infinity, we have
\begin{equation}\label{eq:fluctuations_of_additive_martingales_landau_sum_inversion}
    \sum_{u \in \cN(t)} R_2\left(\theta \gamma(t) \e^{\beta X_u(t) - c(\beta)t}\right) = o\left(\sum_{u \in \cN(t)} \gamma(t)^2 \e^{2 \beta X_u(t) - 2c(\beta)t}\right).
\end{equation}
Let $\varepsilon > 0$.
We can fix $\lambda_0 > 0$ such that for all $\lambda \in [0, \lambda_0]$, we have $|R_2(\lambda)| \leq \varepsilon \lambda^2$.
By \eqref{eq:fluctuations_of_additive_martingales_assumption}, there exists almost surely $t_0 \geq 0$ such that for all $t \geq t_0$ and for all $u \in \cN(t)$,
\begin{equation*}
    0 \leq \theta \gamma(t) \e^{\beta X_u(t) - c(\beta)t} \leq \lambda_0,
\end{equation*}
which implies that
\begin{equation*}
    \sum_{u \in \cN(t)} \left|R_2\left(\theta \gamma(t) \e^{\beta X_u(t) - c(\beta)t}\right)\right| \leq \varepsilon \sum_{u \in \cN(t)} \gamma(t)^2 \e^{2 \beta X_u(t) - 2c(\beta)t}.
\end{equation*}
Hence \eqref{eq:fluctuations_of_additive_martingales_landau_sum_inversion}.
Coming back to \eqref{eq:fluctuations_of_additive_martingales_characteristic_function_1}, this leads to
\begin{equation}\label{eq:fluctuations_of_additive_martingales_characteristic_function_3}
    \condExpec{\exp\left(\iu \theta \gamma(t)(W_\infty(\beta) - W_t(\beta))\right)}{\cF_t} = \exp\left(-\frac{\theta^2}{2} W_t(2\beta) \left(\frac{K}{1 - \beta^2} - 1\right) + o\left(W_t(2\beta)\right)\right).
\end{equation}
Since the additive martingales converge almost surely, \eqref{eq:fluctuations_of_additive_martingales_characteristic_function_3} yields the following almost sure limit
\begin{equation*}
    \lim_{t \to \infty} \condExpec{\exp\left(\iu \theta \gamma(t)(W_\infty(\beta) - W_t(\beta))\right)}{\cF_t} = \exp\left(-\frac{\theta^2}{2} W_\infty(2\beta) \left(\frac{K}{1 - \beta^2} - 1\right)\right).
\end{equation*}
By appealing to Lévy's continuity theorem, we obtain \eqref{eq:fluctuations_1}.

\subsection{Gaussian boundary case}\label{sct:gaussian_boundary_case}

Here, we investigate the case $\beta = \sqrt{2}/2$.
Define $\gamma(t) = t^{1/4} \e^{t/4}$ for all $t \geq 0$.
The calculation \eqref{eq:fluctuations_of_additive_martingales_characteristic_function_3} performed in the last section is still valid as soon as \eqref{eq:fluctuations_of_additive_martingales_assumption} holds.
We use the following lemma to show that \eqref{eq:fluctuations_of_additive_martingales_assumption} indeed holds with high probability.
It is a straightforward consequence of Proposition~\ref{prop:speed_extremal_particle_2}.

\begin{lemma}\label{lem:speed_extremal_particle_sharper}
    For any $\delta > 0$, the event $A_{\delta, t} := \{M(t) \leq m(t) + \delta \log t\}$ occurs with high probability, \ie its probability converges to $1$ as $t$ goes to infinity.
\end{lemma}

Let $\delta > 0$.
We rewrite
    \begin{equation}\label{eq:fluctuations_gaussian_boundary_case_assumption_A2}
        \gamma(t) \e^{\beta M(t) - c(\beta)t} = \e^{\beta(M(t) - \sqrt{2}t + 2^{-3/2} \log t)}.
    \end{equation}
    On the event $A_{\delta, t}$, the exponent in the right-hand side of \eqref{eq:fluctuations_gaussian_boundary_case_assumption_A2} is smaller than $\beta (\delta - 1/\sqrt{2}) \log t$.
    Hence, taking $0 < \delta < 1/\sqrt{2}$, it tends to $-\infty$ as $t$ goes to infinity, which implies that
    \begin{equation*}
        \lim_{t \to \infty} \gamma(t) \e^{\beta M(t) - c(\beta)t} \mathds{1}_{A_{\delta, t}} = 0.
    \end{equation*}
    In other words, on $A_{\delta, t}$, the convergence \eqref{eq:fluctuations_of_additive_martingales_assumption} holds and \eqref{eq:fluctuations_of_additive_martingales_characteristic_function_3} is still valid:
    \begin{equation*}
        \condExpec{\exp\left(\iu \theta \gamma(t)(W_\infty(\beta) - W_t(\beta))\right)}{\cF_t} \mathds{1}_{A_{\delta, t}} = \exp\left(-\frac{\theta^2}{2} (2K - 1) \sqrt{t} W_t(\beta_c) + o\left(\sqrt{t} W_t(\beta_c)\right)\right) \mathds{1}_{A_{\delta, t}}.
    \end{equation*}
    Then, by Proposition~\ref{prop:convergence_of_the_critical_additive_martingale} and Lemma~\ref{lem:speed_extremal_particle_sharper},
    \begin{equation*}
        \lim_{t \to \infty} \condExpec{\exp\left(\iu \theta \gamma(t)(W_\infty(\beta) - W_t(\beta))\right)}{\cF_t} = \exp\left(-\frac{\theta^2}{2} (2K - 1) \sqrt{2/\pi} Z_\infty\right) \quad \text{in probability},
    \end{equation*}
    that is \eqref{eq:fluctuations_2}.

\subsection{Extremal regime, characterization of the fluctuations}\label{sct:extremal_regime}

Assume that $\mu = \delta_2$ and fix $\sqrt{2}/2 < \beta < \sqrt{2}$.
In this section, we reproduce the arguments of \cite{IksanovKoleskoMeiners2020} where Iksanov-Kolesko-Meiners established the fluctuations in the case of the branching random walk.
In the end, this will lead to the following lemma.

\begin{lemma}\label{lem:extremal_fluctuations_characterization}
    We have
    \begin{equation*}
        \e^{c(\beta)t - \beta m(t)} \left(W_\infty(\beta) - W_t(\beta)\right) \xrightarrow[t \to \infty]{} X_{\mathrm{ext}} \quad \text{in distribution},
    \end{equation*}
    where
    \begin{equation}\label{eq:definition_of_X_ext}
        X_{\mathrm{ext}} := \sum_{k = 1}^\infty \e^{\beta \xi_k} (W_k - 1),
    \end{equation}
    $\xi_k$, $k \geq 1$, are the atoms of the limit extremal process $\cE_\infty$ defined in \eqref{eq:definition_of_the_limit_extremal_process}, and $W_k$, $k \geq 1$, are independent copies of $W_\infty(\beta)$, independent of $\cE_\infty$.
\end{lemma}

Using the decomposition \eqref{eq:definition_of_W_s^(u)}, we rewrite
\begin{equation*}
    \gamma(t)(W_\infty(\beta) - W_t(\beta)) = \sum_{u \in \cN(t)} \e^{\beta(X_u(t) - m(t))} \left(W_\infty^{(u)}(\beta) - 1\right).
\end{equation*}
We define, for $\ell \in \R$,
\begin{equation}\label{eq:extremal_fluctuations_definition_of_f_K}
    \chi_\ell^+(x) = \begin{cases}
        0 &\text{if } x \leq \ell, \\
        x - \ell &\text{if } \ell \leq x \leq \ell+1, \\
        1 &\text{if } x \geq \ell+1,
    \end{cases}
\end{equation}
and $\chi_\ell^- = 1 - \chi_\ell^+$. We decompose
\begin{align}
    \gamma(t)(W_\infty(\beta) - W_t(\beta)) &= \sum_{u \in \cN(t)} \e^{\beta(X_u(t) - m(t))} \left(W_\infty^{(u)}(\beta) - 1\right) \chi_\ell^+(X_u(t) - m(t)) \notag \\
    &\quad + \sum_{u \in \cN(t)} \e^{\beta(X_u(t) - m(t))} \left(W_\infty^{(u)}(\beta) - 1\right) \chi_\ell^-(X_u(t) - m(t)) \notag \\
    &=: Y_{t, \ell} + R_{t, \ell}. \label{eq:extremal_fluctuations_decomposition}
\end{align}
The following lemma states that $R_{t, \ell}$ is negligible as $\ell$ goes to $-\infty$.
In other words, only the particles whose positions are at distance $O(1)$ from the extremal one contribute to $\gamma(t)(W_\infty(\beta) - W_t(\beta))$.
This explains why our description of these fluctuations follows from the results of Section~\ref{sct:convergence_of_the_extremal_process}.

\begin{lemma}\label{lem:extremal_fluctuations_negligible_particles}
    For all $\delta > 0$, we have $\lim_{\ell \to -\infty} \limsup_{t \to \infty} \Prob{|R_{t, \ell}| > \delta} = 0$.
\end{lemma}

\begin{proof}
    Let $\varepsilon > 0$ and fix two real numbers $p$ and $q$ such that $1 < \beta_c/\beta < p < q < 2 \wedge \beta_c^2/\beta^2$.
    By Proposition~\ref{prop:additive_martingales_bounded_in_Lp}, the additive martingale $(W_t(\beta))_{t \geq 0}$ and its limit are bounded in $L^q$.
    Besides, by Corollary~\ref{cor:tightness_of_the_supercritical_additive_martingales} the process $(\sum_{u \in \cN(t)} \e^{p\beta(X_u(t) - m(t))})_{t \geq 0}$ is tight, \ie there exists $M > 0$ such that
    \begin{equation*}
        \sup_{t \geq 0} \P\Biggl(\underbrace{\sum_{u \in \cN(t)} \e^{p \beta(X_u(t) - m(t))} > M}_{=: \cQ_t}\Biggr) \leq \varepsilon.
    \end{equation*}
    We have
    \begin{align}
        \Prob{|R_{t, \ell}| > \delta} &\leq \Prob{|R_{t, \ell}| \mathds{1}_{\cQ_t^c} > \delta} + \varepsilon \notag \\
        &\leq \frac{1}{\delta^q} \Expec{|R_{t, \ell}|^q \mathds{1}_{\cQ_t^c}} + \varepsilon, \label{eq:Markov_on_R_t,K}
    \end{align}
    by Markov's inequality.
    Using Lemma~\ref{lem:biggins_lemma}, we obtain
    \begin{align*}
        \Expec{|R_{t, \ell}|^q \mathds{1}_{\cQ_t^c}} &\leq 2^q \E|W_\infty(\beta) - 1|^q \Expec{\sum_{u \in \cN(t)} \e^{q \beta(X_u(t) - m(t))} \chi_\ell^-(X_u(t) - m(t))^q \mathds{1}_{\cQ_t^c}} \\
        &\leq 2^q \E|W_\infty(\beta) - 1|^q \e^{(q - p) \beta (\ell+1)} \Expec{\sum_{u \in \cN(t)} \e^{p \beta(X_u(t) - m(t))} \mathds{1}_{\cQ_t^c}}.
    \end{align*}
    By definition of the event $\cQ_t$ and since $q > p$, the above quantity converges to $0$ as $\ell$ goes to $-\infty$, uniformly in $t \geq 0$.
    Then, coming back to \eqref{eq:Markov_on_R_t,K}, it yields
    \begin{equation*}
        \limsup_{\ell \to -\infty} \limsup_{t \to \infty} \Prob{|R_{t, \ell}| > \delta} \leq \varepsilon,
    \end{equation*}
    which concludes, since $\varepsilon$ is arbitrarily small.
\end{proof}

\begin{lemma}\label{lem:extremal_fluctuations_one_side_compact_support}
    Let $f : \R \to \R$ be a continuous function.
    If there exists $a \in \R$ such that $f(x) = 0$ for all $x \leq a$, then the process $\cE_t(f)$ converges in distribution to $\cE_\infty(f)$ as $t$ goes to infinity.
\end{lemma}

\begin{proof}
    We use the function $\chi_L^-$ defined in \eqref{eq:extremal_fluctuations_definition_of_f_K}.
    By \cite[Theorem~3.2]{Billingsley1999}, it is sufficient to show the three following steps:
    \begin{enumerate}
        \item\label{step:extremal_fluctuations_one_side_compact_support_1} for all $L \in \R$, $\cE_t(\chi_L^- f)$ converges in distribution to $\cE_\infty(\chi_L^- f)$ as $t$ goes to infinity,
        \item\label{step:extremal_fluctuations_one_side_compact_support_2} $\cE_\infty(\chi_L^- f)$ converges in distribution to $\cE_\infty(f)$ as $L$ goes to infinity,
        \item\label{step:extremal_fluctuations_one_side_compact_support_3} for all $\delta > 0$, $\lim_{L \to \infty} \limsup_{t \to \infty} \Prob{\left|\cE_t(f) - \cE_t(\chi_L^- f)\right| > \delta} = 0$.
    \end{enumerate}
    
    Since the function $\chi_L^-f$ is continuous with compact support, Theorem~\ref{th:convergence_of_the_extremal_process} yields \ref{step:extremal_fluctuations_one_side_compact_support_1}.
    By Remark~\ref{rem:limit_of_the_extremal_process}, the number of atoms of $\cE_\infty$ in $[a, \infty)$ is almost surely finite.
    Since $f(x) = 0$ for all $x \leq a$, this implies that $\cE_\infty(\chi_L^- f)$ converges almost surely to $\cE_\infty(f)$ as $L$ goes to infinity.
    Hence \ref{step:extremal_fluctuations_one_side_compact_support_2}.
    Finally, for all $\delta > 0$ and all $L > 0$,
    \begin{align}
        \limsup_{t \to \infty} \Prob{\left|\cE_t(f) - \cE_t(\chi_L^- f)\right| > \delta} &\leq \limsup_{t \to \infty} \Prob{\cE_t([L, \infty)) \geq 1} \notag \\
        &\leq \limsup_{t \to \infty} \Prob{M(t) - m(t) \geq L}. \label{eq:extremal_fluctuations_K_large_enough_2}
    \end{align}
    By Theorem~\ref{th:convergence_of_M(t)-m(t)}, the above quantity converges to $0$ as $L$ goes to infinity.
    Hence \ref{step:extremal_fluctuations_one_side_compact_support_3}.
\end{proof}

\begin{remark}\label{rem:bonnefont_on_extremal_process}
    As explained by Bonnefont in \cite[Remark~2.6]{Bonnefont2022}, one can show a stronger result than Lemma~\ref{lem:extremal_fluctuations_one_side_compact_support}.
    Namely, the convergence in distribution of $\cE_t(f)$ to $\cE_\infty(f)$ still holds if we only assume that $f$ is continuous and that there exists $\beta' > \beta_c$ such that $f(x) = O(e^{\beta'x})$ as $x$ goes to $-\infty$.
\end{remark}

\begin{remark}\label{rem:extremal_fluctuations_one_side_compact_support_jointly}
    If $f : \R \to \R$ is such as in Lemma~\ref{lem:extremal_fluctuations_one_side_compact_support}, then for all $\beta \geq 0$, the pair $(\cE_t(f), W_t(\beta))$ converges jointly in distribution to $(\cE_\infty(f), W_\infty(\beta))$.
    Indeed, the arguments in the proof of Lemma~\ref{lem:extremal_fluctuations_one_side_compact_support} hold for this pair since
    \begin{equation*}
        (\cE_t(\chi_L^- f), W_t(\beta)) \xrightarrow[t \to \infty]{} (\cE_\infty(\chi_L^- f), W_\infty(\beta)) \quad \text{in distribution},
    \end{equation*}
    by Corollary~\ref{cor:joint_convergence_extremal_process_additive_martingale}.
\end{remark}

We now define two point processes on $\R^2$ by
\begin{equation*}
    \cE_t^* = \sum_{u \in \cN(t)} \delta_{\left(X_u(t) - m(t), W_\infty^{(u)}(\beta)\right)} \quad \text{and} \quad \cE_\infty^* = \sum_{k \geq 1} \delta_{\left(\xi_k, W_k\right)},
\end{equation*}
where $W_k$, $k \geq 1$, are independent copies of $W_\infty(\beta)$, independent of $\cE_\infty$.

\begin{lemma}\label{lem:extremal_fluctuations_one_side_compact_support_2}
    Let $f : \R^2 \to \R$ be a bounded continuous function.
    If there exists $a \in \R$ such that $f(x, y) = 0$ for all $x \leq a$ and all $y \in \R$, then the process $\cE_t^*(f)$ converges in distribution to $\cE_\infty^*(f)$ as $t$ goes to infinity.
\end{lemma}

\begin{proof}
    First assume $f$ to be non-negative.
    In order to obtain the convergence in distribution of $\cE_t^*(f)$ towards $\cE_\infty^*(f)$, we show the pointwise convergence of the associated Laplace transforms on $(0, \infty)$.
    Since $f : \R^2 \to \R_+$ is an arbitrary bounded continuous function, it is sufficient to show that
    \begin{equation}\label{eq:extremal_fluctuations_laplace_transform_1}
        \Expec{\e^{-\cE_t^*(f)}} \xrightarrow[t \to \infty]{} \Expec{\e^{-\cE_\infty^*(f)}}.
    \end{equation}
    Let us rewrite the left-hand side of \eqref{eq:extremal_fluctuations_laplace_transform_1} so that we can apply Lemma~\ref{lem:extremal_fluctuations_one_side_compact_support}:
    \begin{align}
        \Expec{\e^{-\cE_t^*(f)}} &= \E\condExpec{\exp\left(-\sum_{u \in \cN(t)} f\left(X_u(t) - m(t), W_\infty^{(u)}(\beta)\right)\right)}{\cF_t} \notag \\
        &= \Expec{\prod_{u \in \cN(t)} \varphi\left(X_u(t) - m(t)\right)}, \label{eq:extremal_fluctuations_laplace_transform_2}
    \end{align}
    where, for all $x \in \R$,
    \begin{equation*}
        \varphi(x) := \Expec{\e^{-f\left(x, W_\infty(\beta)\right)}}.
    \end{equation*}
    The function $\varphi$ clearly takes its values in $(0, 1]$.
    We can rewrite \eqref{eq:extremal_fluctuations_laplace_transform_2} as
    \begin{align}
        \Expec{\e^{-\cE_t^*(f)}} = \Expec{\exp \sum_{u \in \cN(t)} \log \varphi\left(X_u(t) - m(t)\right)} = \Expec{\e^{-\cE_t(- \log \varphi)}}. \label{eq:extremal_fluctuations_laplace_transform_3}
    \end{align}
    By dominated convergence, the function $\varphi$ is continuous on $\R$, so is $-\log \varphi$.
    In addition, by assumption on $f$, there exists $a \in \R$ such that $-\log \varphi(x) = 0$ for all $x \leq a$.
    Then, by Lemma~\ref{lem:extremal_fluctuations_one_side_compact_support}, letting $t$ go to infinity, the Laplace transform \eqref{eq:extremal_fluctuations_laplace_transform_3} converges to
    \begin{equation*}
        \Expec{\e^{-\cE_\infty(- \log \varphi)}} = \Expec{\e^{-\cE_\infty^*(f)}}.
    \end{equation*}
    Hence the result in the case where $f$ is non-negative.
    To treat the general case, we use the positive and negative parts of $f$, denoted by $f^+$ and $f^-$.
    By the above, for any non-negative real coefficients $\lambda_1$ and $\lambda_2$,
    \begin{equation*}
        \cE_t^*(\lambda_1 f^+ + \lambda_2 f^-) \xrightarrow[t \to \infty]{} \cE_\infty^*(\lambda_1 f^+ + \lambda_2 f^-) \quad \text{in distribution}.
    \end{equation*}
    Then, the (multivariate) Laplace transform of the couple $\left(\cE_t^*(f^+), \cE_t^*(f^-)\right)$ converges pointwise to that of $\left(\cE_\infty^*(f^+), \cE_\infty^*(f^-)\right)$ on $\R_+^2$.
    This entails the corresponding convergence in distribution (see \eg \cite[Theorem~6]{Kozakiewicz1947}).
    By taking the difference between the two components, we deduce the convergence stated in Lemma~\ref{lem:extremal_fluctuations_one_side_compact_support_2}.
\end{proof}

\begin{remark}\label{rem:extremal_fluctuations_one_side_compact_support_jointly_2}
    If $f : \R^2 \to \R$ is such as in Lemma~\ref{lem:extremal_fluctuations_one_side_compact_support_2}, then for all $\beta \geq 0$, the pair $(\cE_t^*(f), W_t(\beta))$ converges jointly in distribution to $(\cE_\infty^*(f), W_\infty(\beta))$.
    Indeed, fixing $\lambda > 0$, the same calculations as in the proof of Lemma~\ref{lem:extremal_fluctuations_one_side_compact_support_2} lead to
    \begin{equation*}
        \Expec{\e^{-\cE_t^*(f) - \lambda W_t(\beta)}} = \Expec{\e^{-\cE_t(-\log \varphi) - \lambda W_t(\beta)}}.
    \end{equation*}
    Hence, by Remark~\ref{rem:extremal_fluctuations_one_side_compact_support_jointly},
    \begin{equation*}
        \lim_{t \to \infty} \Expec{\e^{-\cE_t^*(f) - \lambda W_t(\beta)}} = \Expec{\e^{-\cE_\infty(-\log \varphi) - \lambda W_\infty(\beta)}} = \Expec{\e^{-\cE_\infty^*(f) - \lambda W_\infty(\beta)}}.
    \end{equation*}
    This joint convergence will be useful in Section~\ref{sct:renormalized_subcritical_overlap_2}.
\end{remark}

Now, let us show that the series defining $X_{\mathrm{ext}}$ in \eqref{eq:definition_of_X_ext} converges almost surely to a non-degenerate limit.
Note that, given $\cE_\infty$, the sequence $\left(\sum_{k = 1}^n \e^{\beta \xi_k} (W_k - 1)\right)_{n \geq 1}$ is a martingale since its elements are sums of independent centered random variables.
Moreover, it is bounded in $L^p$ for any $p \in (\beta_c/\beta, 2 \wedge \beta_c^2/\beta^2)$.
Indeed, using Lemma~\ref{lem:biggins_lemma}, we obtain
\begin{equation*}
    \sup_{n \geq 1} \condExpec{\left|\sum_{k=1}^n \e^{\beta \xi_k} (W_k - 1)\right|^p}{\cE_\infty} \leq 2^p \E\left|W_\infty(\beta) - 1\right|^p \sum_{k \geq 1} \e^{p\beta \xi_k},
\end{equation*}
which is finite, according to Proposition~\ref{prop:additive_martingales_bounded_in_Lp} and Corollary~\ref{cor:asymptotic_mean_number_decoration}.
Therefore, conditionally on $\cE_\infty$, it converges almost surely to a random variable as $n$ goes to infinity and then also unconditionally.

It remains to show the convergence \eqref{eq:fluctuations_3}.
Let $\ell \in \R$. Recall that we have the following decomposition
\begin{equation*}
    \gamma(t)(W_\infty(\beta) - W_t(\beta)) = Y_{t, \ell} + R_{t, \ell},
\end{equation*}
with $Y_{t, \ell}$ and $R_{t, \ell}$ defined in \eqref{eq:extremal_fluctuations_decomposition}.
By Lemma~\ref{lem:extremal_fluctuations_negligible_particles} and \cite[Theorem~3.2]{Billingsley1999}, it is sufficient to show that for all $\ell \in \R$,
\begin{equation}\label{eq:fluctuations_3_1}
    Y_{t,\ell} \xrightarrow[t \to \infty]{} Y_\ell := \sum_{k \geq 1} \e^{\beta \xi_k} \chi_\ell^+(\xi_k) (W_k - 1) \quad \text{in distribution},
\end{equation}
and that
\begin{equation}\label{eq:fluctuations_3_2}
    Y_\ell \xrightarrow[\ell \to -\infty]{} X_{\mathrm{ext}} \quad \text{in distribution}.
\end{equation}
Note that the sum $Y_\ell$ is almost surely well defined, whatever $\ell \in \R$, since it has almost surely a finite number of terms, by Remark~\ref{rem:limit_of_the_extremal_process}.
Applying Lemma~\ref{lem:extremal_fluctuations_one_side_compact_support_2} to the function $(x, y) \mapsto \e^{\beta x} \chi_\ell^+(x) (y - 1)$, we obtain \eqref{eq:fluctuations_3_1}.
We now show that the convergence \eqref{eq:fluctuations_3_2} holds in $L^p$, given $\cE_\infty$.
To this end, we apply Lemma~\ref{lem:biggins_lemma} combined with Fatou's lemma:
\begin{align}
    \condExpec{\left|X_{\mathrm{ext}} - Y_\ell\right|^p}{\cE_\infty} &= \condExpec{\lim_{n \to \infty} \left|\sum_{k = 1}^n \e^{\beta \xi_k} \chi_\ell^-(\xi_k) (W_k - 1)\right|^p}{\cE_\infty} \notag \\
    &\leq \liminf_{n \to \infty} \left(2^p \sum_{k = 1}^n \e^{p \beta \xi_k} \mathds{1}_{\xi_k \leq \ell+1} \E \left|W_\infty(\beta) - 1\right|^p \right) \notag \\
    &= 2^p \E\left|W_\infty(\beta) - 1\right|^p \sum_{k \geq 1} \e^{p \beta \xi_k} \mathds{1}_{\xi_k \leq \ell+1}. \label{eq:biggins_lemma_plus_fatou_lemma}
\end{align}
The point process $\cE_\infty$ has almost surely no accumulation point.
Therefore, almost surely, the above expression is the remainder of a series which, by Corollary~\ref{cor:asymptotic_mean_number_decoration}, converges.
This shows that the above quantity converges almost surely to $0$ as $\ell$ goes to $-\infty$.
Hence \eqref{eq:fluctuations_3_2}.
This concludes the proof of Lemma~\ref{lem:extremal_fluctuations_characterization}.

\subsection{Extremal regime, stable fluctuations}\label{sct:extremal_fluctuations_stable}

In this section, we still assume $\beta_c/2 < \beta < \beta_c$ and $\mu = \delta_2$.
We go deeper into the analysis than Iksanov-Kolesko-Meiners \cite{IksanovKoleskoMeiners2020} and show that, given $Z_\infty$, the fluctuations of the additive martingale around its almost sure limit are asymptotically $\alpha$-stable with $\alpha = \beta_c/\beta \in (1, 2)$.
Using the notations introduced in Theorem~\ref{th:convergence_of_the_extremal_process}, we can rewrite the limit $X_{\mathrm{ext}}$ defined in \eqref{eq:definition_of_X_ext} as
\begin{equation*}
    X_{\mathrm{ext}} = (C Z_\infty)^{\beta/\beta_c} \sum_{i \geq 1} \e^{\beta p_i} \underbrace{\sum_{j \geq 1} \e^{\beta \Delta_{ij}} (W_{ij} - 1)}_{=: Y_i},
\end{equation*}
where $W_{ij}$, $i, j \geq 1$, are independent copies of $W_\infty(\beta)$ which are independent of $\cE_\infty$.

\begin{lemma}\label{lem:stable_fluctuations}
    There exists $p > \alpha$ such that $\E |Y_1|^p < \infty$.
\end{lemma}

\begin{proof}
    By Proposition~\ref{prop:additive_martingales_bounded_in_Lp}, for $p \in [1, 2]$ such that $\alpha  < p < \alpha^2$, we have $\E\left|W_\infty(\beta) - 1\right|^p < \infty$.
    Applying Lemma~\ref{lem:biggins_lemma} and Fatou's lemma in the same way we did in \eqref{eq:biggins_lemma_plus_fatou_lemma}, we obtain
    \begin{equation*}
        \condExpec{|Y_1|^p}{\cD_1} \leq 2^p \E\left|W_\infty(\beta) - 1\right|^p \sum_{j \geq 1} \e^{p \beta \Delta_{1j}}.
    \end{equation*}
    By Corollary~\ref{cor:asymptotic_mean_number_decoration}, the expectation of the above quantity is finite since $p\beta > \beta_c$.
\end{proof}

\begin{theorem}\label{th:extremal_fluctuation_stable}
    The random variable $S := \sum_{i \geq 1} \e^{\beta p_i} Y_i$ is $\beta_c/\beta$-stable.
\end{theorem}

\begin{proof}
    Let us fix $\theta \in \R$ and compute the characteristic function of $S$.
    For $i \geq 1$, we set $q_i = \e^{\beta p_i}$.
    Note that the point process $\cQ := \sum_i \delta_{q_i}$ is a Poisson point process with intensity $\alpha x^{-1-\alpha} \mathds{1}_{x > 0} \d{x}$.
    By conditioning on $\cQ$ and by dominated convergence, we obtain
    \begin{equation}\label{eq:stable_fluctuations_characteristic_function}
        \Expec{\e^{\iu \theta S}} = \Expec{\prod_{i \geq 1} \phi(\theta q_i)},
    \end{equation}
    where $\phi$ denotes the characteristic function of $Y := Y_1$.
    
    Let $0 < \varepsilon < 1$ and define
    \begin{equation*}
        \phi_\varepsilon(x) = \begin{cases}
            |\phi(x)| &\text{if } |\phi(x)| > \varepsilon \\
            \varepsilon &\text{if } |\phi(x)| \leq \varepsilon.
        \end{cases}
    \end{equation*}
    Let $\psi_\varepsilon : \R \to \C$ be a measurable function such that ${\exp}\circ{\psi_\varepsilon} = \phi_\varepsilon$ so that
    \begin{equation}\label{eq:stable_fluctuations_characteristic_function_2}
        \Expec{\prod_{i \geq 1} \phi_\varepsilon(\theta q_i)} = \Expec{\exp \sum_{i \geq 1}  \psi_\varepsilon(\theta q_i)}.
    \end{equation}
    In order to use the exponential formula for Poisson point processes (see \eg \cite[Section~0.5]{Bertoin1996}), let us show that $x \mapsto (1 - \phi_\varepsilon(\theta x)) x^{-1-\alpha}$ is integrable on $(0, \infty)$.
    Since $\alpha > 0$ and $\phi_\varepsilon$ is bounded, it is sufficient to check the integrability on $(0, 1]$.
    By Lemma~\ref{lem:stable_fluctuations}, there exists $p > \alpha$ such that $\E|Y|^p < \infty$.
    Since $Y$ is centered, we deduce that $1 - \phi_\varepsilon(\theta x) = O(|x|^p)$ as $x$ goes to $0$.
    Hence the integrability of $x \mapsto (1 - \phi_\varepsilon(\theta x)) x^{-1-\alpha}$.
    By the exponential formula, \eqref{eq:stable_fluctuations_characteristic_function_2} becomes
    \begin{equation}\label{eq:stable_fluctuations_characteristic_function_3}
        \Expec{\prod_{i \geq 1} \phi_\varepsilon(\theta q_i)} = \exp\left(-\alpha \int_0^\infty (1 - \phi_\varepsilon(\theta x)) x^{-1-\alpha} \d{x}\right).
    \end{equation}
    By dominated convergence, letting $\varepsilon$ go to $0$, \eqref{eq:stable_fluctuations_characteristic_function} and \eqref{eq:stable_fluctuations_characteristic_function_3} yield
    \begin{equation}\label{eq:stable_fluctuations_characteristic_function_4}
        \Expec{\e^{\iu \theta S}} = \exp\left( - \alpha \int_0^\infty (1 - \phi(\theta x)) x^{-1-\alpha} \d{x}\right).
    \end{equation}

    Note that, since $Y$ is centered, we can rewrite $1 - \phi(\theta x) = \Expec{1 + \iu \theta xY - \e^{\iu \theta x Y}}$.
    This allows us to apply Fubini's theorem to \eqref{eq:stable_fluctuations_characteristic_function_4}.
    Indeed, we have $\E|Y|^p < \infty$ with $p > \alpha > 1$ and then
    \begin{equation*}
        \E\left|1 + \iu \theta xY - \e^{\iu \theta x Y}\right| = \begin{cases}
            O(|x|) &\text{as $x$ goes to $\infty$}, \\
            O(|x|^p) &\text{as $x$ goes to $0$}.
        \end{cases}
    \end{equation*}
    Thus,
    \begin{equation}\label{eq:stable_fluctuations_characteristic_function_5}
        \Expec{\e^{\iu \theta S}} = \exp\left(-\alpha \E\int_0^\infty \left(1 + \iu \theta xY - \e^{\iu \theta x Y}\right) x^{-1-\alpha} \d{x}\right).
    \end{equation}
    The change of variable $y = xY$ turns \eqref{eq:stable_fluctuations_characteristic_function_5} into
    \begin{equation}\label{eq:stable_fluctuations_characteristic_function_6}
        \Expec{\e^{\iu \theta S}} = \exp\left(-\int_\R \left(1 + \iu \theta y - \e^{\iu \theta y}\right) \Pd{\nu}{y}\right),
    \end{equation}
    where
    \begin{equation*}
        \Pd{\nu}{y} := |y|^{-1-\alpha} \left(c_+ \mathds{1}_{y > 0} + c_- \mathds{1}_{y < 0}\right) \d{y} \quad \text{and} \quad \begin{cases}
            c_+ = \alpha \Expec{|Y|^\alpha \mathds{1}_{Y > 0}}, \\
            c_- = \alpha \Expec{|Y|^\alpha \mathds{1}_{Y < 0}}.
        \end{cases}
    \end{equation*}
    In order to follow the notations of \cite{Sato1999}, we rewrite \eqref{eq:stable_fluctuations_characteristic_function_6} as
    \begin{equation*}
        \Expec{\e^{\iu \theta S}} = \exp\left(\iu \gamma \theta - \int_\R \left(1 + \iu \theta y \mathds{1}_{|y| \leq 1} - \e^{\iu \theta y}\right) \Pd{\nu}{y}\right),
    \end{equation*}
    where $\gamma \in \R$.
    By \cite[Theorem~8.1]{Sato1999}, the random variable $S$ is infinitely divisible with generating triplet $(0, \nu, \gamma)$ and is then, by \cite[Theorem~14.3]{Sato1999}, $\alpha$-stable.
    This concludes the proof of \eqref{eq:fluctuations_3} and of Theorem~\ref{th:fluctuations}.
\end{proof}

\section{Convergence of the rescaled overlap distribution}\label{sct:rescaled_overlap_distribution}

In this section, we study the \emph{overlap} between two particles $u$ and $v$ chosen under the Gibbs measure at inverse temperature $\beta \in [0, \beta_c)$ of the branching Brownian motion at time $t > 0$, \ie the quantity $d_{u \wedge v}/t$, where $d_{u \wedge v}$ denotes the deathtime of the last common ancestor of $u$ and $v$.
More precisely, we describe the asymptotic behavior of the random measure
\begin{equation*}
    \nu_{\beta, t} := \frac{1}{W_t(\beta)^2} \sum_{u, v \in \cN(t)} \e^{\beta(X_u(t) + X_v(t)) - 2 c(\beta) t} \delta_{d_{u \wedge v}/t},
\end{equation*}
called \emph{overlap distribution} at inverse temperature $\beta$.
The motivation comes from the mean-field spin glass models, where analogous notions of overlap are used.
The Sherrington-Kirkpatrick model, introduced in \cite{SherringtonKirkpatrick1975}, is certainly the most famous and has been studied through many approaches.
Among others, Parisi ideas made a great contribution to the understanding of this model, in particular by studying the probability distribution of the overlap in \cite{Parisi1983}.
Another strategy has been to introduce simplifications of the model.
The most drastic one is the random energy model (REM) of Derrida \cite{Derrida1980}, where the energies of the spin configurations are assumed to be independent.
Surprisingly, it contains some important effects which are observed in real spin glasses and it can be generalized in order to include correlations\footnote{See the generalized random energy model (GREM) introduced by Derrida-Gardner \cite{DerridaGardner1986}.}.
Derrida-Spohn noticed in \cite{DerridaSpohn1988} that the branching random walk and the branching Brownian motion share many properties with the random energy model.
Subsequently, both processes can serve as intermediate toy models for spin glass systems.
Based on these similarities, Derrida-Spohn predicted that the rescaled overlap distribution of the branching random walk converges in distribution to a random measure whose support is $\{0\}$ at high temperature ($\beta \leq \beta_c$) and $\{0, 1\}$ at low temperature ($\beta > \beta_c$).
This conjecture was proved by Mallein in \cite[Theorem~4.3]{Mallein2018}, with a description of the limit.
Its counterpart for the branching Brownian motion is a consequence of Bonnefont \cite[Theorem~1.1]{Bonnefont2022}.

In the following theorem, we describe the asymptotic fluctuations at high temperature of the overlap distribution.
Recall that the point process $\cE_\infty$ is defined in \eqref{eq:definition_of_the_limit_extremal_process}.

\begin{theorem}\label{th:renormalized_subcritical_overlap}
    Let $0 \leq \beta < \beta_c$ and $0 < a < 1$.
    \begin{enumerate}
        \item If $0 \leq \beta < \beta_c/2$, then
        \begin{equation}\label{eq:renormalized_subcritical_overlap_1}
            \e^{(1 - \beta^2)at} \nu_{\beta, t}([a, 1]) \xrightarrow[t \to \infty]{} \frac{1}{W_\infty(\beta)^2} W_\infty(2\beta) \Expec{W_\infty(\beta)^2} \quad \text{in probability},
        \end{equation}
        on the event of survival.
        \item If $\beta = \beta_c/2$, then
        \begin{equation}\label{eq:renormalized_subcritical_overlap_1.5}
            \sqrt{at} \e^{at/2} \nu_{\beta, t}([a, 1]) \xrightarrow[t \to \infty]{} \frac{1}{W_\infty(\beta)^2} \sqrt{2/\pi} Z_\infty \Expec{W_\infty(\beta)^2} \quad \text{in probability},
        \end{equation}
        on the event of survival.
        \item Assume that the branching is binary $\mu = \delta_2$. If $\beta_c/2 < \beta < \beta_c$, then
        \begin{equation}\label{eq:renormalized_subcritical_overlap_2}
            (at)^{3 \beta/\beta_c} \e^{(\beta_c - \beta)^2 at} \nu_{\beta, t}([a, 1]) \xrightarrow[t \to \infty]{} \frac{(C_1 Z_\infty)^{2\beta/\beta_c}}{W_\infty(\beta)^2} S \quad \text{in distribution},
        \end{equation}
        where $S$ is a non-degenerate $\beta_c/2\beta$-stable random variable independent of $\cF_\infty$ and $C_1$ is a positive constant.
    \end{enumerate}
\end{theorem}

\begin{remark}
    \begin{enumerate}
        \item It should be possible to show that the convergence \eqref{eq:renormalized_subcritical_overlap_2} still holds if we replace the binary branching with the assumption \eqref{eq:assumption_offspring_distribution}.
        We make this assumption to use Corollary~\ref{cor:tightness_of_the_supercritical_additive_martingales}.
        We also need the offspring distribution to satisfy $\mu(0) = 0$ in order to use the results of Section~\ref{sct:convergence_of_the_extremal_process}.
        
        \item We will see in Section~\ref{sct:renormalized_subcritical_overlap_2} that the constant appearing in \eqref{eq:renormalized_subcritical_overlap_2} can be expressed as
        \begin{equation*}
            C_1 = \Expec{C \left(\sum_{j \geq 1} \e^{2\beta \Delta_j} W_j^2\right)^{\beta_c/2\beta}},
        \end{equation*}
        where $\cD := \sum_{j \geq 1} \delta_{\Delta_j}$ and $C$ are defined in Theorem~\ref{th:convergence_of_the_extremal_process} and $W_j$, $j \geq 1$, are independent copies of $W_\infty(\beta)$, independent of $\cD$.
        
        \item A direct consequence of Theorem~\ref{th:renormalized_subcritical_overlap} is that for any $\beta \in [0, \beta_c)$ and $a \in (0, 1)$, the process $\nu_{\beta, t}([a, 1])$ converges in probability to $0$ as $t$ goes to infinity.
        Thanks to the monotone class theorem, we deduce that, as mentioned above,
        \begin{equation*}
            \nu_{\beta, t} \xrightarrow[t \to \infty]{} \delta_0 \quad \text{in probability},
        \end{equation*}
        for the topology of the weak convergence of measures.
    \end{enumerate}
\end{remark}

The proof of Theorem~\ref{th:renormalized_subcritical_overlap} is the subject of Section~\ref{sct:renormalized_subcritical_overlap_1} and Section~\ref{sct:renormalized_subcritical_overlap_2}.
In both, we use a rewriting of the distribution function
\begin{align}
    \nu_{\beta, t}([a, 1]) &= \frac{1}{W_t(\beta)^2} \sum_{u \in \cN(at)} \left(\sum_{\substack{v \in \cN(t)\\v \geq u}} \e^{\beta X_v(t) - c(\beta)t}\right)^2 \notag \\
    &= \frac{1}{W_t(\beta)^2} \sum_{u \in \cN(at)} \e^{2 \beta X_u(at) - 2 c(\beta) at} \underbrace{\left(\sum_{\substack{v \in \cN(t)\\v \geq u}} \e^{\beta(X_v(t) - X_u(at)) - c(\beta)(t - at)}\right)^2}_{= W_{t - at}^{(u)}(\beta)^2}. \label{eq:overlap_rewriting}
\end{align}
Recall that, conditionally on $\cF_{at}$, the random variables $W_{t - at}^{(u)}(\beta)$, $u \in \cN(at)$, are independent copies of $W_{t - at}(\beta)$.
Moreover, the limit $W_\infty^{(u)}(\beta) = \lim_{s \to \infty} W_{s - at}^{(u)}(\beta)$ is almost surely well defined and has the same distribution as $W_\infty(\beta)$.

\subsection{High temperature regime}\label{sct:renormalized_subcritical_overlap_1}

Assume that $0 \leq \beta < \beta_c/2$.
We first prove the convergence \eqref{eq:renormalized_subcritical_overlap_1}.
Since $\e^{(1 - \beta^2)at} = \e^{(2c(\beta) - c(2\beta))at}$, \eqref{eq:overlap_rewriting} yields
\begin{align*}
    \e^{(1 - \beta^2)at} \nu_{\beta, t}([a, 1]) &= \frac{1}{W_t(\beta)^2} \left(\sum_{u \in \cN(at)} \e^{2 \beta X_u(at) - c(2 \beta)at} W_\infty^{(u)}(\beta)^2\right. \\
    &\quad \left.+ \sum_{u \in \cN(at)} \e^{2 \beta X_u(at) - c(2 \beta)at} \left(W_{t - at}^{(u)}(\beta)^2 - W_\infty^{(u)}(\beta)^2\right)\right) \\
    &=: \frac{1}{W_t(\beta)^2} (Y_t + R_t).
\end{align*}
To prove \eqref{eq:renormalized_subcritical_overlap_1}, it is then sufficient to obtain the convergence in probability of $Y_t + R_t$ towards $W_\infty(2\beta) \Expec{W_\infty(\beta)^2}$.

Let us justify that we can ignore the term $R_t$.
We have,
\begin{equation*}
    \condExpec{|R_t|}{\cF_{at}} \leq W_{at}(2\beta) \E\left|W_{t - at}(\beta)^2 - W_\infty(\beta)^2\right|.
\end{equation*}
Since $\beta < \beta_c/2 < 1$, the process $W_{t - at}(\beta)$ converges to $W_\infty(\beta)$ in $L^2$ as $t$ goes to infinity, by Proposition~\ref{prop:phase_transition_easy_case}.
We deduce that the above quantity converges almost surely to $0$ and that $R_t$ converges in probability to $0$.

It remains to show that $Y_t$ converges in probability to $W_\infty(2\beta) \Expec{W_\infty(\beta)^2}$ as $t$ goes to infinity.
To this end, we fix $\theta \in \R$ and compute
\begin{equation*}
    \condExpec{\e^{\iu \theta Y_t}}{\cF_{at}} = \prod_{u \in \cN(at)} \phi(\theta \e^{2 \beta X_u(at) - c(2 \beta) at}),
\end{equation*}
where
\begin{equation*}
    \phi(\lambda) := \Expec{\e^{\iu \lambda W_\infty(\beta)^2}} = 1 + \iu \lambda \Expec{W_\infty(\beta)^2} + o(\lambda)
\end{equation*}
as $\lambda$ goes to $0$.
Since $\e^{2 \beta M(at) - c(2 \beta) at}$ converges almost surely to $0$ as $t$ goes to infinity, we can use the Taylor expansion of the principal value $\Log z$ in the same way as in Section~\ref{sct:gaussian_fluctuations}.
This yields
\begin{equation*}
    \condExpec{\e^{\iu \theta Y_t}}{\cF_{at}} = \exp\left(\iu \theta W_{at}(2\beta) \Expec{W_\infty(\beta)^2} + o\left(W_{at}(2\beta)\right)\right).
\end{equation*}
We deduce that $Y_t - W_{at}(2\beta) \Expec{W_\infty(\beta)^2}$ converges in probability to $0$ as $t$ goes to infinity, which concludes the proof of \eqref{eq:renormalized_subcritical_overlap_1}.

To prove \eqref{eq:renormalized_subcritical_overlap_1.5}, it is sufficient to repeat the above arguments and to apply Proposition~\ref{prop:convergence_of_the_critical_additive_martingale}.

\subsection{Extremal regime}\label{sct:renormalized_subcritical_overlap_2}

Assume that $\beta_c/2 < \beta < \beta_c$ and that $\mu = \delta_2$.
In the same vein as in the proof of \eqref{eq:renormalized_subcritical_overlap_1} and \eqref{eq:renormalized_subcritical_overlap_1.5}, we start by rewriting
\begin{equation*}
    (at)^{3 \beta/\beta_c} \e^{(\beta_c - \beta)^2 at} \nu_{\beta,t}([a, 1]) = \frac{1}{W_t(\beta)^2} \underbrace{\sum_{u \in \cN(at)} \e^{2 \beta (X_u(at) - m(at))} W_{t - at}^{(u)}(\beta)^2}_{=: Y_t}.
\end{equation*}

Let $\ell \in \R$. We decompose
\begin{align*}
    Y_t &=\sum_{u \in \cN(at)} \e^{2 \beta (X_u(at) - m(at))} W_\infty^{(u)}(\beta)^2 \chi_\ell^+(X_u(at) - m(at)) \\
    &\quad + \sum_{u \in \cN(at)} \e^{2 \beta (X_u(at) - m(at))} \left(W_{t - at}^{(u)}(\beta)^2 - W_\infty^{(u)}(\beta)^2\right) \chi_\ell^+(X_u(at) - m(at)) \\
    &\quad + \sum_{u \in \cN(at)} \e^{2 \beta (X_u(at) - m(at))} W_{t - at}^{(u)}(\beta)^2 \chi_\ell^-(X_u(at) - m(at)) \\
    &=: Y_{t, \ell} + R_{t, \ell}^1 + R_{t, \ell}^2,
\end{align*}
where $\chi_\ell^+$ and $\chi_\ell^-$ are defined in \eqref{eq:extremal_fluctuations_definition_of_f_K}.
Below, Lemma~\ref{lem:renormalized_subcritical_overlap_replacement} and Lemma~\ref{lem:renormalized_subcritical_overlap_contribution} state that $R_{t, \ell}^1$ and $R_{t, \ell}^2$ are negligible as $t$ goes to $\infty$ and $\ell$ goes to $-\infty$.
In other words, we can replace the random variables $W_{t - at}^{(u)}(\beta)$, $u \in \cN(at)$, with $W_\infty^{(u)}(\beta)$ and consider only the particles whose positions are at distance $O(1)$ from the extremal one.
Lemma~\ref{lem:renormalized_subcritical_overlap_limit} gives a first expression for the limit distribution \eqref{eq:renormalized_subcritical_overlap_2}.

\begin{lemma}\label{lem:renormalized_subcritical_overlap_replacement}
    For all $\ell \in \R$, $R_{t, \ell}^1$ converges in probability to $0$ as $t$ goes to infinity.
\end{lemma}

\begin{proof}
    First assume that $\beta_c/2 < \beta < 1$.
    We have
    \begin{equation*}
        \condExpec{|R_{t, \ell}^1|}{\cF_{at}} \leq \E\left|W_{t - at}(\beta)^2 - W_\infty(\beta)^2\right| \sum_{u \in \cN(at)} \e^{2\beta(X_u(at) - m(at))} \chi_\ell^+(X_u(at) - m(at)).
    \end{equation*}
    By Proposition~\ref{prop:phase_transition_easy_case}, Lemma~\ref{lem:extremal_fluctuations_one_side_compact_support} and Slutsky's theorem, the above quantity converges in distribution to $0$ as $t$ goes to infinity.
    Then, using Jensen's inequality, we obtain
    \begin{equation*}
        \Expec{|R_{t, \ell}^1| \wedge 1} \leq \Expec{\condExpec{|R_{t, \ell}^1|}{\cF_{at}} \wedge 1} \xrightarrow[t \to \infty]{} 0.
    \end{equation*}
    Equivalently, we have $\lim_{t \to \infty} \Prob{|R_{t, \ell}^1| > \delta} = 0$, whatever $\delta > 0$.
    
    Now assume that $1 \leq \beta < \beta_c$.
    We can fix $p \in (\beta_c/2\beta, 1/\beta^2) \subset [0, 1]$ and replicate the above argument with $|R_{t, \ell}^1|$ raised to the power of $p$, using the subadditivity of $x \mapsto x^p$ on $\R_+$.
\end{proof}

\begin{lemma}\label{lem:renormalized_subcritical_overlap_contribution}
    For all $\delta > 0$, $\lim_{\ell \to -\infty} \limsup_{t \to \infty} \Prob{R_{t,\ell}^2 > \delta} = 0$.
\end{lemma}

\begin{proof}
    Note that if a process $(X_{t, \ell})$ takes its values in $\R_+$, then the following two statements are equivalent
    \begin{enumerate}
        \item\label{step:renormalized_subcritical_overlap_2_2_1} for any $\delta > 0$, $\lim_{\ell \to -\infty} \limsup_{t \to \infty} \Prob{X_{t, \ell} \geq \delta} = 0$,
        \item\label{step:renormalized_subcritical_overlap_2_2_2} $\lim_{\ell \to -\infty} \limsup_{t \to \infty} \Expec{X_{t, \ell} \wedge 1} = 0$.
    \end{enumerate}
    
    Assume that $\beta_c/2 < \beta < 1$.
    The case $1 \leq \beta < \beta_c$ can be treated in the same way by fixing $p \in (\beta_c/2\beta, 1/\beta^2) \subset [0, 1]$ and using the subadditivity of $x \mapsto x^p$ on $\R_+$.
    We have
    \begin{align*}
        \condExpec{R_{t, \ell}^2}{\cF_{at}} &\leq \sup_{s \geq 0} \Expec{W_s(\beta)^2} \underbrace{\sum_{u \in \cN(at)} \e^{2 \beta(X_u(at) - m(at))} \chi_\ell^-(X_u(at) - m(at))}_{= \cE_{at}(f_\ell)},
    \end{align*}
    where $f_\ell : x \mapsto \e^{2 \beta x} \chi_\ell^-(x)$.
    By Jensen's inequality, we have $\Expec{R_{t, \ell}^2 \wedge 1} \leq \Expec{\condExpec{R_{t, \ell}^2}{\cF_{at}} \wedge 1}$.
    Since \ref{step:renormalized_subcritical_overlap_2_2_1} and \ref{step:renormalized_subcritical_overlap_2_2_2} are equivalent, it is then sufficient to show that for all $\delta > 0$,
    \begin{equation}\label{eq:bonnefont_argument}
        \lim_{\ell \to -\infty} \limsup_{t \to \infty} \Prob{\cE_t(f_\ell) \geq \delta} = 0.
    \end{equation}
    This is proved by Bonnefont in \cite[Proposition~A.2]{Bonnefont2022}.
    We present the arguments for the sake of completeness.
    
    Let $\varepsilon > 0$ be such that $\beta_c + \varepsilon < 2\beta$.
    The key element is the introduction of the event
    \begin{equation*}
        E_{t, L} := \left\{\exists k \geq 0 : \cE_t([-L-k, \infty)] \geq \e^{(\beta_c + \varepsilon)(L+k)}\right\},
    \end{equation*}
    for $t \geq 0$ and $L > 0$.
    We have
    \begin{align}
        \Prob{\cE_t(f_{-L}) \geq \delta} &\leq \Prob{\cE_t(f_{-L}) \geq \delta, E_{t, L}^c} + \Prob{E_{t, L}, M(t) - m(t) \leq L} + \Prob{M(t) - m(t) > L} \notag \\
        &=: T_1 + T_2 + T_3. \label{eq:bonnefont_argument_T1_T2_T3}
    \end{align}
    
    Let us first control the term $T_1$. We have
    \begin{align*}
        \cE_t(f_{-L}) &\leq \sum_{u \in \cN(t)} \e^{2\beta(X_u(t) - m(t))} \mathds{1}_{X_u(t) - m(t) \leq -L+1} \\
        &\leq \sum_{u \in \cN(t)} \sum_{k \geq 0} \e^{-2\beta(L+k-1)} \mathds{1}_{-L-k \leq X_u(t) - m(t) \leq -L-k+1} \\
        &\leq \sum_{k \geq 0} \cE_t([-L-k, \infty)) \e^{-2\beta(L + k - 1)}.
    \end{align*}
    Thus, on the event $E_{t, L}^c$,
    \begin{equation*}
        \cE_t(f_{-L}) \leq \sum_{k \geq 0} \e^{(\beta_c + \varepsilon)(L + k)} \e^{-2\beta(L + k - 1)}.
    \end{equation*}
    The above quantity does not depend on $t$ and converges to $0$ as $L$ goes to infinity.
    Hence, letting $t$ and then $L$ go to infinity, the term $T_1$ in \eqref{eq:bonnefont_argument_T1_T2_T3} vanishes.
    
    Applying successively the union bound, Markov's inequality and \eqref{eq:first_moment_estimate_extremal_process}, we obtain
    \begin{align*}
        T_2 &\leq \sum_{k \geq 0} \Prob{\cE_t([-L-k, \infty)) \geq \e^{(\beta_c + \varepsilon)(L + k)}, M(t) - m(t) \leq L + k} \\
        &\leq \sum_{k \geq 0} \e^{-(\beta_c + \varepsilon)(L + k)} \Expec{\cE_t([-L-k, \infty)) \mathds{1}_{M(t) - m(t) \leq L + k}} \\
        &\leq \sum_{k \geq 0} C(L + k + 1)^2 \e^{-\varepsilon(L + k)},
    \end{align*}
    where $C$ is a positive constant.
    
    It remains to control the term $T_3$ in \eqref{eq:bonnefont_argument_T1_T2_T3}.
    By Corollary~\ref{cor:tightness_of_M(t)-m(t)}, the process $M(t) - m(t)$ is tight, \ie $T_3$ converges to $0$ as $L$ goes to infinity, uniformly in $t \geq 0$.
    This concludes the proof of \eqref{eq:bonnefont_argument}.
\end{proof}

\begin{lemma}\label{lem:renormalized_subcritical_overlap_limit}
    Recall that the point process $\cE_\infty = \sum_k \delta_{\xi_k}$ is defined in Theorem~\ref{th:convergence_of_the_extremal_process}.
    We have
    \begin{equation*}
        (at)^{3 \beta/\beta_c} \e^{(\beta_c - \beta)^2 at} \nu_{\beta, t}([a, 1]) \xrightarrow[t \to \infty]{} X_{\mathrm{over}} := \frac{1}{W_\infty(\beta)^2} \sum_{k \geq 1} \e^{2 \beta \xi_k} W_k^2 \quad \text{in distribution},
    \end{equation*}
    where $W_k$, $k \geq 1$, are independent copies of $W_\infty(\beta)$, independent of $W_\infty(\beta)$ and $\cE_\infty$.
\end{lemma}

\begin{proof}
    By Lemma~\ref{lem:renormalized_subcritical_overlap_replacement}, Lemma~\ref{lem:renormalized_subcritical_overlap_contribution} and \cite[Theorem~3.2]{Billingsley1999}, it is sufficient to show that for all $\ell \in \R$,
    \begin{equation}\label{eq:renormalized_subcritical_overlap_2_1}
        \frac{Y_{t, \ell}}{W_t(\beta)^2} \xrightarrow[t \to \infty]{} \frac{1}{W_\infty(\beta)^2} \sum_{k \geq 1} \e^{2 \beta \xi_k} \chi_\ell^+(\xi_k) W_k^2 \quad \text{in distribution},
    \end{equation}
    and that
    \begin{equation}\label{eq:renormalized_subcritical_overlap_2_2}
        \frac{1}{W_\infty(\beta)^2} \sum_{k \geq 1} \e^{2 \beta \xi_k} \chi_\ell^+(\xi_k) W_k^2 \xrightarrow[\ell \to -\infty]{} X_{\mathrm{over}} \quad \text{in distribution}.
    \end{equation}
    
    By Remark~\ref{rem:extremal_fluctuations_one_side_compact_support_jointly_2}, the pair $(\cE_{at}^*, W_{at}(\beta))$ converges jointly in distribution to $(\cE_\infty^*, W_\infty(\beta))$.
    The same is true for $(\cE_{at}^*, W_t(\beta))$, by Slutsky's theorem.
    Hence \eqref{eq:renormalized_subcritical_overlap_2_1}.
    
    Let us show that the series $\sum_{k \geq 0} \e^{2 \beta \xi_k} W_k^2$ is almost surely convergent.
    It will follow that the convergence \eqref{eq:renormalized_subcritical_overlap_2_2} holds almost surely.
    If $\beta_c/2 < \beta < 1$, then
    \begin{equation*}
        \condExpec{\sum_{k \geq 1} \e^{2 \beta \xi_k} W_k^2}{\cE_\infty} = \sum_{k \geq 1} \e^{2 \beta \xi_k} \Expec{W_\infty(\beta)^2},
    \end{equation*}
    which is almost surely finite, by Corollary~\ref{cor:asymptotic_mean_number_decoration} and Proposition~\ref{prop:phase_transition_easy_case}.
    If $1 \leq \beta < \beta_c$, then we can find a real number $p$ such that $1/2 < \beta_c/2\beta < p < 1/\beta^2 \leq 1$.
    Since the function $x \mapsto x^p$ is subadditive on $\R_+$, we have
    \begin{equation*}
        \condExpec{\left|\sum_{k \geq 1} \e^{2 \beta \xi_k} W_k^2\right|^p}{\cE_\infty} \leq \sum_{k \geq 1} \e^{2p\beta \xi_k} \Expec{W_\infty(\beta)^{2p}},
    \end{equation*}
    which is almost surely finite, by Corollary~\ref{cor:asymptotic_mean_number_decoration} and Proposition~\ref{prop:additive_martingales_bounded_in_Lp}.
\end{proof}

Let us finish the proof of Theorem~\ref{th:renormalized_subcritical_overlap}.
In view of the definition of $\cE_\infty$ given in Theorem~\ref{th:convergence_of_the_extremal_process}, we can rewrite
\begin{equation}\label{eq:renormalized_subcritical_overlap_2_equal_in_distribution_1}
    X_{\mathrm{over}} = \frac{(C Z_\infty)^{2\beta/\beta_c}}{W_\infty(\beta)^2} \sum_{i\geq 1} \e^{2\beta(p_i + X_i)},
\end{equation}
where $X_i := \frac{1}{2\beta} \log \sum_j \e^{2 \beta \Delta_{ij}} W_{ij}^2$, $i \geq 1$, and where $W_{ij}$, $i, j \geq 1$, are independent copies of $W_\infty(\beta)$, independent of $\cF_\infty$, $\cP$ and $\cD_i$, $i \geq 1$.
Let us then apply \cite[Proposition~8.7.a]{BolthausenSznitman2002} as done for Corollary~\ref{cor:asymptotic_mean_number_decoration}.
The random variables $X_i$, $i \geq 1$, are independent and identically distributed, independent of $\cP$.
Furthermore, taking $p \in (\beta_c/\beta, 2 \wedge \beta_c^2/\beta^2)$, which is possible since $\beta_c/2 < \beta < \beta_c$, we have
\begin{equation*}
    \Expec{\e^{\beta_c X_1}} = \Expec{\left(\sum_{j \geq 1} \e^{2 \beta \Delta_{1j}} W_{1j}^2\right)^{\beta_c/2\beta}} \leq \Expec{\left(\sum_{j \geq 1} \e^{2 \beta \Delta_{1j}} W_{1j}^2\right)^{p/2}} \leq \Expec{\sum_{j \geq 1} \e^{p \beta \Delta_{1j}}} \Expec{W_\infty(\beta)^{p}},
\end{equation*}
where the last inequality is obtained by using the subadditivity of $x \mapsto x^{p/2}$ on $\R$ and the independence of $\cD_1$ and $W_{1j}$, $j \geq 1$.
By Corollary~\ref{cor:asymptotic_mean_number_decoration} and Proposition~\ref{prop:additive_martingales_bounded_in_Lp}, the above quantity is finite.
We can then apply \cite[Proposition~8.7.a]{BolthausenSznitman2002}, which gives
\begin{equation*}
    \sum_{i \geq 1} \delta_{p_i + X_i} \overset{d}{=} \sum_{i \geq 1} \delta_{p_i + \frac{1}{\beta_c} \log \Expec{\e^{\beta_c X_1}}}.
\end{equation*}
Thus, \eqref{eq:renormalized_subcritical_overlap_2_equal_in_distribution_1} becomes
\begin{equation}\label{eq:renormalized_subcritical_overlap_2_equal_in_distribution_2}
    X_{\mathrm{over}} \overset{d}{=} \frac{\left(C Z_\infty \Expec{\e^{\beta_c X_1}}\right)^{2\beta/\beta_c}}{W_\infty(\beta)^2} \sum_{i\geq 1} \e^{2\beta p_i}.
\end{equation}
Finally, by Corollary~\ref{cor:asymptotic_mean_number_decoration}, the sum $S := \sum_i \e^{2\beta p_i}$ is a non-degenerate $\beta_c/2\beta$-stable random variable.
The limit in distribution \eqref{eq:renormalized_subcritical_overlap_2} then follows from Lemma~\ref{lem:renormalized_subcritical_overlap_limit} and \eqref{eq:renormalized_subcritical_overlap_2_equal_in_distribution_2}.
This concludes the proof of Theorem~\ref{th:renormalized_subcritical_overlap}.

\bibliographystyle{abbrv}
\bibliography{biblio}

\end{document}